\documentclass[a4paper,twoside,10pt]{article}
\usepackage[a4paper,left=2.5cm,right=2.5cm, top=2.5cm, bottom=2.5cm]{geometry}
\usepackage[latin1]{inputenc}
\usepackage{stmaryrd}
\SetSymbolFont{stmry}{bold}{U}{stmry}{m}{n}
\usepackage{cuted}
\usepackage{amsthm}
\usepackage{scalerel}
\usepackage{mathrsfs}
\usepackage{mathtools}
\usepackage{accents}
\usepackage{multirow}
\usepackage{graphicx}
\usepackage[hidelinks]{hyperref}
\usepackage[nice]{nicefrac}
\usepackage{algorithm2e}
\usepackage{amscd,amsmath,amssymb,mathrsfs,bbm,listings}
\usepackage{cancel}
\usepackage{dirtytalk}
\usepackage{epstopdf}
\allowdisplaybreaks
\graphicspath{{figs/}}
\usepackage{amsmath}
\usepackage{footmisc}
\usepackage{amsthm}
\usepackage{amssymb}
\usepackage{stmaryrd}
\usepackage{float}
\usepackage{bigints}
\usepackage{cite}
\usepackage{color}
\usepackage[abs]{overpic}
\usepackage[font=footnotesize,labelfont=bf]{caption}
\usepackage{cases}
\usepackage{tikz}
\usepackage{rotating}
\usepackage{blkarray}
\usepackage{scalerel}
\usetikzlibrary{matrix,calc,arrows,cd}
\usepackage{pdfcomment}
\usepackage{soul,xcolor}
\usepackage{enumitem}
\usepackage{verbatim}
\usepackage{graphicx}
\usepackage{subcaption}
\usepackage{mwe}
\usepackage{tikz}

\newtheorem{theorem}{Theorem}[section]
\newtheorem{lemma}[theorem]{Lemma}
\newtheorem{proposition}[theorem]{Proposition}
\newtheorem{corollary}[theorem]{Corollary}
\newtheorem{remark}[theorem]{Remark}

\newtheorem{assumption}[theorem]{Assumption}
\newcommand{\eremk}{\hbox{}\hfill\rule{0.8ex}{0.8ex}}

\numberwithin{equation}{section}
\definecolor{myblue}{rgb}{0,0,0.6}
\hypersetup{
    colorlinks,
    linktoc=section,
    citecolor=red,
    linkcolor=myblue,
    pdfborder={0 0 0}
}

\newcommand{\supg}{{\sf{supg}}}
\newcommand{\Knsupg}{\hyperref[EQN::LOCAL-SUPG-NORM]{\scaleto{K_n,\sf{\supg}}{5.5pt}}}
\newcommand{\calD}{\mathcal{D}}
\newcommand{\Cinv}{C_{\mathrm{inv}}}
\newcommand{\SO}{\Sigma_0}
\newcommand{\Sn}{\Sigma_n}
\newcommand{\Snmo}{\Sigma_{n-1}}
\newcommand{\ST}{\Sigma_T}
\newcommand{\GD}{\Gamma_{\mathrm{D}}}
\newcommand{\bbeta}{\boldsymbol{\beta}}
\newcommand{\betaKn}{\beta_{\Kn}}
\newcommand{\betaQT}{\bar{\beta}_{\QT}}
\newcommand{\N}{\mathbb{N}}
\newcommand{\R}{\mathbb{R}}
\newcommand{\QT}{Q_T}
\newcommand{\Kn}{K_n}
\newcommand{\sfJ}{{\hyperref[EQN::JUMP-FUNCTIONAL]{\scaleto{\mathsf{J}}{4.5pt}}}}
\newcommand{\dpt}{\partial_t}
\newcommand{\TNorm}[1]{|\!|\!| #1 |\!|\!|_{\scriptstyle{\hyperref[EQN::ENERGY-NORM]{\scaleto{\mathcal{E}}{4.5pt}}}}}
\newcommand{\Norm}[2]{\|#1\|_{#2}}
\newcommand{\jump}[1]{[\![#1]\!]}
\newcommand{\SemiNorm}[2]{|#1|_{#2}}
\newcommand{\Pp}[1]{\mathbb{P}_{#1}}
\newcommand{\Id}{{\sf{Id}}}
\newcommand{\uI}{u_{\mathcal{I}}}
\newcommand{\uIt}{u_{\mathcal{I}}^{\tau}}
\newcommand{\vI}{v_{\mathcal{I}}}

\newcommand{\eIt}{e_{\mathcal{I}}^{\tau}}
\newcommand{\eh}{e_{h\tau}}
\newcommand{\ePiO}{e_{\pi}^0}
\newcommand{\ePiN}{e_{\pi}^{\nabla}}
\newcommand{\PiO}[2]{\Pi_{#1}^{0, #2}}
\newcommand{\gPiO}[1]{\Pi_{#1}^{0}}
\newcommand{\PPiO}[2]{\boldsymbol{\Pi}_{#1}^{0, #2}}
\newcommand{\PiN}[2]{\Pi_{#1}^{\nabla, #2}}
\newcommand{\gPiN}[1]{\Pi_{#1}^{\nabla}}
\newcommand{\ma}{m_{\boldsymbol{\alpha}}}
\newcommand{\mb}{m_{\boldsymbol{\beta}}}
\newcommand{\mc}{m_{\boldsymbol{\gamma}}}
\newcommand{\ba}{\boldsymbol{\alpha}}

\newcommand{\Bk}{\mathbb{B}_k}
\newcommand{\Wk}{\mathbb{W}_k}
\newcommand{\Vkt}{\widetilde{V}_k}
\newcommand{\Vk}{V_k}
\newcommand{\Vh}{\mathcal{V}_h}
\newcommand{\Vht}{\mathcal{V}_h^{\tau}}

\renewcommand{\u}{u}
\newcommand{\uh}{\u_{h, \tau}}

\renewcommand{\v}{v}
\newcommand{\vh}{\v_{h, \tau}}

\newcommand{\wh}{w_h}
\newcommand{\wt}{w_{\tau}}

\renewcommand{\div}{\operatorname{div}}
\newcommand{\hK}{h_K}

\newcommand{\Th}{\Omega_h}
\newcommand{\BK}{B_K}
\newcommand{\Bf}{B_F}
\newcommand{\Omegah}{\Omega_h}
\newcommand{\Tt}{\mathcal{T}_{\tau}}

\newcommand{\bx}{\boldsymbol{x}}
\renewcommand{\d}{\text{d}}
\newcommand{\dx}{\, \mathrm{d}\bx}
\newcommand{\dS}{\, \mathrm{d}S}

\newcommand{\dt}{\, \mathrm{d}t}
\newcommand{\ds}{\, \mathrm{d}s}

\newcommand{\tnmo}{t_{n-1}}
\newcommand{\tn}{t_{n}}

\newcommand{\mht}{m_{h,\tau}^{\partial_t}}
\newcommand{\mh}{m_h}
\newcommand{\mhK}{m_{h}^K}
\newcommand{\smK}{s_m^K}
\newcommand{\aht}{a_{h, \tau}}
\newcommand{\ahK}{a_h^K}
\newcommand{\ahKn}{a_h^{\Kn}}
\newcommand{\ah}{a_h}
\newcommand{\saK}{s_a^{K}}
\newcommand{\bht}{b_{h, \tau}^{\sf skew}}
\newcommand{\bh}{b_h}

\newcommand{\sht}{s_{h, \tau}^{\sf supg}}
\newcommand{\shtKn}{s_{h, \tau}^{K_n, \sf supg}}
\newcommand{\LhtK}{\hyperref[eq:LhtK]{\mathcal{L}_{h, \tau}^{K}}}
\newcommand{\tLhtK}{\hyperref[eq:tLhtK]{\widetilde{\mathcal{L}}_{h, \tau}^K}}
\newcommand{\tLK}{\hyperref[DEF::LK]{\widetilde{\mathcal{L}}^{K}}}
\newcommand{\LK}{\hyperref[DEF::LK]{\mathcal{L}^{K}}}
\newcommand{\lambdaKn}{\lambda_{\Kn}}
\newcommand{\Bht}{\hyperref[DEF::Bht]{\mathcal{B}_{h, \tau}^{\sf supg}}}
\newcommand{\ellht}{\hyperref[DEF::lht]{\ell_{h, \tau}^{\sf supg}}}
\newcommand{\chil}{\chi_{\ell}}
\newcommand{\chiFone}{\chi_{\mathcal{F}}}
\newcommand{\chiUo}{\chi_{\mathcal{U}_0}}
\newcommand{\chia}{\chi_a}
\newcommand{\chimb}{\chi_{m, b}}
\newcommand{\chiV}{\chi_{V}}
\newcommand{\chiVone}{\chi_{V,1}}
\newcommand{\chiVtwo}{\chi_{V,2}}
\newcommand{\chiVthree}{\chi_{V,3}}
\newcommand{\chiVfour}{\chi_{V,4}}
\newcommand{\chiVfive}{\chi_{V, 5}}

\newcommand{\chiT}{\chi_{T}}
\newcommand{\chiJ}{\chi_{J}}
\newcommand{\chisupg}{\chi_{\supg}}
\newcommand{\chisupgone}{\chi_{\supg, 1}}
\newcommand{\chisupgtwo}{\chi_{\supg, 2}}
\newcommand{\chisupgthree}{\chi_{\supg, 3}}
\newcommand{\chisupgfour}{\chi_{\supg, 4}}
\newcommand{\chisupgi}{\chi_{\supg, i}}

\newcommand{\stab}{\texttt{SUPG}}
\newcommand{\nostab}{\texttt{NONE}}

\title{SUPG-stabilized time-DG finite and virtual elements for the time-dependent advection--diffusion equation}
\author{L. Beir\~ao da Veiga\thanks{Department of Mathematics and Applications, University of Milano-Bicocca, 20125 Milan, Italy (\href{mailto:lourenco.beirao@unimib.it}{lourenco.beirao@unimib.it}, \href{mailto:franco.dassi@unimib.it}{franco.dassi@unimib.it}, \href{mailto:sergio.gomezmacias@unimib.it}{sergio.gomezmacias@unimib.it})}\, \thanks{IMATI-CNR, 27100 Pavia, Italy} \and F. Dassi\footnotemark[1] \and S. G\'omez\footnotemark[1]}
\date{}

\begin{document}
\maketitle

\begin{abstract}
\noindent We carry out a stability and convergence analysis for the fully discrete scheme obtained by combining a finite or virtual element spatial discretization with the upwind-discontinuous Galerkin time-stepping applied to the time-dependent advection--diffusion equation. A space--time streamline-upwind Petrov--Galerkin term is used to stabilize the method.
More precisely, we show that the method is inf-sup stable with constant independent of the diffusion coefficient, which ensures the robustness of the method in the convection- and diffusion-dominated regimes. Moreover, we prove optimal convergence rates in both regimes for the error in the energy norm.
An important feature of the presented analysis is the control in the full $L^2(0,T;L^2(\Omega))$ norm without the need of introducing an artificial reaction term in the model.
We finally present some numerical experiments in~$(3 + 1)$-dimensions that validate our theoretical results.

\end{abstract}



\section{Introduction \label{SECT::INTRODUCTION}}

The present contribution focuses on the classical time-dependent advection--diffusion equations, also thought as a first step towards more complex nonlinear fluid dynamic problems.
More specifically, let the space--time cylinder~$\QT = \Omega \times (0, T)$, where~$\Omega \subset \R^d$ ($d = 2, 3$) {is} an open, bounded polytopic domain with Lipschitz boundary~$\partial \Omega$, and let~$T > 0$ represent the final time.
Then, for given strictly positive diffusion coefficient~$\nu$, transport solenoidal field~$\bbeta$, source term~$f$, and initial datum~$u_0$, we consider the following advection--diffusion IBVP:
\begin{equation}
\label{EQN::MODEL-PROBLEM}
\begin{cases}
\dpt u - \nu \Delta u + \bbeta \cdot \nabla u = f & \text{in }\QT,\\
u = 0 & \text{on } \GD, \\
u = u_0 & \text{on } \SO ,
\end{cases}
\end{equation}
where the surfaces~$\SO := \Omega \times \{0\}$, $\ST := \Omega \times \{T\}$, and~$\GD := \partial \Omega \times (0, T)$.

In the numerical analysis literature, problem \eqref{EQN::MODEL-PROBLEM}, in addition to its specific interest, has often represented an important step towards the study of more complex models, such as those describing incompressible fluid flows at high Reynolds numbers.
This concurs in motivating the very large amount of articles dealing with the so called advection-dominated case, that is in the development and analysis of numerical methods able to deliver accurate and reliable solutions also when
$|\bbeta| >\!\!> \nu$. Indeed, many stabilization techniques have been designed to address the well-known issue of spurious oscillations or instabilities of conforming finite element (FE) discretizations of model~\eqref{EQN::MODEL-PROBLEM} in the advection-dominated regime.
Such techniques include space--time least squares~\cite{Shakib_Hughes:1991}, streamline-upwind Petrov--Galerkin (SUPG) and variants \cite{brooks1982streamline,johnson1986streamline,franca1992stabilized,Bochev_Gunzburger_Shadid:2004,John_Schmeyer:2008}, local projection stabilization (LPS) \cite{Ahmed_Matthies_Tobiska_Xie:2011,Dong_Li:2021}, and other symmetric stabilization terms~\cite{Burman_Fernandez:2009}; see also the analysis in an abstract framework for spatial discretizations based on symmetric stabilization terms with discontinuous Galerkin time stepping in~\cite{Ern_Schieweck:2016}.

Some simplifications are commonly found in the literature, such as omitting the effect of the time discretizations~\cite{Burman_SMith:2011}, restricting to low-order time stepping schemes~\cite{Bochev_Gunzburger_Shadid:2004,John_Novo:2011,Burman_Fernandez:2009,DeFrutos_Garcia_Novo:2016,Zhang_Liu:2022}, and requiring a sufficiently strong reaction term~$\sigma u$ such that (see, e.g., \cite{Ahmed_Matthies_Tobiska_Xie:2011,John_Novo:2011,Ahmed_Matthies:2015,Srivastava_Ganesan:2020})
\begin{equation}\label{eq:ass:puzzle}
\inf_{(x, t) \in \QT} \Big(\sigma - \frac12 \div \bbeta \Big) \geq \sigma_0 > 0.
\end{equation}
Note that the latter condition is relevant to higher order schemes in time, where the standard discrete energy argument only leads to control in $L^2(\Omega)$ at discrete time instants (plus the sum of the time-jumps at the time-mesh nodes) but not in $L^2(0,T;L^2(\Omega))$ and even less so in $L^\infty(0,T;L^2(\Omega))$.
Condition \eqref{eq:ass:puzzle} is typically justified by the fact that the problem for the variable~$w = e^{-t/T} u$ satisfies such an assumption (at least when $\sigma - \frac12 \div \bbeta \ge 0$).
Although this is surely acceptable, we prefer to tackle the more complex (from the theoretical standpoint) case in which such transformation is not assumed so that no data modification is needed in the method, see Remark \ref{rem:ourchoice}.
This choice is also motivated by the possible extension to nonlinear problems, where the above transformation would induce the introduction of time dependent factors also in front of the nonlinear terms.
Similar observations hold for the more recent Virtual Element {(VE)} technology~\cite{volley}, its literature being clearly less rich than the
{FE} one; some articles dealing with the above issue are \cite{SUPG-Berrone,SUPG-virtual-BDLV,li2021local,beirao2024cip}.

The present work concerns the design and analysis of {an} SUPG-stabilized version of the fully discrete scheme obtained by combining a conforming
{FE} or
{VE} spatial discretization with an upwind-DG time stepping; in particular, we focus on the robustness analysis of the method in the advection-dominated regime (i.e., when~$0 \le \nu \ll 1$). Our main contributions are the following:

\begin{itemize}

\item We carry out the first stability and error analysis of a high-order-in-time SUPG-stabilized scheme for the time-dependent advection--diffusion IBVP~\eqref{EQN::MODEL-PROBLEM}, which does not require the transformation of the original problem, and does not rely on the presence of a positive reaction term.
Although stability and optimal convergence of the upwind-discontinuous Galerkin (DG) SUPG-stabilized finite element method (FEM) have been hypothesized~\cite[\S3.3]{Ahmed_Matthies:2016}, a thorough analysis was missing in the literature even for
FEM (and less so for VEM, for which the proposed methodology is novel).

\item We address, in a unified framework, conforming
{FE and VE} spatial discretizations. Our analysis focuses on VE spaces; however, the same ideas apply to conforming FE spaces with some simplifications (as detailed in Section~\ref{SUBSECT::VE-VERSIONS}).

\item We show an inf-sup estimate with stability constant independent of the meshsize, the time step, and the diffusion coefficient~$\nu$.
Such an estimate is used to prove that, in a certain energy norm, the fully discrete solution satisfies: \emph{i)} a continuous dependence on the data of the problem uniformly in~$\nu$, and \emph{ii)} some~\emph{a priori} error bounds, which do not degenerate when the diffusion coefficient~$\nu$ is small, and depend only on the interpolation and the nonconsistency errors.

\item At the end of the article we evaluate the practical performance of the proposed scheme through a set of numerical tests in~{$(3 + 1)$-}dimensions, for different orders of approximation in time and space.

\end{itemize}

The manuscript is organized as follows. In Section~\ref{S1}, we present some basic notation, the variational form of the continuous problem and a well-known stability result. In Section \ref{SECT::METHOD}, we present the proposed VEM, discuss the particular case of FEM and the extension to Serendipity
VEM, and present our main theoretical results.
Section~\ref{SECT::WELL-POSEDNESS} is devoted to the proof of the well-posedness and stability of the scheme.
In Section \ref{SECT::CONVERGENCE}, we develop
the convergence analysis, first addressing the general case of changing meshes (that is, when the spatial mesh can change form one time-slab to the next {one}) and then particularizing to the more favorable case with fixed spatial mesh.
Finally, we present some numerical tests
in Section \ref{sec:num} and make some concluding remarks in Section~\ref{SECT::CONCLUSIONS}.

\section{Basic notation and weak formulation of the problem.}\label{S1}

We start by reviewing some basic notation we will use through the article.
We denote the first-order time derivative operator by~$\dpt$, and the spatial gradient and Laplacian operators by~$\nabla$  and~$\Delta$, respectively.
We will use standard notation for Sobolev spaces, seminorms, and norms~\cite{McLean:2000}. For instance, given an open, bounded domain~$\Upsilon \subset \R^d$ ($d \in \N$), and scalars~$p \in [2, \infty]$ and~$s \in \R$, we denote by~$W^{s, p}(\Upsilon)$ the standard Sobolev space, and its associated seminorm and norm by~$\SemiNorm{\cdot}{W^{s, p}(\Upsilon)}$ and~$\Norm{\cdot}{W^{s, p}(\Upsilon)}$, respectively.
In particular, for~$p = 2$, we use the notation~$H^s(\Upsilon) = W^{s, 2}(\Upsilon)$, and denote its associated seminorm and norm by~$\SemiNorm{\cdot}{H^s(\Upsilon)}$ and~$\Norm{\cdot}{H^s(\Upsilon)}$, respectively. Moreover, the space~$L^2(\Upsilon) = H^0(\Upsilon)$ denotes the space of Lebesgue square integrable functions over~$\Upsilon$ {with its corresponding inner product~$(\cdot, \cdot)_{\Upsilon}$}, and~$H_0^1(\Upsilon)$ denotes the space of functions in~$H^1(\Upsilon)$ with zero trace on~$\partial \Upsilon$.
A superscript~$d$ is used to represent the seminorms and norms of vector fields with~$d$-components.
{In addition}, given a Banach space~$(X, \Norm{\cdot}{X})$, a time interval~$(a, b)$, and a scalar~$s \in \R$, we denote the Bochner--Sobolev space by~$H^s(a, b; X)$.
{Finally, we use the following notation for the algebraic tensor product of two spaces, say~$V$ and~$W$:
\begin{equation*}
V \otimes W := \mathrm{span} \{vw \, :\, v \in V \text{ and } w \in W\}.
\end{equation*}
}
Given~$k \in \N$, we denote the space of polynomials of degree at most~$k$ defined on~$\Upsilon$ by~$\Pp{k}(\Upsilon)$.

\smallskip
For the time being, we assume the following data regularity.
The transport advective field~$\bbeta \in
{W^{1, \infty}(0, T; L^{\infty}(\Omega))}$
with~$\div \bbeta = 0$, the source term~$f \in L^2(\QT)$, and the initial datum~$u_0 \in L^2(\Omega)$.
Then, denoting by~$a : H_0^1(\Omega) \times H_0^1(\Omega) \rightarrow \R$ and~$b : H_0^1(\Omega) \times H_0^1(\Omega) \rightarrow \R$ the following bilinear forms:
\begin{equation}
\label{EQN::BILINEAR-FORMS}
a(u, v) = \int_{\Omega} \nabla u \cdot \nabla v \dx \quad \text{ and } \quad b(u, v) = \int_{\Omega} v \bbeta \cdot \nabla u \dx,
\end{equation}
the continuous week formulation of the IBVP~\eqref{EQN::MODEL-PROBLEM} reads~(see~\cite[\S7.1.1]{Evans:2022}): find~$u \in L^2(0, T; H_0^1(\Omega)) \cap H^1(0, T; H^{-1}(\Omega)) \subset C([0, T]; L^2(\Omega))$ such that~$u = u_0$ on~$\SO$, and for almost all~$t \in (0, T)$, it holds
\begin{equation}
\label{EQN::CONTINUOUS-WEAK-FORMULATION}
\qquad \langle \dpt u, v \rangle + \nu a(u, v) + b(u, v) = (f, v)_{\Omega}  \qquad \forall v \in H_0^1(\Omega),
\end{equation}
where~$\langle \cdot, \cdot \rangle$ denotes the duality between~$H^{-1}(\Omega)$ and~$H_0^1(\Omega)$.

For any~$\tilde{t} \in (0, T]$, integrating in time equation~\eqref{EQN::CONTINUOUS-WEAK-FORMULATION} over~$(0, \tilde{t})$ and using the skew symmetry of the bilinear form~$b(\cdot, \cdot)$ (i.e., $b(u, v) = - b(v, u)$), and the H\"older and the Young inequalities, we obtain the following bound:
\begin{equation}
\label{EQN::AUX-STABILITY-CONTINUOUS-FORMULATION}
\frac12 \Norm{u(\cdot, \tilde{t})}{L^2(\Omega)}^2 + \nu \Norm{u}{L^2(0, \tilde{t}; H^1(\Omega))}^2
\leq \frac12 \Norm{u_0}{L^2(\Omega)}^2 + \Norm{f}{L^1(0, \tilde{t}; L^2(\Omega))}^2 + \frac14 \Norm{u}{L^{\infty}(0, \tilde{t}; L^2(\Omega))}^2.
\end{equation}

Since~$u \in C^0([0, T]; L^2(\Omega))$, we can take the maximum over~$[0, T]$ in~\eqref{EQN::AUX-STABILITY-CONTINUOUS-FORMULATION} and deduce the following stability estimate:
\begin{equation*}
\frac14 \Norm{u}{L^{\infty}(0, T; L^2(\Omega))}^2 + \nu \Norm{u}{L^2(0, T; H^1(\Omega))}^2 \leq \frac12 \Norm{u_0}{L^2(\Omega)}^2 + \Norm{f}{L^1(0, T; L^2(\Omega))}^2.
\end{equation*}

The above inequality shows a uniform-in-$\nu$ continuous dependence of the solution to~\eqref{EQN::CONTINUOUS-WEAK-FORMULATION} on the data of the problem.
Such a property is clearly desirable to be reproduced at the discrete level.

\section{Description of the method and main results\label{SECT::METHOD}}

In this section, we describe the proposed {SUPG-stabilized time-DG VEM} for the discretization of model~\eqref{EQN::MODEL-PROBLEM} and present the major theoretical results.
Some notation for tensor-product-in-time meshes is introduced in Section~\ref{SUBSECT::MESHES}.
In Section~\ref{SUBSECT::LOCAL-SPACES-PROJECTIONS}, we recall the definition of the local enhanced VE spaces in two and three dimensions, their corresponding degrees of freedom, and some computable polynomial projections.
In Section~\ref{SUBSECT::GLOBAL-SPACE-BILINEAR-FORMS}, global discrete spaces are defined as the tensor product of the space of piecewise polynomials in time and~$H_0^1(\Omega)$-conforming VE spaces, and we present the discrete bilinear forms in the definition of the {SUPG-stabilized time-DG VEM} in Section~\ref{SUBSECT::VEM--DG}.
Finally the main theoretical results, asserting the well posedness of the discrete problem and its convergence properties, are presented in Section~\ref{SUBSECT::theores}.

\subsection{Space--time mesh notation and assumptions \label{SUBSECT::MESHES}}
Let~$\{\Omegah\}_{h > 0}$ be a family of polytopic partitions of the spatial domain~$\Omega \subset \R^d$ with $d = 2, 3$ (see Remark \ref{rem:mesh-changing} regarding the case with variable spatial meshes).
For each~$K \in \{\Omegah\}_{h>0}$ and each facet~${F}$ of~$K$, we denote by~$\hK$ and~$h_{{F}}$ the diameters of~$K$ and~${F}$, respectively.
{We make the following assumption on the family~$\{\Omegah\}_{h > 0}$.}
\begin{assumption}[Mesh regularity]
\label{ASM::MESH-REGULARITY}
There exists a {strictly} positive constant~$\varrho$ such that the following conditions hold for any element~$K \in \{\Omegah\}_{h>0}$:
\begin{enumerate}[label = \textbf{\emph{(A\arabic*)}}, ref = (A\arabic*)]
\item \label{A1} $K$ is star-shaped with respect to a ball~$\BK$ of radius larger than or equal to~$\varrho \hK$;
\item \label{A2} (if $d=2$) each facet~$F$ has length larger than or equal to~$\varrho \hK$;
\item \label{A3} (if $d=3$) each facet~$F$ is star-shaped with respect to a disk~$\Bf$ of radius larger than or equal to~$\varrho \hK$ and each edge~$e$ of the (polygonal) facet~$F$ has length larger than or equal to~$\varrho \hK$.
\end{enumerate}
In particular, Assumptions~\ref{A1}--\ref{A3} imply the existence of a uniform maximum number of facets for each element~$K$ of~$\{\Omegah\}_{h>0}$.
\end{assumption}
Let~$\Tt$ be a partition of the time interval~$(0, T)$ given by~$0 := t_0 < t_1 < \ldots < t_N := T$.
For~$n = 1, \ldots, N$, we define the time interval~$I_n := (\tnmo, \tn)$, the surface~$\Sn := \Omega \times \{\tn\}$, and the time step~$\tau_n := \tn - \tnmo$. We further define the spatial meshsize~$h := \max_{K \in \Omegah} \hK$, the minimum element diameter~$h_{\min} := \min_{K \in \Omegah} \hK$, and the maximum time step~$\tau := \max_{n = 1, \ldots, N} \tau_n$.

Finally, for each~$K \in \Omegah$ and~$n = 1, \ldots, N$, we define the space--time prism~$\Kn := K \times I_n$.

\begin{remark}[Relaxation of the mesh assumptions]
The above regularity assumptions on the mesh could be partially relaxed, see for instance~\cite[\S2.3]{BLR-stability}, \cite[\S2 and~\S5.1]{Brenner-smallfaces}, and~\cite[\S3]{Chen-anisotropic} for some results in this direction.
\eremk
\end{remark}
\subsection{Local virtual element spaces and projections \label{SUBSECT::LOCAL-SPACES-PROJECTIONS}}
Let~$k$ and~$r$ be integer numbers such that~$k \geq 1$ and~$r \geq 0$, which denote the ``degrees of approximation" in space and time, respectively.

Let~$\calD \subset \R^d$ be an open~$(d - j)$-polytope for some~$j \in \{0, \ldots, d - 1\}$. We denote by~$\bx_{\calD}$ the centroid of~$\calD$, and introduce the following scaled and shifted monomial basis for the space~$\Pp{k}(\calD)$:
\begin{equation*}
\mathcal{M}_k(\calD) := \bigcup_{\ell = 0}^{k} \mathbb{M}_{\ell}(\calD) \ \text{ with }\ \mathbb{M}_{\ell}(\calD) := \Big\{\ma = \Big(\frac{\bx - \bx_{\calD}}{\hK}\Big)^{\ba} \ : \ \ba \in \N^d \text{ with } |\ba| = \ell \Big\}.
\end{equation*}
In the VE context, the use of these bases is particularly convenient for implementation.

We denote by~$\PiO{k}{K} : L^2(K) \rightarrow \Pp{k}(K)$ and~$\PPiO{k}{K} : L^2(K)^d \rightarrow \Pp{k}(K)^d$ the~$L^2(K)$-orthogonal projection operators in~$\Pp{k}(K)$ and~$\Pp{k}(K)^d$, respectively. Moreover, we denote by~$\PiN{k}{K}: H^1(K) \rightarrow \Pp{k}(K)$ the~$H^1(K)$-orthogonal projection operator, defined for any~$v \in H^1(K)$ as the solution to the following local problem:
\begin{equation*}
\begin{cases}
\displaystyle \int_K \nabla (\PiN{k}{K}v - v) \cdot  \nabla q_k \dx = 0 & \quad  \forall q_k \in \Pp{k}(K), \\[0.1in]
\displaystyle \int_{\partial K} (\PiN{k}{K} - v) \dS = 0. &
\end{cases}
\end{equation*}

\paragraph{Two-dimensional virtual element spaces.}
If~$\Omega \subset \R^2$, for each element~$K \in \Omegah$, we define the following local spaces:
\begin{subequations}
\begin{align}
\Bk(\partial K) & := \{v \in C^0(\partial K) \, : \, v_{\mid e} \in \Pp{k}(e) \text{ for any edge~$e$ of~$K$}\}, \\
\Vkt^{\sf 2D}(K) & := \{v \in H^1(K) \, : \, v_{\mid_{\partial K}} \in \Bk(\partial K) \text{ and } \Delta v \in \Pp{k}(K)\}, \\
\label{vem-loc-2d}
\Vk^{\sf 2D}(K) & := \big\{v \in \Vkt^{\sf 2D}(K) \, : \, \int_K (v - \PiN{k}{K} v) \ma \dx = 0 \ \text{ for all } \ma \in \mathbb{M}_{k-1}(K) \cup \mathbb{M}_k(K)\big\},
\end{align}
\end{subequations}
where the latter is the standard local enhanced VE space introduced in~\cite[\S3]{projectors}.
The following linear functionals constitute a set of unisolvent {degrees of freedom} (DoFs) for~$\Vk^{\sf 2D}(K)$
 (see~\cite[Prop.~2]{projectors}):
\begin{enumerate}[topsep = 0.2em, itemsep = 0.1em, label = \textbf{Dv\arabic*)}, ref = Dv\arabic*)]
\item\label{DOF-VERTICES-2D} the values of~$v$ at the vertices of~$K$;
\item\label{DOF-EDGES-2D} (if~$k \geq 2$) the values of~$v$ at~$(k - 1)$ distinct internal points along each edge~$e$ of~$K$;
\item\label{DOF-BULK-2D} (if~$k \geq 2$) the following moments of~$v$ against the elements of~$\mathcal{M}_{k - 2}(K)$:
\begin{equation*}
\frac{1}{|K|} \int_K v \ma \dx \quad \forall \ma \in \mathcal{M}_{k - 2}(K).
\end{equation*}
\end{enumerate}

\paragraph{Three-dimensional virtual element spaces.}
If~$\Omega \subset \R^3$, for each element~$K \in \Omegah$, we define the following local spaces:
\begin{subequations}
\begin{align}
\label{L-temp-1}
\Wk(\partial K) & := \{v \in C^0(\partial K) \, : \, v_{\mid F} \in \Vk^{\sf 2D}(F) \text{ for any face~$F$ of~$K$}\}, \\
\label{L-temp-2}
\Vkt^{\sf 3D}(K) & := \{v \in H^1(K) \, : \, v_{\mid_{\partial K}} \in \Wk(\partial K) \text{ and } \Delta v \in \Pp{k}(K)\}, \\
\label{vem-loc-3d}
\Vk^{\sf 3D}(K) & := \big\{v \in \Vkt^{\sf 3D}(K) \, : \, \int_K (v - \PiN{k}{K} v) \ma \dx = 0 \ \text{ for all } \ma \in \mathbb{M}_{k-1}(K) \cup \mathbb{M}_k(K)\big\},
\end{align}
\end{subequations}
where we have denoted by~$\Vk^{\sf 2D}(F)$ the local enhanced VE space on the face~$F$,
as~$F$ is contained in a (two-dimensional) plane.
The following linear functionals constitute a set of unisolvent DoFs for~$\Vk^{\sf 3D}(K)$
 (see~\cite[\S4.1]{projectors}):
\begin{enumerate}[topsep = 0.2em, itemsep = 0.1em, label = \textbf{Dv\arabic*${}^{*}$)}, ref = Dv\arabic*${}^{*}$)]
\begin{subequations}
\item\label{DOF-VERTICES-3D} the values of~$v$ at the vertices of~$K$;
\item\label{DOF-EDGES-3D} (if~$k \geq 2$) the following moments of~$v$ on each edge~$e$ of~$K$:
\begin{equation*}
\frac{1}{|e|} \int_e v \mc \ds \qquad \forall \mc \in \mathcal{M}_{k - 2}(e);
\end{equation*}
\item\label{DOF-FACE-3D} (if~$k \geq 2$) the following moments of~$v$ on each face~$F$ of~$K$:
\begin{equation*}
\frac{1}{|F|} \int_K v \mb \dx \qquad \forall \mb \in \mathcal{M}_{k - 2}(F);
\end{equation*}
\item\label{DOF-BULK-3D} (if~$k \geq 2$) the following moments of~$v$ on~$K$:
\begin{equation*}
\frac{1}{|K|} \int_K v \ma \dx \qquad \forall \ma \in \mathcal{M}_{k - 2}(K).
\end{equation*}
\end{subequations}
\end{enumerate}

Henceforth, we will denote by~$\Vk(K)$ the local enhanced VE space, regardless of the spatial dimension~$d$.

For any~$K \in \Omegah$ and~$v \in \Vk(K)$, the following polynomial projections are computable using the DoFs in sets~\ref{DOF-VERTICES-2D}--\ref{DOF-BULK-2D} (if~$d = 2$) or in sets~\ref{DOF-VERTICES-3D}--\ref{DOF-BULK-3D} (if~$d = 3$):
\begin{equation*}
\PiO{k}{K} v, \quad \PPiO{k}{K} \nabla v, \quad \PiN{k}{K} v.
\end{equation*}

\subsection{Global space and bilinear forms \label{SUBSECT::GLOBAL-SPACE-BILINEAR-FORMS}}
For~$d = 2$ or~$d = 3$, we define the global VE space
\begin{equation}\label{vem-glob}
\Vh := \{v \in H_0^1(\Omega)\, :\, v_{|_K} \in \Vk(K) \ \forall K \in \Omegah\},
\end{equation}
and the global space--time VE--DG space
\begin{equation*}
\begin{split}
\Vht & := \prod_{n = 1}^N \Pp{r}(I_n) {\otimes} \Vh.
\end{split}
\end{equation*}

For any piecewise scalar function~$w$ and~$m \in \{1, \ldots, N\}$,
we denote by~$w^{(m)}$ the restriction of~$w$ to the time slab~$\Omega \times I_m$. Moreover, for~$n = 1, \ldots, N - 1$, we define the time jump~$(\jump{\cdot}_n)$ of~$w$ as follows:
\begin{equation*}
\jump{w}_n (\bx) := w(\bx, \tn^-) - w(\bx, \tn^+) \quad \forall \bx \in \Omega,
\end{equation*}
where
$$w(\bx, t_n^{-}) := \lim_{\varepsilon \rightarrow 0^+} w^{(n)}(\bx, t_n - \varepsilon) \quad \text{ and } \quad w(\bx, t_n^{+}) := \lim_{\varepsilon \rightarrow 0^+} w^{(n+1)}(\bx, t_n + \varepsilon) .$$

We now introduce the discrete bilinear forms we use in the definition of the space--time VEM--DG formulation in Section~\ref{SUBSECT::VEM--DG} below.
Henceforth, we denote by~$\Id$ the identity operator. In the following the projection operators defined in Section~\ref{SUBSECT::LOCAL-SPACES-PROJECTIONS} are to be understood as applied pointwise in time.
\paragraph{The bilinear form~$\mht(\cdot, \cdot)$.}
We define the upwind-DG VE discretization of the first-order time derivative operator~$\dpt(\cdot)$ as follows:
\begin{equation*}
\begin{split}
\mht(\uh, \vh)  & := \sum_{n = 1}^N
\int_{I_n} \mh\big(\dpt \uh(\cdot, t), \vh(\cdot, t)\big) \dt
- \sum_{n = 1}^{N - 1} \mh\big(\jump{\uh}_n, \vh(\cdot, \tn^+)\big) \\
& \quad + \mh\big(\uh(\cdot, 0), \vh(\cdot, 0)\big),
\end{split}
\end{equation*}
where~$\mh : \Vh \times \Vh \rightarrow \R$ is the standard VE discretization of the~$L^2(\Omega)$-inner product, which can be written as
\begin{equation*}
\mh(u_h, v_h) = \sum_{K \in \Th} \mhK(u_h, v_h),
\end{equation*}
with local contributions~$\mhK(\cdot, \cdot)$ given by
\begin{align*}
\nonumber
\mhK(u_h, v_h) := (\PiO{k}{K} u_h, \PiO{k}{K} v_h)_{K} + \smK((\Id - \PiO{k}{K}) u_h, (\Id - \PiO{k}{K}) v_h),
\end{align*}
for some symmetric bilinear form~$\smK(\cdot, \cdot)$ chosen so that the following condition holds:

\begin{itemize}
\item[$\circ$] \textbf{Stability of~$\smK(\cdot, \cdot)$:} there exist positive constants~$\check{\mu}$ and~$\hat{\mu}$ independent of~$h$ and~$K$, but depending on the degree~$k$ and the parameter~$\varrho$ in Assumption~\ref{ASM::MESH-REGULARITY} such that
\begin{equation}
\label{EQN::STAB-smK}
\check{\mu}\Norm{v_h}{L^2(K)}^2 \leq \smK(v_h, v_h) \leq \hat{\mu} \Norm{v_h}{L^2(K)}^2 \quad \forall v_h \in \ker(\PiO{k}{K}) \cap \Vk(K).
\end{equation}
\end{itemize}

Defining
$$
\mu_* := \min\{1, \check{\mu}\} \quad \text{and} \quad \mu^* := \max\{1, \hat{\mu}\},
$$
the stability property~\eqref{EQN::STAB-smK} implies
\begin{equation}
\label{EQN::STAB::mK}
\mu_*\Norm{v_h}{L^2(K)}^2 \leq \mhK(v_h, v_h) \leq \mu^* \Norm{v_h}{L^2(K)}^2 \quad \forall v_h \in \Vk(K).
\end{equation}

\paragraph{The bilinear form~$\aht(\cdot, \cdot)$.} We discretize the spatial Laplacian operator~$(-\Delta)(\cdot)$ as follows:
\begin{equation*}
\aht(\uh, \vh)  := \sum_{n = 1}^{N} \int_{I_n} \ah\big(\uh(\cdot, t), \vh(\cdot, t) \big) \dt,
\end{equation*}
where~$\ah : \Vh \times \Vh$ is the VE discretization of the bilinear form~$a(\cdot, \cdot)$ in~\eqref{EQN::BILINEAR-FORMS}, which can be written as
\begin{align*}
\ah(u_h, v_h) = \sum_{K \in \Th} \ahK(u_h, v_h),
\end{align*}
with local contributions~$\ahK(\cdot, \cdot)$ given by
\begin{equation*}
\ahK(u_h, v_h) := \big(\PPiO{k-1}{K} \nabla u_h, \PPiO{k-1}{K} \nabla v_h\big)_{K} + \saK\big((\Id - \PiN{k}{K})u_h, (\Id - \PiN{k}{K}) v_h\big),
\end{equation*}
for some symmetric bilinear form~$\saK(\cdot, \cdot)$ chosen so that the following condition holds:

\begin{itemize}
\item[$\circ$] \textbf{Stability of~$\saK(\cdot, \cdot)$:} there exist positive constants~$\check{\alpha}$ and~$\hat{\alpha}$ independent of~$h$ and~$K$, but depending on the degree~$k$ and the parameter~$\varrho$ in Assumption~\ref{ASM::MESH-REGULARITY} such that
\begin{equation}
\label{EQN::STAB-saK}
\check{\alpha}\Norm{\nabla v_h}{L^2(K)^d}^2 \leq \saK(v_h, v_h) \leq \hat{\alpha} \Norm{\nabla v_h}{L^2(K)^d}^2 \qquad \forall v_h \in \ker(\PiN{k}{K}) \cap \Vk(K).
\end{equation}
\end{itemize}

Defining
\begin{equation*}
\alpha_* := \min\{1, \check{\alpha}\} \quad \text{and} \quad \alpha^* := \max\{1, \hat{\alpha}\},
\end{equation*}
using the stability property~\eqref{EQN::STAB-saK}, the stability of~$\PPiO{k}{K}$ in the~$L^2(K)^d$-norm, and the fact that~$\nabla \PiN{k}{K} v_h \in \Pp{k-1}(K)^d$, we deduce that
\begin{equation}
\label{EQN::STAB-aK}
\alpha_* \Norm{\nabla v_h}{L^2(K)^d}^2 \leq \ahK(v_h, v_h) \leq \alpha^* \Norm{\nabla v_h}{L^2(K)^d}^2 \qquad \forall v_h \in \Vk(K).
\end{equation}
\paragraph{The bilinear form~$\bht(\cdot, \cdot)$.}

As for the discretization of the advective term~$(\bbeta \cdot \nabla u)$, we introduce the following skew-symmetric bilinear form:

\begin{equation*}
\bht(\uh, \vh) := \frac12 \sum_{n = 1}^{N} \int_{I_n} \Big(\bh\big(\uh(\cdot, t), \vh(\cdot, t)\big) - \bh\big(\vh(\cdot, t), \uh(\cdot, t)\big) \Big) \dt,
\end{equation*}
where~$\bh : \Vh \times \Vh \rightarrow \R$ is given by
\begin{align}
\label{def_bh}
\bh(u_h, v_h) & := \sum_{K \in \Th}
\big(\bbeta \cdot \PPiO{k}{K} \nabla u_h,\, \PiO{k}{K} v_h\big)_{K}.
\end{align}
The above form \eqref{def_bh} corresponds to choice in~\cite[{Eq.~(4.5)}]{SUPG-virtual-BDLV} (see also~\cite[{Eq.~(15)}]{SUPG-Berrone}); in the two-dimensional case, it can be substituted with the alternative choice in~\cite[{Eq.~(4.6)}]{SUPG-virtual-BDLV}.

\paragraph{The bilinear form~$\sht(\cdot, \cdot)$.}

Finally, we introduce the {SUPG}-stability bilinear form
\begin{equation*}
\sht(\uh, \vh) := \sum_{n = 1}^{N} \sum_{K \in \Th} \shtKn\big(\uh, \vh\big),
\end{equation*}
with
\begin{align*}
\nonumber
\shtKn(\uh, \vh) := & \lambdaKn \big(\LhtK \uh, \tLhtK
\vh\big)_{\Kn} \\
& + \betaKn^2 \lambdaKn \int_{I_n}
\saK\big((\Id - \PiN{k}{K})\uh, (\Id - \PiN{k}{K}) \vh \big) \dt,
\end{align*}
for some parameter~$\lambdaKn > 0$ to be specified later, $\betaKn := \max\{\beta_{\varepsilon}, \Norm{\bbeta}{L^{\infty}(\Kn)^d}\}$ for some mesh-independent strictly positive ``safeguard'' constant~$\beta_{\varepsilon}$,
the stability term~$\saK(\cdot, \cdot)$ as in the definition of~$\ah(\cdot, \cdot)$, and the linear operators~$\LhtK$ and~$\tLhtK$ defined as follows:
\begin{subequations}
\begin{align}
\label{eq:tLhtK}
\tLhtK \vh &:= \dpt \PiO{k}{K} \vh + \bbeta \cdot \PPiO{k-1}{K} \nabla \vh, \\
\label{eq:LhtK}
\LhtK \uh &:= \dpt \PiO{k}{K} \uh - \nu \div \PPiO{k-1}{K} \nabla \uh + \bbeta \cdot \PPiO{k-1}{K} \nabla \uh.
\end{align}
\end{subequations}
For convenience, we also define
$$
\betaQT := \Norm{\bbeta}{L^{\infty}(\QT)^d} .
$$
and assume, up to suitable scalings of the data, that $\betaQT \simeq 1$.

\begin{remark}[Stability terms]
There exists a very large literature concerning different choices for the stabilization terms {for VE discretizations} and developing the associated theoretical support. Explicit expressions for the definition of the stability terms~$\smK(\cdot, \cdot)$ and~$\saK(\cdot, \cdot)$ can be found, for instance in \cite{volley,projectors,apollo} and some related proofs for instance in \cite{BLR-stability,Brenner-smallfaces}; see also~\cite{Mascotto:2023} for a recent discussion on the role of the stability terms for virtual element methods.
\eremk
\end{remark}
\subsection{Finite element and serendipity VE spaces \label{SUBSECT::VE-VERSIONS}}

The proposed method immediately extends to the case of the classical Lagrangian
{FEM}.
If the mesh is simplicial, one can substitute the local spaces \eqref{vem-loc-2d} (and \eqref{vem-loc-3d} in three space dimensions) with
$$
\Vk(K) := \Pp{k}(K)\qquad \forall K \in \Omegah ,
$$
thus obtaining, c.f. \eqref{vem-glob}, the standard Lagrangian
{FE} space.
In such a case, the scheme boils down to a standard {SUPG-stabilized time-DG FEM} approach, as all the polynomial projections appearing in the definition of the discrete forms disappear and, for the same reason, the stability terms vanish.
Therefore, the theoretical results of this article trivially extend to such (simpler) case, yielding new results also for classical FEMs.
Indeed, on our knowledge, results of this kind are missing in the literature, as previous SUPG schemes are low-order accurate in time (see, e.g., \cite{Bochev_Gunzburger_Shadid:2004,John_Novo:2011,DeFrutos_Garcia_Novo:2016,Zhang_Liu:2022}) and the analysis of many techniques rely on the
presence of a reaction term
(see, e.g., \cite{Ahmed_Matthies_Tobiska_Xie:2011,Ahmed_Matthies:2015}, the discussion in the introduction of this contribution and Remark \ref{rem:ourchoice}).

Another variant that can be considered is that of Serendipity Virtual Elements, which is a construction allowing to reduce the number of
DoFs, an asset which is particularly useful for high-order approaches as the present one. We refer to \cite{SERE-1} for a detailed presentation of Serendipity VEM (see also \cite{SERE-2} for the associated interpolation and stability analysis) and here limit ourselves to a very brief review of the construction in three space dimensions.
The idea is to eliminate DoFs that are internal to faces (since those that are internal to elements can be statically condensed) by introducing, for every face~$F$ of the polyhedral mesh, a projection operator
$$
\Pi^{\mathsf{S}}_F \ : \ \Vkt^{\sf 2D}(F) \longrightarrow \Pp{k}(F)
$$
that depends only on the DoFs associated
with the boundary of~$F$ (vertex values and edge pointwise values).
Clearly, such an operator can be constructed only if the~$\Pp{k}$-bubbles space on~$F$ reduces to~$\{0\}$, a condition that depends on the geometry of~$F$
and on the polynomial degree~$k$. Alternatively, one can use an extended construction (see~\cite[{\S3}]{SERE-1}), but in the present brief review we prefer to stick to the simpler case where no such bubbles exist. Once such an operator is available, one can introduce the smaller space
$$
\Vk^{\mathsf{S},\sf 2D}(F) := \big\{v \in \Vkt^{\sf 2D}(F) \, : \, \int_F (v - \Pi^{\mathsf{S}}_F v) \ma \dx = 0 \ \text{ for all } \ma \in \mathbb{M}_k(F)\big\},
$$
whose associated DoFs are only~\ref{DOF-VERTICES-2D} and \ref{DOF-EDGES-2D}.
Afterwards, one follows the same identical {\bf 3D} construction as in Section \ref{SUBSECT::LOCAL-SPACES-PROJECTIONS} but substituting the face spaces in \eqref{L-temp-1} with its Serendipity variant
$$
\Wk^{\mathsf{S}}(\partial K) := \{v \in C^0(\partial K) \, : \, v_{\mid f} \in \Vk^{\mathsf{S},\sf 2D}(f) \text{ for any face~$f$ of~$K$}\},
$$
and using such a boundary space in~\eqref{L-temp-2} instead of~$\Wk(\partial K)$.

The final space has only degrees of freedom of type \ref{DOF-VERTICES-3D}, \ref{DOF-EDGES-3D}, and \ref{DOF-BULK-3D} but retains all the key properties, such as approximation, of the original VE space \cite{SERE-1,SERE-2}.

\subsection{{SUPG-stabilized time-DG VEM}\label{SUBSECT::VEM--DG}}
The proposed
{SUPG-stabilized time-DG VEM} variational formulation is: find~$\uh \in \Vht$ such that
\begin{equation}
\label{EQN::FULLY-DISCRETE-SCHEME}
\begin{split}
\Bht(\uh, \vh) = \ellht(\vh) \quad \forall \vh \in \Vht,
\end{split}
\end{equation}
where
\begin{subequations}
\begin{align}
\label{DEF::Bht}
& \Bht(\uh, \vh) := \mht(\uh, \vh) + \nu \aht(\uh, \vh) + \bht(\uh, \vh) + \sht(\uh, \vh), \end{align}
and
\begin{align}
\label{DEF::lht}
& \ellht(\vh) := \sum_{K \in \Omegah}
\sum_{n = 1}^{N} \Big(f,\, \PiO{k}{K} \vh + \lambdaKn \tLhtK \vh
\Big)_{\Kn}  + \sum_{K \in \Omegah} \big( u_0,\, \PiO{k}{K} \vh(\cdot, 0)\big)_{K}.
\end{align}
\end{subequations}

\begin{remark}\label{rem:mesh-changing}
The method proposed above, and (unless clearly stated as will happen in Section \ref{sec:no-tau-below}) the stability and convergence analysis here developed apply identically to the case when the spatial mesh $\Omega_h$ changes at every time slab, that is we have a different mesh $\Omega_h^n$ for all $n=0,1,...,N$. Nevertheless, in order to allow {for} a simpler notation and a clearer exposition, we prefer to keep the above simpler setting in the following developments.
\eremk
\end{remark}

\subsection{Main theoretical results \label{SUBSECT::theores}}

In this section, we present the main theoretical results of the article, namely the well-posedness of the discrete problem~\eqref{EQN::FULLY-DISCRETE-SCHEME} and the associated quasi-robust error bounds. The proofs are postponed to Sections~\ref{SECT::WELL-POSEDNESS} and \ref{SECT::CONVERGENCE}, respectively.

In what follows, we write~$a \lesssim b$ to indicate the existence of a positive constant~$C$ independent of the meshsize~$h$, the time step~$\tau$, and the diffusion coefficient~$\nu$ such that~$a \le C b$. Moreover, we write~$a \simeq b$ meaning that~$a \lesssim b$ and~$b \lesssim a$.

We need to introduce some preliminary quantities. Let the upwind-jump functional
\begin{equation}
\label{EQN::JUMP-FUNCTIONAL}
\SemiNorm{\uh}{\sfJ}^2 := \frac12 \Big(\Norm{\uh}{L^2(\ST)}^2 + \sum_{n = 1}^{N - 1} \Norm{\jump{\uh}_n}{L^2(\Omega)}^2 + \Norm{\uh}{L^2(\SO)}^2 \Big),
\end{equation}
and the SUPG-functional
\begin{equation*}
\SemiNorm{\uh}{\supg}^2 := \sum_{n = 1}^N \sum_{K \in \Th} \SemiNorm{\uh}{\Knsupg}^2,
\end{equation*}
where
\begin{equation}
\label{EQN::LOCAL-SUPG-NORM}
\SemiNorm{\uh}{\Knsupg}^2:= \lambdaKn \Norm{\tLhtK \uh}{L^2(\Kn)}^2 + \lambdaKn \betaKn^2 \Norm{\nabla (\Id - \PiN{k}{K}) \uh}{L^2(\Kn)^d}^2.
\end{equation}

We also define the following norm in~$\Vht$:
\begin{equation}
\label{EQN::ENERGY-NORM}
\TNorm{\uh}^2 := \Norm{\uh}{L^2(\QT)}^2 + \SemiNorm{\uh}{\sfJ}^2 + \nu \Norm{\nabla \uh}{L^2(\QT)^d}^2 + \SemiNorm{\uh}{\supg}^2.
\end{equation}

The following assumption will be adopted in the sequel.

\begin{assumption}\label{TheAssumption}
For any~$K \in \Th$ and~$n = 1, \ldots, N$, let~$\lambdaKn$ be chosen so that
\begin{equation}
\label{EQN::CHOICE-LAMBDA}
\lambdaKn \leq \zeta \min\Big\{\frac{\hK^2}{\nu \Cinv^2}, \frac{\hK}{\betaQT}\Big\},
\end{equation}
with~$\Cinv$ an inverse-estimate constant (c.f. Lemma~\ref{LEMMA::INVERSE-ESTIMATE} below) and some scalar constant~$\zeta$ independent of~$h$, $\tau$, and~$\nu$.
Moreover, let the following mild condition hold for some positive constant~$C_*$ independent of~$h$ and~$\nu$:
\begin{equation}
\label{EQN::TAU-H}
\tau \leq C_* h_{\min}^{\frac12}.
\end{equation}
\end{assumption}

The proposed method satisfies the following inf-sup estimate with stability constant independent of the diffusion coefficient~$\nu$;
as a consequence, method~\eqref{EQN::FULLY-DISCRETE-SCHEME} is well posed and remains stable even in the advection-dominated regime~$(0 < \nu \ll 1)$.
The proof can be found in Section \ref{SECT::WELL-POSEDNESS}.

\begin{theorem}[Inf-sup stability]
\label{THM::FULLY-DISCRETE-STABILITY}
There exist positive constants~$C_*$, $\zeta$, and~$\gamma_I$ independent of~$h$, $\tau$, and~$\nu$ such that, if~$\tau \leq C_* h_{\min}^{\frac12}$, and the stability parameters~$\lambdaKn$ are chosen so that~\eqref{EQN::CHOICE-LAMBDA} is satisfied,
it holds
\begin{equation*}
\TNorm{\uh} \leq \gamma_I \sup_{\vh \in \Vht \setminus \{0\}} \frac{\Bht(\uh, \vh)}{\TNorm{\vh}} \quad \forall \vh \in \Vht.
\end{equation*}
\end{theorem}

The well-posedness of the scheme is a simple consequence of the above result (see the end of Section \ref{SECT::WELL-POSEDNESS}) and is stated here below.

\begin{corollary}[Well-posedness]\label{COROL:1}
Under the assumptions of Theorem~\ref{THM::FULLY-DISCRETE-STABILITY}, there exists a unique solution~$\uh \in \Vht$ to the discrete variational formulation~\eqref{EQN::FULLY-DISCRETE-SCHEME}.
Moreover, if~$h$ and~$\zeta$ are such that~$\lambdaKn \le 1$ for all~$K \in \Omegah$ and~$n \in \{1, \ldots, N\}$,
the following continuous dependence on the data is satisfied:
\begin{equation*}
    \TNorm{\uh} \leq \sqrt{2} \gamma_I \big(\Norm{u_0}{L^2(\Omega)} + \Norm{f}{L^2(\QT)} \big).
\end{equation*}
\end{corollary}

The constant~$1$ appearing above in the condition $\lambdaKn \le 1$ was introduced only for simplicity of exposition. What is actually required for the well-posedness is that all the~$\lambdaKn$ are uniformly bounded (in practice the~$\lambdaKn$ are expected to be ``small'', c.f. \eqref{EQN::CHOICE-LAMBDA}).

\begin{remark}[Stability in $L^2(0,T;L^2(\Omega))$]\label{rem:ourchoice}
Theorem \ref{THM::FULLY-DISCRETE-STABILITY} shows the stability of the fully discrete scheme in a norm including $L^2(0,T;L^2(\Omega))$ without resorting to the exponential transformation $w = e^{-t/T}u$, thus underlining that also the method applied directly to \eqref{EQN::MODEL-PROBLEM} is able to recover the full stability properties of the continuous problem.
\eremk
\end{remark}

We now introduce the main \emph{a priori} error estimates for our method, yielding optimal convergence rates in the energy norm~$\TNorm{\cdot}$ defined in~\eqref{EQN::ENERGY-NORM}. The error estimates hold under the following assumption on the regularity of the data and the exact solution.
\begin{assumption}[Data assumption]
\label{ASM::DATA}
For all~$K \in \Omegah$ and~$n \in \{1, \ldots, N\}$, the solution~$u$ to the continuous weak formulation in~\eqref{EQN::CONTINUOUS-WEAK-FORMULATION}, the source term~$f$, and the initial condition~$u_0$ satisfy:
\begin{alignat*}{2}
{u \in H^{q + 1}(I_n; H^1(K)) \cap H^1(I_n; H^{s + 1}(K)),  \
f \in L^2(I_n; H^{s+1}(K)), \ \text{ and } \ u_0 \in H^{s+1}(K),}
\end{alignat*}
and the advective field~$\bbeta$ satisfies,
\begin{equation*}
\bbeta\in {L^{\infty}(0, T; W^{s + 1, \infty}(\Omega)) \cap W^{1, \infty}(0, T; L^{\infty}(\Omega))},
\end{equation*}
for~$0 \le q \le r$ and~$0 \le s \le k$.
\end{assumption}

The following convergence result holds (the proof can be found in Section \ref{SECT::CONVERGENCE}).

\begin{theorem}[\emph{A priori} error estimates]
\label{THM::ERROR-ESTIMATES}
Let Assumption~\ref{ASM::MESH-REGULARITY} (on the mesh-regularity) and Assumption~\ref{TheAssumption} (on the choice of~$\lambdaKn$ and the relation of~$\tau$ and~$h_{\min}$) hold. Let also Assumption~\ref{ASM::DATA} hold with~$q = s = r = k$.
If the time steps~$\{\tau_n\}_{n = 1}^N$ are quasi-uniform and the solution $u$ is sufficiently regular, we have the following error estimates (all bounds being uniform in $\nu$):
\end{theorem}
\begin{itemize}
\item Advection-dominated regime $(\nu \ll \hK)$: $\lambdaKn = \zeta \frac{\hK}{\betaQT} \simeq \hK$
\begin{equation}\label{eq:fin:conv}
\TNorm{u - \uh}^2 \lesssim \tau^{2k+1} + h^{2k+1} + h^2 \tau^{2k} + h^{2k+1} \big( h/\tau + h^2/\tau^2) + \tau^{2k+2}/h_{\textrm{min}}.
\end{equation}
\item Diffusion-dominated regime $(\nu \simeq 1, \, \hK \lesssim \nu)$: $\lambdaKn = \zeta \frac{\hK^2}{\nu \Cinv^2} \simeq \hK^2 $
\begin{equation}\label{eq:fin:diff}
\TNorm{u - \uh}^2 \lesssim \tau^{2k+1} + h^{2k} + h^2 \tau^{2k} + h^{2k+2}/\tau^2 + \tau^{2k+2}/h_{\textrm{min}}^2.
\end{equation}
\end{itemize}

The quantities with negative powers of $h_{\textrm{min}}$ in \eqref{eq:fin:conv}--\eqref{eq:fin:diff} suggest assuming quasi-uniformity also of the spatial mesh.
In such a case, whenever the meshsize~$h$ and the time step~$\tau$ are orthotropic (i.e.,~$\tau \simeq h$) we immediately observe that the error satisfies $\TNorm{u - \uh} \lesssim h^k$ in diffusion dominated cases and~$\TNorm{u - \uh} \lesssim h^{k+1/2}$ in advection dominated cases, which are the optimal behaviours expected for quasi-robust schemes.
Another important observation is that, since reducing $\tau$ is computationally cheaper than reducing $h$, one may be interested in the case $\tau \ll h$, representing the situation in which the time mesh is substantially finer than the spatial mesh. The presence of terms of the kind~$h/\tau$ in the error estimate above are detrimental in this respect: we investigate the possibility of eliminating such terms in Section \ref{sec:no-tau-below}.

\section{Well-posedness of the method\label{SECT::WELL-POSEDNESS}}

This section is devoted to prove Theorem~\ref{THM::FULLY-DISCRETE-STABILITY} and Corollary \ref{COROL:1}.

\subsection{Some useful tools\label{SUBSECT::TOOLS-STABILITY}}
In the proof of the inf-sup stability estimate in Theorem~\ref{THM::FULLY-DISCRETE-STABILITY} we make use of the following auxiliary exponential weight function:
\begin{equation}
\label{DEF::VARPHI}
\varphi = \varphi(t) := T\exp((T-t)/T),
\end{equation}
which satisfies the following two important uniform bounds:
\begin{subequations}
\label{EQN::BOUNDS-VARPHI}
\begin{align}
\label{EQN::BOUNDS-VARPHI-1}
T \leq & \varphi(t) \leq eT  \quad \quad\ \forall t \in [0, T], \\
\label{EQN::BOUNDS-VARPHI-2}
-1  \leq & \varphi'(t) \leq -e^{-1}  \quad \forall t \in [0, T].
\end{align}
\end{subequations}
We denote by~$\Pi_r^t : L^2(0, T) \rightarrow \Pp{k}(\Tt)$ the~$L^2(0, T)$-orthogonal projection operator in~$\Pp{r}(\Tt)$.
In what follows, the operator~$\Pi_r^t$ is to be understood as applied pointwise in space.

We start by some inverse estimates for VE functions and polynomials.
\begin{lemma}[Local inverse estimates]
\label{LEMMA::INVERSE-ESTIMATE}
Let~$\Th$ satisfy Assumption~\ref{ASM::MESH-REGULARITY}. Then, for all~$K \in \Th$ and~$n = 1, \ldots, N$, the following bounds hold:
\begin{subequations}
\begin{align}
\label{EQN::INVERSE-ESTIMATE-SPACE}
\Norm{\nabla {w_{r,k}}}{L^2(I_n; L^2(K)^d)} & \leq \Cinv \hK^{-1} \Norm{{w_{r,k}}}{L^2(I_n; L^2(K))} && \forall {w_{r, k}} \in \Pp{r}(I_n) \otimes \Vk(K), \\
\label{EQN::INVERSE-ESTIMATE-TIME}
\Norm{\dpt w_{r}}{L^2(I_n; L^2(K))} & \leq \Cinv \tau_n^{-1} \Norm{w_{r}}{L^2(I_n; L^2(K))} && \forall w_{r} \in \Pp{r}(I_n) \otimes L^2(K), \\
\label{EQN::INVERSE-ESTIMATE-DIVERGENCE}
\Norm{\div \mathbf{p}_{r,k}}{L^2(I_n; L^2(K))} & \leq \Cinv \hK^{-1} \Norm{\mathbf{p}_{r, k}}{L^2(I_n; L^2(K)^d)} && \forall \mathbf{p}_{r, k} \in \Pp{r}(I_n) \otimes \Pp{k}(K)^d,
\end{align}
\end{subequations}
for some positive constant~$\Cinv$ independent of~$K$, $\hK$, and~$\tau_n$.
\end{lemma}
\begin{proof}
The inverse estimate in~\eqref{EQN::INVERSE-ESTIMATE-SPACE} follows from the tensor-product structure of the element~$K \times I_n$ and the space~$\Pp{r}(I_n) {\otimes} \Vk(K)$, and the inverse estimates for VE functions in~\cite[Lemma~6.1]{Vacca:2018} (for~$d = 2$) and~\cite[Thm.~3.4]{Huang_Yu:2023} (for~$d = 3$).
The inverse estimates~\eqref{EQN::INVERSE-ESTIMATE-TIME} and~\eqref{EQN::INVERSE-ESTIMATE-DIVERGENCE} are standard; see e.g., \cite[\S4.5]{Brenner-Scott:book}.
\end{proof}

In next lemma, we recall some approximation properties of~$\Pi_r^t$ from~\cite[Lemma~4.3]{Wang_Rhebergen:2023}.
\begin{lemma}
\label{LEMMA::APPROXIMATION-I-PI}
There exits a positive constant~$C_S$ independent of~$h$ and~$\tau$ such that, for~$n = 1, \ldots, N$ and $K \in \Omegah$, the following bounds hold:
\begin{subequations}
\begin{alignat}{3}
\label{EQN::APPROX-I-PI-1}
\Norm{(\Id - \Pi_r^t) (\varphi {\wt})}{L^2(\Kn)} & \leq C_S \tau_n \Norm{{\wt}}{L^2(\Kn)} & & \qquad \forall \wt \in {\Pp{r}(I_n) \otimes L^2(K)},\\
\label{EQN::APPROX-I-PI-Time-Der}
\Norm{\dpt (\Id - \Pi_r^t) (\varphi {\wt})}{L^2(\Kn)} & \leq C_S \Norm{{\wt}}{L^2(\Kn)} & & \qquad \forall \wt \in {\Pp{r}(I_n) \otimes L^2(K)},\\
\label{EQN::APPROX-I-PI-2}
\Norm{(\Id - \Pi_r^t) (\varphi \wt)}{L^2(\Snmo)} & \leq C_S \tau_{n}^{\frac12} \Norm{\wt}{L^2(I_n; L^2(\Omega))}
& & \qquad \forall \wt \in {\Pp{r}(I_n) \otimes L^2(\Omega)},\\
\label{EQN::APPROX-I-PI-3}
\Norm{(\Id - \Pi_r^t) (\varphi \wt)}{L^2(\Sn)} & \leq C_S \tau_n^{\frac12} \Norm{\wt}{L^2(I_n; L^2(\Omega))}
& & \qquad \forall {\wt \in \Pp{r}(I_n) \otimes L^2(\Omega)}.
\end{alignat}
\end{subequations}
\end{lemma}
\begin{proof}
We show only the proof of \eqref{EQN::APPROX-I-PI-Time-Der}; the other bounds can be derived with very similar arguments.
We start by introducing $\overline{\varphi}:=\Pi_0^t \varphi$ and applying some simple steps
\begin{equation}\label{L:eq:new}
\begin{aligned}
& \Norm{\dpt (\Id - \Pi_r^t) (\varphi \wt)}{L^2(\Kn)} = \Norm{\dpt (\Id - \Pi_r^t) \big( (\varphi-\overline{\varphi} \big) \wt)}{L^2(\Kn)}
\le  \Norm{\dpt {((\varphi-\overline{\varphi}) \wt) }}{L^2(\Kn)} \\
& \quad + \Norm{\dpt \Pi_r^t ((\varphi-\overline{\varphi}) \wt)}{L^2(\Kn)} \le
\Norm{\varphi' \wt}{L^2(\Kn)} + \Norm{(\varphi-\overline{\varphi}) \dpt \wt}{L^2(\Kn)}
+ \Norm{\dpt \Pi_r^t ((\varphi-\overline{\varphi}) \wt)}{L^2(\Kn)} .
\end{aligned}
\end{equation}
The first term on the right-hand side is bounded trivially by $\Norm{\wt}{L^2(\Kn)}$ using \eqref{EQN::BOUNDS-VARPHI-2}.
The second term is bounded first by standard approximation properties of constant polynomials, afterwards by recalling \eqref{EQN::BOUNDS-VARPHI-2} and \eqref{EQN::INVERSE-ESTIMATE-TIME}:
$$
\Norm{(\varphi-\overline{\varphi}) \dpt \wt}{L^2(\Kn)} \le
\Norm{\varphi-\overline{\varphi}}{L^\infty(I_n)} \Norm{\dpt \wt}{L^2(\Kn)}
\le {C} \tau_n \Norm{\varphi'}{L^\infty(I_n)} \Norm{\dpt \wt}{L^2(\Kn)}
\le C \Norm{\wt}{L^2(\Kn)},
$$
{where~$C$ is a generic constant independent of~$\tau_n$.}

The last term in \eqref{L:eq:new} is bounded similarly. Since the function is polynomial in time we can apply an inverse estimate, afterwards use the continuity of the projection operator, and finally deploy again standard approximation properties of constant polynomials. We obtain
$$
\Norm{\dpt \Pi_r^t ((\varphi-\overline{\varphi}) \wt)}{L^2(\Kn)} \le
C \tau_n^{-1} \Norm{\Pi_r^t ((\varphi-\overline{\varphi}) \wt)}{L^2(\Kn)}
\le C \tau_n^{-1} \Norm{(\varphi-\overline{\varphi}) \wt}{L^2(\Kn)}
\le C \Norm{\wt}{L^2(\Kn)},
$$
{which completes the proof}.
\end{proof}

Furthermore, the following estimate follows from~\cite[Lemma~A.1(e)]{Diening_Storn_Tscherpel:2023} and the equivalence of~$\Pi_0^t$ and the averaged Taylor polynomial~$T^0$ defined in~\cite[Eq.~(A.1)]{Diening_Storn_Tscherpel:2023}:
for all~$K \in \Th$ and~$n = 1, \ldots, N$, it holds
\begin{equation}
\label{EQN::BETA-ESTIMATE}
\Norm{\bbeta - \Pi_0^t \bbeta}{L^{\infty}(I_n; L^q(K))} \leq C_0 \tau_n \Norm{\dpt \bbeta}{L^{\infty}(I_n; L^q(K))} \quad \forall q \in [1, \infty],
\end{equation}
for some positive constant~$C_0$ independent of~$h$ and~$\tau$.

We finally prove a simple orthogonality property for generic bilinear forms~$s(\cdot, \cdot)$ on $\Vk(K)$.
\begin{lemma}
\label{LEMMA::ORTHOGONALITY-STAB}
Let~$\varphi$ be defined as in~\eqref{DEF::VARPHI} and $s(\cdot,\cdot)$ be a bilinear form on $\Vk(K)$. For any~$K \in \Th$ and~$n = 1, \ldots, N$, it holds
\begin{subequations}
\label{EQN::ORTHOGONALITY-STAB}
\begin{align}
\label{EQN::ORTHOGONALITY-STAB-1}
\int_{I_n} s(\uh, (\Id - \Pi_r^t)(\varphi \vh)) \dt & = 0  \qquad \forall \uh, \vh \in \Pp{r}(I_n) {\otimes} \Vk(K).
\end{align}
\end{subequations}
\end{lemma}
\begin{proof}
Let~$K \in \Th$, $n \in \{1, \ldots, N\}$, and~$\uh, \vh \in \Pp{r}(I_n) \otimes \Vk(K)$.
Moreover, let~$\{\phi_\ell\}_{\ell = 1}^{\dim(\Vk(K))}$ be a basis for the space~$\Vk(K)$.
Then, there exist polynomials~$\{\alpha_{\ell}\}_{\ell = 1}^{\dim(\Vk(K))} \subset \Pp{r}(I_n)$ and~$\{\beta_{m}\}_{m = 1}^{\dim(\Vk(K))} \subset \Pp{r}(I_n)$ such that
\begin{align*}
\uh(x, t) & = \sum_{\ell = 1}^{\dim(\Vk(K))} \alpha_{\ell} (t) \phi_{\ell}(x)  & & \forall (x, t) \in \Kn, \\
(\Id - \Pi_r^t)(\varphi \vh)(x, t) & = \sum_{m = 1}^{\dim(\Vk(K))} (\Id - \Pi_r^t)\big(\varphi(t) \beta_{m} (t)\big) \phi_m(x)  & & \forall (x, t) \in \Kn.
\end{align*}
Therefore,
\begin{align*}
\int_{I_n} s(\uh, (\Id - \Pi_r^t)(\varphi \vh)) \dt = \sum_{\ell = 1}^{\dim(\Vk(K))} \sum_{m = 1}^{\dim(\Vk(K))} s(\phi_{\ell}, \phi_{m}) \int_{I_n} \alpha_{\ell}(t) (\Id - \Pi_r^t)(\varphi(t) \beta_m(t)) \dt = 0,
\end{align*}
which completes the proof.
\end{proof}

\subsection{Inf-sup stability \label{SUBSECT::INF-SUP-STABILITY}}

For the sake of clarity, in next Lemmas we show some bounds that will be used to prove the inf-sup stability estimate in Theorem~\ref{THM::FULLY-DISCRETE-STABILITY}.

\begin{lemma}
\label{LEMMA::BOUND-VH-UH}
Let Assumption \ref{TheAssumption} hold.
Then, under the same notation, for any~$\uh \in \Vht$ and~$\vh:= \Pi_r^t(\varphi \uh) + \theta \uh \in \Vht$, with~$\varphi$ defined in~\eqref{DEF::VARPHI} and some positive constant~$\theta$, the following bound holds:
\begin{equation}
\label{EQN::BOUND-VH-UH}
    \TNorm{\vh} \leq
    \sqrt{2} \Big[2(eT)^2 + 2C_S^2 \tau + 4 \zeta \big((1 + C_S^2)\betaQT^{-1} h + \delta^2 \betaQT C_*^2 \big) + \theta^2\Big]^{\frac12} \TNorm{\uh},
\end{equation}
with~$\delta = C_S \Cinv$.
\end{lemma}
\begin{proof}
Let~$\uh \in \Vht$. Using the triangle inequality and the definition of~$\vh$, we get
\begin{equation}
\label{EQN::TRIANGLE-INEQ-VH}
\TNorm{\vh}^2 \leq 2 \TNorm{\Pi_r^t(\varphi \uh)}^2 + 2 \theta^2 \TNorm{\uh}^2.
\end{equation}
Hence, it only remains to bound the first term on the right-hand side of~\eqref{EQN::TRIANGLE-INEQ-VH}.

\begin{subequations}
The following estimates follow immediately from the stability properties of~$\Pi_r^t$, the commutativity of the spatial gradient~$\nabla$ and the operator~$\Pi_r^t$, and bound~\eqref{EQN::BOUNDS-VARPHI-1} {for~$\varphi$}:
\begin{align}
\label{EQN::ESTIMATE-VH-1}
\Norm{\Pi_r^t(\varphi \uh)}{L^2(\QT)}^2 & \leq (eT)^2 \Norm{\uh}{L^2(\QT)}^2, \\
\label{EQN::ESTIMATE-VH-2}
\nu \Norm{\nabla (\Pi_r^t(\varphi \uh))}{L^2(\QT)^d}^2 & \leq \nu (eT)^2\Norm{\nabla \uh}{L^2(\QT)^d}^2, \\
\label{EQN::ESTIMATE-VH-AUX}
\SemiNorm{\varphi \uh}{\sfJ}^2 & \leq (eT)^2 \SemiNorm{\uh}{\sfJ}^2.
\end{align}
\end{subequations}

By using the triangle inequality, estimates~\eqref{EQN::APPROX-I-PI-2} and~\eqref{EQN::APPROX-I-PI-3}, bound~\eqref{EQN::BOUNDS-VARPHI-1} for~$\varphi$,
and estimate~\eqref{EQN::ESTIMATE-VH-AUX}, we obtain
\begin{equation}
\label{EQN::ESTIMATE-VH-3}
\SemiNorm{\Pi_r^t(\varphi \uh)}{\sfJ}^2 \leq 2\SemiNorm{(\Id - \Pi_r^t)(\varphi \uh)}{\sfJ}^2 + 2\SemiNorm{\varphi \uh}{\sfJ}^2 \leq 2C_S^2 \tau\Norm{\uh}{L^2(\QT)}^2 + 2(eT)^2 \SemiNorm{\uh}{\sfJ}^2.
\end{equation}

We now bound the SUPG-seminorm of~$\Pi_r^t (\varphi \uh)$. For all~$K \in \Th$ and~$n = 1, \ldots, N$, we have
\begin{equation}
\label{EQN::AUX-SUPG}
\begin{split}
\SemiNorm{\Pi_r^t(\varphi \uh)}{\Knsupg}^2 & = \lambdaKn \Norm{\tLhtK \big( \Pi_r^t(\varphi \uh)\big) }{L^2(\Kn)}^2 + \lambdaKn \betaKn^2 \Norm{\Pi_r^t\big(\varphi \nabla (\Id - \PiN{k}{K}) \uh\big)}{L^2(\Kn)^d}^2.
\end{split}
\end{equation}
We consider the first term on the right-hand side of~\eqref{EQN::AUX-SUPG}. Using the triangle inequality,
the orthogonality properties of~$\PiO{k}{K}$ and~$\PPiO{k-1}{K}$,
estimates~\eqref{EQN::APPROX-I-PI-1} and~\eqref{EQN::APPROX-I-PI-Time-Der},
the inverse estimate~\eqref{EQN::INVERSE-ESTIMATE-SPACE},
bounds~\eqref{EQN::BOUNDS-VARPHI-1} and~\eqref{EQN::BOUNDS-VARPHI-2}, assumption~\eqref{EQN::CHOICE-LAMBDA} on the choice of~$\lambdaKn$, and the mild condition~\eqref{EQN::TAU-H}, we get
\begin{align}
\nonumber
\lambdaKn & \Norm{\tLhtK\big(\Pi_r^t(\varphi \uh)\big)}{L^2(\Kn)}^2\\
\nonumber
& \leq 2 \lambdaKn \Norm{\tLhtK (\varphi \uh)}{L^2(\Kn)}^2 + 2\lambdaKn\Norm{\tLhtK (\Id - \Pi_r^t)(\varphi \uh)}{L^2(\Kn)}^2 \\
\nonumber
& \leq 2 \lambdaKn \Norm{\dpt(\varphi \PiO{k}{K} \uh) + \varphi \bbeta \cdot \PPiO{k-1}{K} \nabla \uh}{L^2(\Kn)}^2 + 4\lambdaKn\Norm{\dpt (\Id - \Pi_r^t)(\varphi \PiO{k}{K} \uh)}{L^2(\Kn)}^2 \\
\nonumber
& \quad + 4 \lambdaKn\Norm{\bbeta \cdot (\Id - \Pi_r^t)(\varphi \PPiO{k-1}{K} \nabla \uh)}{L^2(\Kn)}^2 \\
\nonumber
& \leq 4 \lambdaKn \Norm{\varphi' \PiO{k}{K} \uh}{L^2(\Kn)}^2 + 4\lambdaKn \Norm{\varphi \tLhtK \uh}{L^2(\Kn)}^2 + 4\lambdaKn \Norm{\dpt (\Id - \Pi_r^t)(\varphi \PiO{k}{K} \uh)}{L^2(\Kn)}^2 \\
\nonumber
& \quad + 4 \lambdaKn\Norm{\bbeta \cdot (\Id - \Pi_r^t)(\varphi \PPiO{k-1}{K} \nabla \uh)}{L^2(\Kn)}^2 \\
\nonumber
& \leq 4(1 + C_S^2) \lambdaKn \Norm{\uh}{L^2(\Kn)}^2 + 4 (eT)^2 \lambdaKn \Norm{\tLhtK \uh}{L^2(\Kn)}^2 + 4 \betaQT^2 C_S^2 \lambdaKn \tau_n^2 \Norm{\nabla \uh}{L^2(\Kn)^d}^2 \\
\nonumber
& \leq 4 \zeta (1 + C_S^2)\betaQT^{-1} \hK \Norm{\uh}{L^2(\Kn)}^2 + 4(eT)^2 \lambdaKn \Norm{\tLhtK \uh}{L^2(\Kn)}^2 + 4 \betaQT C_S^2 \Cinv^2 \zeta \tau_n^2 \hK^{-1} \Norm{\uh}{L^2(\Kn)}^2 \\
\label{EQN::AUX-SUPG-FIRST-TERM}
& \leq 4\zeta \big((1 + C_S^2) \betaQT^{-1} \hK + \delta^2 \betaQT C_*^2 \big) \Norm{\uh}{L^2(\Kn)}^2 + 4(eT)^2 \lambdaKn \Norm{\tLhtK \uh}{L^2(\Kn)}^2,
\end{align}
where~$\delta = C_S \Cinv$.

Moreover, using the stability properties of~$\Pi_r^t$ and bound~$\eqref{EQN::BOUNDS-VARPHI-1}$, the second term on the right-hand side of~\eqref{EQN::AUX-SUPG} can be bounded as follows:
\begin{equation*}
\lambdaKn \betaKn^2 \Norm{\Pi_r^t\big(\varphi \nabla (\Id - \PiN{k}{K}) \uh\big)}{L^2(\Kn)^d}^2 \leq (eT)^2 \lambdaKn \betaKn^2 \Norm{\nabla (\Id - \PiN{k}{K}) \uh}{L^2(\Kn)^d}^2,
\end{equation*}
which, combined with~\eqref{EQN::AUX-SUPG-FIRST-TERM} and summing up for all~$n$, gives
\begin{equation}
\label{EQN::ESTIMATE-VH-4}
\SemiNorm{\Pi_r^t(\varphi \uh)}{\supg}^2 \leq 4(eT)^2 \SemiNorm{\uh}{\supg}^2 + 4\zeta \big((1 + C_S^2) \betaQT^{-1} h + \delta^2 \betaQT C_*^2 \big)\Norm{\uh}{L^2(\QT)}^2.
\end{equation}

Therefore, combining bound~\eqref{EQN::TRIANGLE-INEQ-VH} with estimates~\eqref{EQN::ESTIMATE-VH-1}, \eqref{EQN::ESTIMATE-VH-2}, \eqref{EQN::ESTIMATE-VH-3}, and~\eqref{EQN::ESTIMATE-VH-4}, we get the desired result.
\end{proof}

\begin{lemma}
\label{LEMMA::BOUND-AT}
Under the notation and assumptions of Lemma~\ref{LEMMA::BOUND-VH-UH}, the following bound holds:
\begin{equation*}
\nu \aht(\uh, \vh) \geq \nu \alpha_* (T + \theta) \Norm{\nabla \uh}{L^2(\QT)^d}^2,
\end{equation*}
where~$\alpha_*$ is the stability constant in~\eqref{EQN::STAB-aK}.
\end{lemma}
\begin{proof}
Let~$\uh \in \Vht$. By the stability property in~\eqref{EQN::STAB-aK} of the bilinear form~$\ah(\cdot, \cdot)$, and bound~\eqref{EQN::BOUNDS-VARPHI-1} {for~$\varphi$}, we have
\begin{align}
\nonumber
\nu \aht(\uh, \vh) & = \nu \aht(\uh, \varphi \uh) - \nu \aht(\uh, (\Id - \Pi_r^t) (\varphi \uh)) + \theta \nu \aht(\uh, \uh) \\
\label{EQN::AUX-BOUND-AT}
& \geq (T + \theta) \alpha_* \nu \Norm{\nabla \uh}{L^2(\QT)^d}^2 - \nu \aht(\uh, (\Id - \Pi_r^t)(\varphi \uh)).
\end{align}
The last term on the right-hand side of~\eqref{EQN::AUX-BOUND-AT} vanishes due to the orthogonality properties of~$\Pi_r^t$ and Lemma~\ref{LEMMA::ORTHOGONALITY-STAB}, which completes the proof.
\end{proof}

\begin{lemma}
\label{LEMMA::BOUND-BT}
Under the notation and assumptions of Lemma~\ref{LEMMA::BOUND-VH-UH}, the following bound holds:
\begin{equation}
\label{EQN::BOUND-BT}
\bht(\uh, \vh) \geq - \delta_t C_*^2 \Norm{\uh}{L^2(\QT)}^2,
\end{equation}
with~$\delta_t = \delta C_0 \Norm{\dpt \bbeta}{L^{\infty}(\QT)^d}$, and~$C_*$ as in the mild condition~\eqref{EQN::TAU-H}.
\end{lemma}
\begin{proof}
Let~$\uh \in \Vht$. Using the skew symmetry of the bilinear form~$\bht(\cdot, \cdot)$, and the commutativity of~$\PPiO{k}{K}$ and~$\Pi_r^t$,
we have
\begin{align}
\nonumber
\bht(\uh, \vh) & = \bht(\uh, \varphi \uh) - \bht(\uh, (\Id - \Pi_r^t) (\varphi \uh)) + \theta \bht(\uh, \uh) \\
\nonumber
& = -\bht(\uh, (\Id - \Pi_r^t)(\varphi \uh)) \\
\nonumber
& = -\frac12 \sum_{n = 1}^N \sum_{K \in \Th}
\Big[
\big(\bbeta \cdot \PPiO{k}{K} \nabla \uh, \ (\Id - \Pi_r^t) (\varphi \PiO{k}{K} \uh\big)_{\Kn} \\
\label{EQN::AUX-IDENTITY-bh-1}
& \qquad \qquad \quad - \big((\Id - \Pi_r^t) (\varphi \PPiO{k}{K} \nabla \uh),\ \bbeta \PiO{k}{K} \uh \big)_{\Kn}
\Big].
\end{align}
In particular, if~$\bbeta = \bbeta(\bx)$, then~$\bht(\uh, \vh) = 0$. Otherwise, since~$\Pi_0^t (\bbeta) \cdot \PPiO{k}{K} \nabla \uh \in \Pp{r}(\Tt; L^2(\Omega))$ and~$\Pi_0^t (\bbeta) \PiO{k}{K} \uh\in \Pp{r}(\Tt; L^2(\Omega)^d)$, using identity~\eqref{EQN::AUX-IDENTITY-bh-1} and the orthogonality properties of~$\Pi_r^t$, we obtain
\begin{equation}
\label{EQN::AUX-IDENTITY-bh-2}
\begin{split}
\bht(\uh, \vh) & = -\frac12 \sum_{n = 1}^N \sum_{K \in \Th}
\Big[
\big((\Id - \Pi_0^t)\bbeta \cdot \PPiO{k}{K} \nabla \uh, \ (\Id - \Pi_r^t) (\varphi \PiO{k}{K} \uh)\big)_{\Kn} \\
& \qquad \qquad \quad - \big((\Id - \Pi_r^t) (\varphi \PPiO{k}{K} \nabla \uh),\ (\Id - \Pi_0^t)\bbeta \PiO{k}{K} \uh\big)_{\Kn}
\Big].
\end{split}
\end{equation}

For any~$K \in \Th$ and~$n = 1, \ldots, N$, using the stability properties of~$\PPiO{k}{K}$, the local inverse estimate in~\eqref{EQN::INVERSE-ESTIMATE-SPACE}, estimate~\eqref{EQN::APPROX-I-PI-1}, error bound~\eqref{EQN::BETA-ESTIMATE} {for~$\bbeta$}, and the H\"older inequality, we get
\begin{align}
\nonumber
-\big((\Id - \Pi_0^t) & \bbeta \cdot \PPiO{k}{K} \nabla \uh, \ (\Id - \Pi_r^t) (\varphi \PiO{k}{K} \uh)\big)_{\Kn} \\
\nonumber
& \geq -\Norm{\bbeta - \Pi_0^t \bbeta}{L^{\infty}(\Kn)^d} \Norm{(\Id - \Pi_r^t)(\varphi \PiO{k}{K} \uh) \PPiO{k}{K} \nabla \uh}{L^1(\Kn)^d} \\
\nonumber
& \geq - \Norm{\bbeta - \Pi_0^t \bbeta}{L^{\infty}(\Kn)^d} \Norm{(\Id - \Pi_r^t)(\varphi \PiO{k}{K} \uh)}{L^2(\Kn)} \Norm{\PPiO{k}{K} \nabla \uh}{L^2(\Kn)^d} \\
\label{EQN::AUX-BOUND-bh}
& \geq -\delta_t \hK^{-1} \tau_n^2 \Norm{\uh}{L^2(\Kn)}^2,
\end{align}
where~$\delta_t = \delta C_0 \Norm{\dpt \bbeta}{L^{\infty}(\QT)^d}$.

The following bound can be proven in a similar way:
\begin{equation*}
-\big((\Id - \Pi_r^t) (\varphi \PPiO{k}{K} \nabla \uh),\ (\Id - \Pi_0^t)\bbeta \PiO{k}{K} \uh\big)_{\Kn} \geq -\delta_t \hK^{-1} \tau_n^2 \Norm{\uh}{L^2(\Kn)}^2,
\end{equation*}
which, combined with identity~\eqref{EQN::AUX-IDENTITY-bh-2}, the mild condition~\eqref{EQN::TAU-H} and bound~\eqref{EQN::AUX-BOUND-bh}, leads to~\eqref{EQN::BOUND-BT}.
\end{proof}

\begin{lemma}
\label{LEMMA::BOUND-MT}
Under the assumptions of Lemma~\ref{LEMMA::BOUND-VH-UH}, the following bound holds:
\begin{equation}
\label{EQN::BOUND-MT}
\mht(\uh, \vh) \geq \mu_* \bigg(\Big(\frac12
- \Big(\frac{\mu^*}{\mu_*}\Big)^2 \frac{C_S^2 \tau}{2\theta}\Big)\Norm{\uh}{L^2(\QT)}^2 + T \SemiNorm{\uh}{\sfJ}^2 \bigg),
\end{equation}
where~$\mu_*$ and~$\mu^*$ are the stability constants in~\eqref{EQN::STAB-smK}, and $C_S$ is the constant from Lemma~\ref{LEMMA::APPROXIMATION-I-PI}.
\end{lemma}
\begin{proof}
Let~$\uh \in \Vht$. By adding and subtracting suitable terms, we have
\begin{align}
\nonumber
\mht(\uh, \vh)
& = \mht(\uh, \varphi \uh) - \mht(\uh, (\Id - \Pi_r^t) (\varphi \uh)) + \theta \mht(\uh, \uh) \\
\label{EQN::IDENTITY-MT}
& =: M_1 + M_2 + M_3.
\end{align}
We treat each term~$M_i$, $i = 1, 2, 3$, separately.

\paragraph{Bound for~$M_1$.} Since~$\varphi'(t) = -(1/T) \varphi(t)$ for all~$t \in [0, T]$, the following identity follows:
\begin{equation*}
\varphi \dpt(\uh^2) = (1/T) \uh^2 \varphi + \dpt(\varphi \uh^2),
\end{equation*}
which, together with the identity~$\frac12 \jump{\varphi w^2}_n - \jump{\varphi w}_n {w(\cdot, \tn^+)} = \frac12 \jump{\sqrt{\varphi} w}_n^2$, the symmetry and stability of the bilinear form~$\mh(\cdot, \cdot)$, and bound~\eqref{EQN::BOUNDS-VARPHI-1} {for~$\varphi$}, leads to
\begin{align}
\nonumber
M_1 & = \mht(\uh, \varphi \uh) \\
\nonumber
& = \frac12 \sum_{n = 1}^N \int_{I_n} \varphi(t) \frac{\d}{\dt}\mh\big(\uh, \uh\big) \dt - \sum_{n = 1}^{N - 1} \mh(\jump{\uh}_n, \varphi(\tn) \uh(\cdot, \tn^+)) \\
\nonumber
& \quad
+ \mh(\uh(\cdot, 0), \varphi(0) \uh(\cdot, 0)) \\
\nonumber
& = \frac{1}{2T} \sum_{n = 1}^N \int_{I_n} \varphi(t) \mh\big(\uh, \uh\big) \dt + \frac12\sum_{n = 1}^N \int_{I_n} \frac{\d}{\dt} \Big(\varphi(t) \mh\big(\uh, \uh\big) \Big)\dt \\
\nonumber
& \quad - \sum_{n = 1}^{N - 1} \mh\big(\jump{\uh}_n, \varphi(\tn) \uh(\cdot, \tn^+)\big)
+ \mh\big(\uh(\cdot, 0), \varphi(0) \uh(\cdot, 0)\big) \\
\nonumber
& \geq \frac{\mu_*}{2T} \Norm{\sqrt{\varphi} \uh}{L^2(\QT)}^2 + \frac12 \Big[\mh\big(\varphi(T)\uh(\cdot, T), \uh(\cdot, T)\big) \\
\nonumber
& \quad + \sum_{n = 1}^{N - 1} \jump{\varphi(\tn) \mh\big(\uh, \uh\big)}_n - \mh\big(\varphi(0)\uh(\cdot, 0), \uh(\cdot, 0)\big)\Big] \\
\nonumber
& \quad  - \sum_{n = 1}^{N - 1} \mh\big(\jump{\uh}_n, \varphi(\tn) \uh(\cdot, \tn^+)\big)
+ \mh\big(\varphi(0) \uh(\cdot, 0), \uh(\cdot, 0)\big)
\\
\nonumber
& = \frac{\mu_*}{2T} \Norm{\sqrt{\varphi} \uh}{L^2(\QT)}^2 + \frac12 \Big[\mh\big(\varphi(T)\uh(\cdot, T), \uh(\cdot, T)\big) \\
\nonumber
& \quad + \sum_{n = 1}^{N - 1}  \mh(\varphi(\tn) \jump{\uh}_n, \jump{\uh}_n) + \mh\big(\varphi(0)\uh(\cdot, 0), \uh(\cdot, 0)\big)\Big]
\\
\nonumber
& \geq \mu_* \Big(\frac{1}{2T} \Norm{\sqrt{\varphi} \uh}{L^2(\QT)}^2 + \SemiNorm{\sqrt{\varphi} \uh}{\sfJ}^2\Big)
\\
& \geq \mu_* \Big( \frac12 \Norm{\uh}{L^2(\QT)}^2 + T \SemiNorm{\uh}{\sfJ}^2\Big).
\label{EQN::BOUND-M1}
\end{align}

\paragraph{Bound for~$M_2$.} Using the orthogonality properties of~$\Pi_r^t$, Lemma~\ref{LEMMA::ORTHOGONALITY-STAB}, estimate~\eqref{EQN::APPROX-I-PI-2}, the local inverse estimate in~\eqref{EQN::INVERSE-ESTIMATE-TIME}, the symmetry and stability of the bilinear form~${\mh(\cdot, \cdot)}$,
and the Young inequality, we obtain
\begin{align}
\nonumber
M_2 & = -\mht(\uh, (\Id - \Pi_r^t) (\varphi \uh)) \\
\nonumber
& = -\sum_{n = 1}^N \sum_{K \in \Th} \Big[\big(\dpt \uh, (\Id - \Pi_r^t)(\varphi \uh)\big)_{\Kn} \\
\nonumber
& \quad  + \int_{I_n} \smK\Big((\Id - \PiO{k}{K}) \dpt \uh, (\Id - \Pi_r^t) \big(\varphi (\Id - \PiO{k}{K}) \uh\big) \Big) \dt \Big] \\
\nonumber
& \quad + \sum_{n = 1}^{N - 1} \mh\big(\jump{\uh}_n, (\Id - \Pi_r^t) (\varphi \uh)(\cdot, \tn^+)\big) - \mh\big(\uh(\cdot, 0), (\Id - \Pi_r^t)(\varphi \uh)(\cdot, 0)\big) \\
\nonumber
& = \sum_{n = 1}^{N - 1} \mh\big(\jump{\uh}_n, (\Id - \Pi_r^t) (\varphi \uh)(\cdot, \tn^+)\big) - \mh\big(\uh(\cdot, 0), (\Id - \Pi_r^t)(\varphi \uh)(\cdot, 0)\big) \\
\nonumber
& \geq - \mu^* \Big(\sum_{n = 1}^{N - 1} \Norm{\jump{\uh}_n}{L^2(\Omega)} \Norm{(\Id - \Pi_r^t)(\varphi \uh)}{L^2(\Sn)}  + \Norm{\uh}{L^2(\SO)} \Norm{(\Id - \Pi_r^t)(\varphi \uh)}{L^2(\SO)}\Big) \\
\label{EQN::BOUND-M2}
& \geq - \mu^* \bigg(
\theta \frac{\mu_*}{\mu^*}\SemiNorm{\uh}{\sfJ}^2 + \frac{C_S^2 \tau}{2 \theta} \cdot \frac{\mu^*}{\mu_*} \Norm{\uh}{L^2(\QT)}^2 \bigg).
\end{align}

\paragraph{Bound for~$M_3$.} Integration by parts {in time}, the symmetry and stability of the bilinear form~$\mht(\cdot, \cdot)$, and the identity~$\frac12 \jump{w^2}_n - \jump{w}_n
{w(\cdot, \tn^+)} = \frac12 \jump{w}_n^2$ yield
\begin{align}
\nonumber
M_3 & = \theta \mht(\uh, \uh) \\
\nonumber
 & =  \frac{\theta}{2} \sum_{n = 1}^N \int_{I_n} \frac{\d}{\dt} \mh(\uh, \uh) \dt - \theta \sum_{n = 1}^{N - 1} \mh\big(\jump{\uh}_n, \uh(\cdot, \tn^+)\big) + \theta \mh\big(\uh(\cdot, 0), \uh(\cdot, 0)\big)\\
\nonumber
 & = \frac{\theta}{2} \Big(\mh\big(\uh(\cdot, T), \uh(\cdot, T)\big) + \sum_{n = 1}^{N - 1} \jump{\mh\big(\uh, \uh\big)}_n - \mh\big(\uh(\cdot, 0), \uh(\cdot, 0)\big)\Big) \\
\nonumber
 & \quad - \theta \sum_{n = 1}^{N - 1} \mh\big(\jump{\uh}_n, \uh(\cdot, \tn^+)\big) + \theta \mh\big(\uh(\cdot, 0), \uh(\cdot, 0)\big)\\
\nonumber
& = \frac{\theta}{2} \Big( \mh\big(\uh(\cdot, T), \uh(\cdot, T)\big) + \sum_{n = 1}^{N-1} \mh\big(\jump{\uh}_n, \jump{\uh}_n\big) + \mh\big(\uh(\cdot, 0), \uh(\cdot, 0)\big)\Big)\\
\label{EQN::BOUND-M3}
& \geq \theta \mu_* \SemiNorm{\uh}{\sfJ}^2.
\end{align}
Combining identity~\eqref{EQN::IDENTITY-MT} with bounds~\eqref{EQN::BOUND-M1}, \eqref{EQN::BOUND-M2}, and~\eqref{EQN::BOUND-M3}, we get~\eqref{EQN::BOUND-MT}.
\end{proof}

\begin{lemma}
\label{LEMMA::BOUND-ST}
Under the assumptions of Lemma~\ref{LEMMA::BOUND-VH-UH}, the following bound holds:
\begin{align}
\nonumber
\sht(\uh, \vh) & \geq \frac{\alpha_* T}{2} \SemiNorm{\uh}{\supg}^2  - \zeta \Big[\frac{h}{\betaQT}\big(\frac{1}{T} + \frac{1}{2\theta}\big) + \frac{1}{\alpha_* \theta} \big(C_S^2 \betaQT^{-1} h + \delta^2 \betaQT C_*^2 \big)\Big] \Norm{\uh}{L^2(\QT)}^2 \\
\label{EQN::BOUND-ST}
& \quad - \zeta\Big[e^2T + \bigg(\frac{\alpha_*}{2} + 1 \bigg)\theta\Big] \nu \Norm{\nabla \uh}{L^2(\QT)^d}^2,
\end{align}
with~$\delta = C_S \Cinv$, $\alpha_*$ is the stability constant in~\eqref{EQN::STAB-saK}, and~$C_*$ is the constant in the mild condition~\eqref{EQN::TAU-H}.
\end{lemma}
\begin{proof}
Let~$\uh \in \Vht$. By adding and subtracting suitable terms, we have
\begin{align}
\nonumber
\sht(\uh, \vh) & = \sht(\uh, \varphi \uh) - \sht(\uh, (\Id - \Pi_r^t) (\varphi \uh)) + \theta\sht(\uh, \uh) \\
\label{EQN::AUX-SUPG-IDENTITY}
& = : {J_1 + J_2 + J_3}.
\end{align}
We bound each term~${J}_i$, $i = 1, 2, 3$, separately.

\paragraph{Bound for~${J_1}$.} Using the triangle and the Young inequalities, the stability property in~\eqref{EQN::STAB-saK} of~$\saK(\cdot, \cdot)$, the inverse estimates in~\eqref{EQN::INVERSE-ESTIMATE-TIME} and~\eqref{EQN::INVERSE-ESTIMATE-DIVERGENCE}, bounds~\eqref{EQN::BOUNDS-VARPHI-1} and~\eqref{EQN::BOUNDS-VARPHI-2} {for~$\varphi$},
the stability properties of~$\PiO{k}{K}$ and~$\PPiO{k-1}{K}$, and assumption~\eqref{EQN::CHOICE-LAMBDA} on the choice of~$\lambdaKn$, we get
\begin{align}
\nonumber
{J_1} & = \sht(\uh, \varphi \uh) \\
\nonumber
& = \sum_{n = 1}^N \sum_{K \in \Th} \lambdaKn \Big[
\Big(\LhtK \uh, \tLhtK (\varphi \uh) \Big)_{\Kn} + \betaKn^2 \int_{I_n} \varphi \saK\big((\Id - \PiN{k}{K}) \uh, (\Id - \PiN{k}{K}) \uh\big) \dt \Big] \\
\nonumber
& = \sum_{n = 1}^N \sum_{K \in \Th} \Big[
\lambdaKn \Big(\LhtK \uh, \dpt (\varphi \PiO{k}{K} \uh) + \varphi \bbeta \cdot \PiO{k-1}{K} \nabla \uh \Big)_{\Kn} \\
\nonumber
& \quad + \betaKn^2 \lambdaKn \int_{I_n} \varphi \saK\big((\Id - \PiN{k}{K}) \uh, (\Id - \PiN{k}{K}) \uh\big) \dt \Big] \\
\nonumber
& {\geq} \sum_{n = 1}^N \sum_{K \in \Th}  \Big[
\lambdaKn \big(\tLhtK \uh, \varphi' \PiO{k}{K} \uh)_{\Kn} + \lambdaKn \Norm{\sqrt{\varphi} \tLhtK \uh}{L^2(\Kn)}^2 \\
\nonumber
& \quad + \check{\alpha} \betaKn^2 \lambdaKn
\Norm{\sqrt{\varphi} \nabla (\Id - \PiN{k}{K}) \uh}{L^2(\Kn)^d}^2  - \nu \lambdaKn \big(\div \PPiO{k-1}{K} \nabla \uh, \varphi' \PiO{k}{K} \uh + \varphi \tLhtK \uh \big)_{\Kn} \Big] \\
\nonumber
& \geq \sum_{n = 1}^N \sum_{K \in \Th} \Big[
\frac{T}{2} \lambdaKn \Norm{\tLhtK \uh}{L^2(\Kn)}^2 + \check{\alpha} T \lambdaKn \betaKn^2 \Norm{\nabla (\Id - \PiN{k}{K}) \uh}{L^2(\Kn)^d}^2  \\
\nonumber
& \quad - \lambdaKn\big(\frac{1}{T}+ \frac{1}{2\theta}\big) \Norm{\varphi' \PiO{k}{K} \uh}{L^2(\Kn)}^2 -\nu^2 \lambdaKn \Big(\frac{\theta}{2} + e^2T\Big) \Norm{\div \PPiO{k-1}{K} \nabla \uh}{L^2(\Kn)}^2 \Big] \\
\nonumber
& \geq \sum_{n = 1}^N \sum_{K \in \Th} \Big[
\frac{T}{2} \lambdaKn \Norm{\tLhtK \uh}{L^2(\Kn)}^2 + (\check{\alpha} T) \lambdaKn \betaKn^2 \Norm{\nabla (\Id - \PiN{k}{K}) \uh}{L^2(\Kn)^d}^2  \\
\nonumber
& \quad - \lambdaKn\big(\frac{1}{T} + \frac{1}{2\theta}\big) \Norm{\varphi' \PiO{k}{K} \uh}{L^2(\Kn)}^2 - \nu^2 \lambdaKn \Big(\frac{\theta}{2} + e^2T\Big) \Cinv^2 \hK^{-2} \Norm{\PPiO{k-1}{K} \nabla \uh}{L^2(\Kn)}^2 \Big] \\
\label{EQN::BOUND-T1}
& \geq \frac{T\alpha_*}{2} \SemiNorm{\uh}{\supg}^2
-  \frac{\zeta h}{\betaQT}\big(\frac{1}{T} + \frac{1}{2\theta}\big)
\Norm{\uh}{L^2(\QT)}^2 - \nu \zeta \Big(\frac{\theta}{2} + e^2 T\Big) \Norm{\nabla \uh}{L^2(\QT)^d}^2.
\end{align}

\paragraph{Bound for~${J_2}$.} The following identity follows from Lemma~\ref{LEMMA::ORTHOGONALITY-STAB}
\begin{align}
\nonumber
{J_2} & = -\sht(\uh, (\Id - \Pi_r^t)(\varphi \uh)) \\
\nonumber
& = -\sum_{n = 1}^N \sum_{K \in \Th}\Big[ \lambdaKn \big( \LhtK \uh, \tLhtK (\Id - \Pi_r^t)(\varphi \uh) \big)_{\Kn} \\
\nonumber
& \quad + \betaKn^2 \lambdaKn \int_{I_n} \saK\big((\Id - \PiN{k}{K}) \uh, (\Id - \Pi_r^t)(\varphi (\Id - \PiN{k}{K}) \uh)\big) \dt\Big] \\
\nonumber
& = -\sum_{n = 1}^N \sum_{K \in \Th} \lambdaKn \big( \tLhtK \uh - \nu \div \PPiO{k-1}{K} \nabla \uh, \tLhtK (\Id - \Pi_r^t)(\varphi \uh) \big)_{\Kn}.
\end{align}
Using the Cauchy--Schwarz, the triangle, and the Young inequalities, estimates~\eqref{EQN::APPROX-I-PI-1} and~\eqref{EQN::APPROX-I-PI-Time-Der},
the local inverse estimates in Lemma~\ref{LEMMA::INVERSE-ESTIMATE}, assumption~\eqref{EQN::CHOICE-LAMBDA} on the choice of~$\lambdaKn$, and condition~\eqref{EQN::TAU-H}, we get
\begin{align}
\nonumber
{J_2} & \geq - \sum_{n = 1}^N \sum_{K \in \Th} \bigg[\frac{\lambdaKn \alpha_*\theta}{2} \Norm{\tLhtK \uh}{L^2(\Kn)}^2 + \frac{\nu^2 \lambdaKn \alpha_* \theta}{2} \Norm{\div \PPiO{k-1}{K} \nabla \uh}{L^2(\Kn)}^2 \\
\nonumber
& \quad + \frac{2\lambdaKn}{\theta \alpha_*} \Big(\Norm{\dpt (\Id - \Pi_r^t)\Pi_{k}^{0,K}(\varphi \uh)}{L^2(\Kn)}^2  + \Norm{\bbeta \cdot (\Id - \Pi_r^t)(\varphi \PPiO{k-1}{K} \nabla \uh}{L^2(\Kn)}^2 \Big)\bigg] \\
\nonumber
& \geq - \sum_{n = 1}^N \sum_{K \in \Th} \bigg[\frac{\lambdaKn \alpha_*\theta}{2} \Norm{\tLhtK \uh}{L^2(\Kn)}^2 + \frac{\nu^2 \lambdaKn \alpha_* \theta}{2} \Cinv^2 \hK^{-2} \Norm{\nabla \uh}{L^2(\Kn)^d}^2 \\
\nonumber
& \quad + \frac{2\lambdaKn}{\theta \alpha_*} \Big(C_S^2 + \delta^2 \betaQT^2 \tau_n^2 \hK^{-2} \Big) \Norm{\uh}{L^2(\Kn)}^2 \bigg] \\
\nonumber
& \geq - \sum_{n = 1}^N \sum_{K \in \Th} \bigg[\frac{\lambdaKn \alpha_*\theta}{2} \Norm{\tLhtK \uh}{L^2(\Kn)}^2 + \frac{\nu \zeta \alpha_* \theta}{2} \Norm{\nabla \uh}{L^2(\Kn)^d}^2 \\
\nonumber
& \quad + \frac{2\zeta}{\theta \alpha_*} \Big(C_S^2 \betaQT^{-1} \hK + \delta^2 \betaQT \tau_n^2 \hK^{-1} \Big) \Norm{\uh}{L^2(\Kn)}^2 \bigg] \\
\label{EQN::BOUND-T2}
& \geq - \frac{\alpha_* \theta}{2} \SemiNorm{\uh}{\supg}^2 - \frac{\nu \zeta \alpha_* \theta}{2} \Norm{\nabla \uh}{L^2(\QT)^d}^2 - \frac{2\zeta}{\alpha_*\theta} \big(C_S^2\betaQT^{-1} h + \delta^2 \betaQT C_*^2\big) \Norm{\uh}{L^2(\QT)}^2,
\end{align}
with~$\delta = C_S \Cinv$.

\paragraph{Bound for~${J_3}$.} Similarly as for the bound for~${J_1}$, it can be shown that
\begin{equation}
\label{EQN::BOUND-T3}
{J_3} \geq \frac{\alpha_* \theta}{2} \SemiNorm{\uh}{\supg}^2 - \frac{\nu \zeta \theta}{2} \Norm{\nabla \uh}{L^2(\QT)^d}^2.
\end{equation}
Combining identity~\eqref{EQN::AUX-SUPG-IDENTITY} with bounds~\eqref{EQN::BOUND-T1}, \eqref{EQN::BOUND-T2}, and~\eqref{EQN::BOUND-T3}, we get~\eqref{EQN::BOUND-ST}
\end{proof}

We are now in a position to prove the stabiliy and well-posedness of the method.

\medskip\noindent
{\bf Proof of Theorem \ref{THM::FULLY-DISCRETE-STABILITY}.}
Let~$\uh \in \Vht$ and~$\vh := \Pi_r^t(\varphi \uh) + \theta \uh \in \Vht$, with~$\varphi$ defined in~\eqref{DEF::VARPHI} and the real~$\theta > 0$ to be fixed later.
From Lemmas~\ref{LEMMA::BOUND-AT}, \ref{LEMMA::BOUND-BT}, \ref{LEMMA::BOUND-MT}, and~\ref{LEMMA::BOUND-ST}, we get
\begin{align*}
\nonumber
\Bht(\uh, \vh) & \geq
\bigg[\mu_* \Big(\frac12 - \Big(\frac{\mu^*}{\mu_*}\Big)^2 \frac{C_S^2 \tau}{2\theta} \Big)
- \delta_t C_*^2  \\
\nonumber
& \qquad - \zeta \Big( \frac{h}{\betaQT}\Big(\frac{1}{T} + \frac{1}{2\theta}\Big) + \frac{2}{\theta \alpha_*} \big(C_S^2 \betaQT^{-1} h + \delta^2 \betaQT C_*^2 \big)\Big)\bigg] \Norm{\uh}{L^2(\QT)}^2
\\
\nonumber
& \quad + \mu_* T \SemiNorm{\uh}{\sfJ}^2 + \frac{\alpha_* T}{2}\SemiNorm{\uh}{\supg}^2 \\
& \quad + \bigg[
T \Big(\alpha_* - \zeta e^2 \Big)
+ \theta \Big(\alpha_* - \zeta\big(\frac{\alpha_*}{2} + 1\big)\Big) \bigg]
\nu \Norm{\nabla \uh}{L^2(\QT)^d}^2.
\end{align*}
Choosing~$\zeta$ and~$C_*$ small enough, and~$\theta$ large enough, we deduce that there exists a constant~$\eta^* > 0$ independent of~$h$, $\tau$, and~$\nu$ such that
\begin{equation}
\label{EQN::BOUND-Bht-lower}
\Bht(\uh, \vh) \geq \eta^* \TNorm{\uh}^2.
\end{equation}
which, combined with Lemma~\ref{LEMMA::BOUND-VH-UH}, completes the proof.

\medskip\noindent
{\bf Proof of Corollary \ref{COROL:1}.}
The result follows from the inf-sup stability estimate in Theorem~\ref{THM::FULLY-DISCRETE-STABILITY} and the following continuity bound for the linear functional~$\ellht(\cdot)$ in~\eqref{DEF::lht}: for all~$\vh \in \Vht$, it holds
\begin{equation*}
\begin{split}
|\ellht(\vh)| & \leq \sqrt{2} \big(\Norm{u_0}{L^2(\Omega)} + \Norm{f}{L^2(\QT)}\big) \Big(\Norm{\vh}{L^2(\QT)}^2  + \frac12 \Norm{\vh}{L^2(\SO)}^2 \\
& \quad + \sum_{n = 1}^N \sum_{K \in \Th} \lambdaKn^2 \Norm{\tLhtK \vh}{L^2(\Kn)}^2 \Big)^{\frac12} \leq \sqrt{2} (\Norm{u_0}{L^2(\Omega)} + \Norm{f}{L^2(\QT)}) \TNorm{\vh}.
\end{split}
\end{equation*}
\smallskip

\begin{remark}[Inf-sup constant for long-time simulations] The inf-sup stability constant~$\gamma_I$ is of the form
$$
\gamma_I = \frac{\eta_*}{\eta^*},
$$
where~$\eta^*$ is the constant in bound~\eqref{EQN::BOUND-Bht-lower} and~$\eta_*$ is an upper bound for the constant on the right-hand side of~\eqref{EQN::BOUND-VH-UH}.

Assuming that~$T \gg 1$ and
$$
\tau < \frac{1}{C_S^2} \Big(\frac{\mu^*}{\mu_*}\Big)^2, \quad  h < \betaQT \min\Big\{\frac{\alpha_*}{2 C_S^2},\ 1\Big\},  \quad C_*^2 < \min\Big\{\frac{\mu_*}{8 \delta_t},\ \frac{\alpha_*}{2 \delta^2 \betaQT} \Big\},
$$
$$
\zeta < \min\Big\{\frac{\alpha_* }{2e^2},\  \frac{\alpha_*}{\alpha_* + 2},\ \frac{\mu^* T}{32}\Big\}, \quad \theta = \max\Big\{4,\ \frac{T}{2}\Big\},
$$
we have that
$$\eta^* \geq \min\Big\{\frac{\mu_*}{8},\ \mu_* T,\ \alpha_* T\Big\} = \frac{\mu_*}{8} \quad \text{ and } \quad \eta_* = \mathcal{O}(T).$$
Therefore, the inf-sup stability constant $\gamma_I \sim T$.
\eremk
\end{remark}

\section{Convergence analysis \label{SECT::CONVERGENCE}}

In the present section we prove Theorem \ref{THM::ERROR-ESTIMATES}, which will follow as a simplified case of the more general error bounds derived below.
We start by recalling some polynomial and VE approximation results. The Bramble-Hilbert lemma~(see e.g., \cite[Lemma 4.3.8]{Brenner-Scott:book}) implies the following approximation properties of the polynomial projections introduced in Section~\ref{SUBSECT::LOCAL-SPACES-PROJECTIONS}.

\begin{lemma}[Estimates for polynomial projections]
\label{LEMMA::BRAMBLE-HILBERT}
Under Assumption~\ref{A1}, for any element~$K \in \Omegah$ and any sufficiently smooth function~$\phi$ defined on~$K$, the following estimates hold:
\begin{subequations}
\begin{align}
\label{EQN::ESTIMATE-L2-SPACE}
\Norm{(\Id - \PiO{k}{K})\phi }{W^{m, p}(K)} & \lesssim \hK^{s - m} \SemiNorm{\phi}{W^{s, p}(K)} & & s, m \in \N, \ m \le s \le k + 1, \ p \in [1, +\infty],\\
\label{EQN::ESTIMATE-PI-NABLA}
\Norm{(\Id - \PiN{k}{K}) \phi}{H^m(K)} & \lesssim \hK^{s - m}|\phi|_{H^s(K)} && s, m \in \N,\ m \le s \le k + 1, \ s \geq 1.
\end{align}
\end{subequations}

Moreover, for any~$I_n \in \Tt$ and for any sufficiently smooth function~$\psi$ defined on~$I_n$, the~$L^2(\Tt)$-orthogonal projection in~$\Pp{r}(\Tt)$ of~$\psi$ satisfies
\begin{equation}
\label{EQN::ESTIMATE-L2-TIME}
\Norm{(\Id - \Pi_r^t)\psi}{H^{m}(I_n)} \lesssim \tau_n^{\ell - m} \SemiNorm{\psi}{H^{\ell}(I_n)} \qquad \ell, m \in \N,\ m \le \ell \le r + 1.
\end{equation}
\end{lemma}

The next lemma concerns the optimal approximation properties of the VE space~$\Vh$, see~\cite[Lemma~3.15 (for~2D) and \S5.2 (for~3D)]{Brenner-smallfaces} and \cite[Thm.~11]{Cangiani_Georgoulis_Pryer_Sutton:2017} for more details; see instead \cite{SERE-2} for the case of serendipity VE spaces.
\begin{lemma}[Approximation by VE functions]
\label{LEMMA::VE-APPROX}
Under Assumption~\ref{A1}, for any~$v \in H_0^1(\Omega) \cap H^{s+1}(\Omegah)$ $(0 < s \le k + 1)$, there exists~$\vI \in \Vh$ such that, for all~$K \in \Omegah$, it holds
\begin{equation}
\label{EQN::VE-INTERPOLATION}
\Norm{v - \vI}{L^2(K)} + \hK \Norm{\nabla(v - \vI)}{L^2(K)^d} \lesssim \hK^{s+1}|v|_{H^{s+1}(K)}.
\end{equation}
\end{lemma}

The last ingredients are a standard trace inequality in one dimension, a stability bound for~$\Pi_r^t$ in the~$H^1(I_n)$-seminorm, and a scaled Poincar\'e--Friedrichs inequality on polytopes.
\begin{lemma}[Trace inequality]
\label{LEMMA::TRACE-INEQ}
Let~$I_n \in \Tt$. For any~$\psi \in H^1(I_n)$, it holds
\begin{equation}
\label{EQN::TRACE-INEQ}
|\psi(\tnmo)|^2 + |\psi(\tn)|^2 \lesssim \tau_n^{-1} \Norm{\psi}{L^2(I_n)}^2 + \tau_n \Norm{\dpt \psi}{L^2(I_n)}^2.
\end{equation}
\end{lemma}

\begin{lemma}[Stability of~$\Pi_r^t$]
\label{LEMMA::STAB-PI-TIME-H1}
Let~$I_n \in \Tt$. For any~$\psi \in H^1(I_n)$, the following bound holds:
\begin{equation*}
\Norm{\dpt \Pi_r^t \psi}{L^2(I_n)} \lesssim \Norm{\dpt \psi}{L^2(I_n)}.
\end{equation*}
\end{lemma}
\begin{proof}
Let~$I_n \in \Tt$ and~$\psi \in H^1(I_n)$. Using the polynomial inverse estimate~\eqref{EQN::INVERSE-ESTIMATE-TIME}, the stability properties of~$\Pi_r^t$, and the standard Poincar\'e inequality, we get
\begin{equation*}
\begin{split}
\Norm{\dpt \Pi_r^t \psi}{L^2(I_n)} & = \Norm{\dpt (\Pi_r^t \psi - \Pi_0^t\psi)}{L^2(I_n)} \lesssim \tau_n^{-1} \Norm{\Pi_r^t (\psi - \Pi_0^t \psi)}{L^2(I_n)} \\
& \lesssim \tau_n^{-1}\Norm{(\Id - \Pi_0^t)\psi}{L^2(I_n)} \lesssim \Norm{\dpt \psi}{L^2(I_n)},
\end{split}
\end{equation*}
which completes the proof.
\end{proof}

\begin{lemma}[Scaled Poincar\'e--Friedrichs inequality (see e.g., {\cite[Lemma 2.2]{Chen_Huang:2018}})]
\label{LEMMA::POINCARE-Friedrichs}
Under Assumption~\ref{ASM::MESH-REGULARITY}, for any element~$K \in \Th$ and~$v \in H^1(K)$, it holds
\begin{equation*}
\Norm{(\Id - \PiN{k}{K}) v}{L^2(K)} \lesssim \hK \Norm{\nabla (\Id - \PiN{k}{K}) v}{L^2(K)^d}.
\end{equation*}
\end{lemma}

In the forthcoming convergence analysis, estimates~\eqref{EQN::ESTIMATE-L2-SPACE}, \eqref{EQN::ESTIMATE-PI-NABLA}, and~\eqref{EQN::VE-INTERPOLATION}, as well as the scaled Poincar\'e--Friedrichs inequality in Lemma~\ref{LEMMA::POINCARE-Friedrichs}, are applied pointwise in time, whereas estimate~\eqref{EQN::ESTIMATE-L2-TIME} and the trace inequality~\eqref{EQN::TRACE-INEQ} are applied pointwise in space.

\subsection{Some preliminary assumptions and notations}
Henceforth, we assume that the solution~$u$ to the continuous weak formulation~\eqref{EQN::CONTINUOUS-WEAK-FORMULATION} has the following parabolic regularity:
\begin{equation}\label{EQN::PARABOLIC-REGULARITY}
u \in L^2(0, T; H^2(\Omega)) \cap L^{\infty}(0, T; H_0^1(\Omega)) \cap H^1(0, T; L^2(\Omega)).
\end{equation}
Consequently, a density argument can be used to show that
\begin{equation}
\label{EQN::SUPG-IDENTITY}
(\dpt u - \nu \Delta u + \bbeta \cdot \nabla u, v)_{\QT} = (f, v)_{\QT} \qquad \forall v \in L^2(\QT).
\end{equation}
For convenience, we also define the following operators:
\begin{alignat}{2}
\label{DEF::LK}
\LK u & := \dpt u - \nu \Delta u + \bbeta \cdot \nabla u, & \qquad \tLK v & := \dpt v + \bbeta \cdot \nabla v,
\end{alignat}
and the following bilinear forms and linear functionals:
\begin{alignat*}{3}
m^{\dpt}(u, v) & := (\dpt u, v)_{\QT} + (u, v)_{\SO}, & \qquad a^T(u, v) & := \int_0^T a(u, v) \dt, \\
b^T(u, v) & := \int_0^T b(u, v) \dt, & \qquad s^{\supg}(u, v) & := \sum_{K \in \Th} \sum_{n = 1}^N \lambdaKn (\LK u, \tLK v)_{\Kn}, \\
\ell_f(v) & := (f, v)_{\QT}, & \ell_{f}^{\supg}(v) & := \sum_{K \in \Omegah} \sum_{n = 1}^N (f, \lambdaKn \tLK v)_{\Kn}, \\
\ell_{u_0}(v) & := (u_0, v)_{\Omega}.
\end{alignat*}

\begin{remark}[Parabolic regularity]
The parabolic regularity assumption~\eqref{EQN::PARABOLIC-REGULARITY} holds, for instance, if~$\Omega$ is convex, compatibility of the initial and boundary conditions is satisfied (i.e., $u_0 \in H_0^1(\Omega)$), and~$\bbeta$ does not depend on~$t$ (see~\cite[Thm.~5 in Ch.~7.1]{Evans:2022} for smooth domains, which can be extended to convex domains using the results in~\cite[Ch.~3]{Grisvard:2011})
\eremk
\end{remark}

\subsection{\emph{A priori} error estimates\label{SUBSECT::A-PRIORI-ERROR}}
{Let
the solution~$u$ to the {continuous} weak formulation~\eqref{EQN::CONTINUOUS-WEAK-FORMULATION} satisfy the parabolic regularity~\eqref{EQN::PARABOLIC-REGULARITY}},
and let~$\uh \in \Vht$ be the solution to the {SUPG-stabilized time-DG VEM} formulation~\eqref{EQN::FULLY-DISCRETE-SCHEME}. We define~$\uIt := \Pi_r^t \uI$, and the following error functions:
\begin{equation*}
\eIt = u - \uIt, \quad \ePiO = u - \gPiO{k} u, \quad \ePiN = u - \gPiN{k} u, \quad \eh = \uIt - \uh.
\end{equation*}

\begin{proposition}[\emph{A priori} error bounds]
\label{PROP::A-PRIORI-BOUNDS}
Let Assumptions~\ref{ASM::MESH-REGULARITY} and~\ref{TheAssumption} hold, then the following bound holds:
\begin{equation}
\label{EQN::A-PRIORI-BOUND}
\begin{split}
\TNorm{u - \uh} \le \TNorm{\eIt} + \gamma_I \Bigg(& \sup_{\substack{\vh \in \Vht \setminus \{0\}, \\
\TNorm{\vh} = 1}} \chil(\vh) + \sup_{\substack{\vh \in \Vht \setminus \{0\}, \\
\TNorm{\vh} = 1}} \chisupg(\vh)  \\
& + \sup_{\substack{\vh \in \Vht \setminus \{0\}, \\
\TNorm{\vh} = 1}} \chia(\vh) + \sup_{\substack{\vh \in \Vht \setminus \{0\}, \\
\TNorm{\vh} = 1}} \chimb(\vh) \Bigg),
\end{split}
\end{equation}
where
\begin{align*}
\chil(\vh) & :=
\ell_f(\vh) - \ell_f(\PiO{k}{K} \vh) + \ell_{u_0}(\vh) - \ell_{u_0}(\PiO{k}{K} \vh), \\
\chisupg(\vh) & :=
\ell_f^{\supg}(\vh) - \sht(\uIt, \vh),\\
\chia(\vh) & := \nu a^T(u, \vh) - \nu \aht(\uIt, \vh), \\
\chimb(\vh) & := m^{\dpt}(u, \vh) + b^T(u, \vh) - \mht(\uIt, \vh) - \bht(\uIt, \vh).\\
\end{align*}
\end{proposition}
\begin{proof}
The result follows easily from the triangle inequality, the inf-sup stability estimate in Theorem~\ref{THM::FULLY-DISCRETE-STABILITY}, the definition of the {SUPG-stabilized time-DG VEM} in~\eqref{EQN::FULLY-DISCRETE-SCHEME}, and observing that the continuos solution satisfies 
$$
m^{\dpt}(u, \vh) + \nu a^T(u, \vh) + b^T(u, \vh) = \ell_f(\vh) + \ell_{u_0}(\vh)
\quad \forall \vh \in \Vht \, .
$$
\end{proof}

We now estimate each term on the right-hand side of~\eqref{EQN::A-PRIORI-BOUND}.

\begin{lemma}[Estimate of~$\TNorm{\eIt}$]
\label{LEMMA::INTERPOLATION-ERROR-ESTIMATE}
Under Assumption~\ref{ASM::MESH-REGULARITY} on the mesh-regularity, Assumption~\ref{TheAssumption} on the choice of~$\lambdaKn$ and the relation of~$\tau$ and~$h_{\min}$, and Assumption~\ref{ASM::DATA} on the data of the problem, it holds
\begin{align}
\nonumber
\TNorm{\eIt} & \lesssim \sum_{n = 1}^N \sum_{K \in \Th} \bigg( \big(\tau_n^{2q + 1} + \lambdaKn \tau_n^{2q}\big)\Norm{u}{H^{q + 1}(I_n; L^2(K))}^2 + \lambdaKn \hK^{2s + 2} \Norm{\dpt u}{L^2(I_n; H^{s+1}(K))}^2\\
\nonumber
& \quad + \Big(\hK^{2s + 2} + \frac{\hK^{2s + 2}}{\tau_n} + \nu \hK^{2s} + \lambdaKn \betaKn^2 \hK^{2s}\Big) \Norm{u}{L^2(I_n; H^{s+1}(K))}^2 \\
\label{EQN::INTERPOLATION-ESTIMATE}
& \quad + \big(\nu + \lambdaKn \betaKn^2 \big) \tau_n^{2q + 2} \Norm{\nabla u}{H^{q + 1}(I_n; L^2(K)^d)}^2\bigg).
\end{align}
\end{lemma}
\begin{proof}
By the definition of the energy norm~$\TNorm{\cdot}$ in~\eqref{EQN::ENERGY-NORM}, we have
\begin{equation}
\label{EQN::SPLIT-INTERPOLATION-ERROR}
\TNorm{\eIt}^2 = \Norm{\eIt}{L^2(\QT)}^2 + \SemiNorm{\eIt}{\sfJ}^2 + \nu \Norm{\nabla \eIt}{L^2(\QT)^d}^2 + \SemiNorm{\eIt}{\supg}^2.
\end{equation}
We treat each term on the right-hand side of~\eqref{EQN::SPLIT-INTERPOLATION-ERROR} separately.

\paragraph{$\bullet$~Estimate of~$\Norm{\eIt}{L^2(\QT)}$.} Using the triangle inequality, the stability of~$\Pi_r^t$ and its approximation properties in Lemma~\ref{LEMMA::BRAMBLE-HILBERT}, and the VE interpolation estimate in Lemma~\ref{LEMMA::VE-APPROX}, for all~$K \in \Omegah$ and~$n \in \{1, \ldots, N\}$, we get
\begin{equation*}
\begin{split}
\Norm{\eIt}{L^2(\Kn)}^2 & \leq 2\Norm{u - \Pi_r^t u}{L^2(\Kn)}^2 + 2\Norm{\Pi_r^t(u - \uI)}{L^2(\Kn)}^2 \\
& \lesssim \tau_n^{2q + 2} \Norm{u}{H^{q + 1}(I_n; L^2(K))}^2 + \hK^{2s + 2} \Norm{u}{L^2(I_n; H^{s + 1}(K))}^2.
\end{split}
\end{equation*}

\paragraph{$\bullet$~Estimate of~$\SemiNorm{\eIt}{\sfJ}$.} Using the trace inequality in Lemma~\ref{LEMMA::TRACE-INEQ}, the stability of~$\Pi_r^t$ and its approximation properties in Lemma~\ref{LEMMA::BRAMBLE-HILBERT}, the inverse estimate in~\eqref{EQN::INVERSE-ESTIMATE-TIME}, and the VE interpolation estimate in Lemma~\ref{LEMMA::VE-APPROX}, we obtain
\begin{align*}
\nonumber
\SemiNorm{\eIt}{\sfJ}^2
& \lesssim \sum_{n = 1}^N \Big( \tau_n^{-1} \Norm{\eIt}{L^2(I_n; L^2(\Omega))}^2 + \tau_n \Norm{\dpt \eIt}{L^2(I_n; L^2(\Omega))}^2\Big) \\
\nonumber
& \lesssim \sum_{n = 1}^{N} \Big(\tau_n^{-1} \big(\Norm{u - \Pi_r^t u}{L^2(I_n; L^2(\Omega))}^2 + \Norm{\Pi_r^t(u - \uI)}{L^2(I_n; L^2(\Omega))}^2 \big)  \\
\nonumber
& \qquad + \tau_n \big(\Norm{\dpt(u - \Pi_r^t u)}{L^2(I_n; L^2(\Omega))}^2 + \Norm{\dpt(\Pi_r^t(u - \uI))}{L^2(I_n; L^2(\Omega))}^2 \big)\Big) \\
& \lesssim \sum_{n = 1}^N \Big(\tau_n^{2q + 1} \Norm{u}{H^{q + 1}(I_n; L^2(\Omega))}^2 + \sum_{K \in \Omegah} \frac{\hK^{2s + 2}}{\tau_n} \Norm{u}{L^2(I_n; H^{s+1}(K))}^2
\Big).
\end{align*}

\paragraph{$\bullet$~Estimate of~$\nu \Norm{\nabla \eIt}{L^2(\QT)^d}$.} Using the triangle inequality, the stability of~$\Pi_r^t$ and its approximation properties in Lemma~\ref{LEMMA::BRAMBLE-HILBERT}, the commutativity of the spatial gradient operator~$\nabla$ and the~$L^2(\Tt)$-orthogonal projection~$\Pi_r^t$, and the VE interpolation estimate in Lemma~\ref{LEMMA::VE-APPROX}, for all~$K \in \Omegah$ and~$n \in \{1, \ldots, N\}$, we have
\begin{align}
\nonumber
\nu \Norm{\nabla \eIt}{L^2(\Kn)^d}^2 & \le 2\nu \Norm{\nabla u - \Pi_r^t \nabla u}{L^2(\Kn)^d}^2 + 2\nu \Norm{\Pi_r^t \nabla (u - \uI)}{L^2(\Kn)^d}^2 \\
\nonumber
& \lesssim \nu \tau_n^{2q + 2}\Norm{\nabla u}{H^{q + 1}(I_n; L^2(K)^d)}^2 + \nu \hK^{2s} \Norm{u}{L^2(I_n; H^{s+1}(K))}^2.
\end{align}

\paragraph{$\bullet$~Estimate of~$\SemiNorm{\eIt}{\supg}$.}
We bound this term using the triangle inequality, the commutativity of the first-order time derivative operator~$\dpt$ and the~$L^2(K)$-orthogonal projection operator~$\PiO{k}{K}$, the stability properties of~$\PiO{k}{K}$, $\PPiO{k-1}{K}$, $\PiN{k}{K}$, and~$\Pi_r^t$, the stability bound in Lemma~\ref{LEMMA::STAB-PI-TIME-H1}, the estimates for~$\Pi_r^t$ in Lemma~\ref{LEMMA::BRAMBLE-HILBERT}, and the VE interpolation estimate in Lemma~\ref{LEMMA::VE-APPROX}, as follows:
\begin{align*}
\nonumber
\SemiNorm{\eIt}{\Knsupg}^2 & = \lambdaKn \Norm{\dpt \PiO{k}{K} \eIt + \bbeta \cdot \PPiO{k-1}{K} \nabla \eIt}{L^2(\Kn)}^2 + \lambdaKn \betaKn^2 \Norm{\nabla (\Id - \PiN{k}{K}) \eIt}{L^2(\Kn)^d}^2 \\
\nonumber
& \lesssim \lambdaKn \Norm{\dpt \eIt}{L^2(\Kn)}^2 + \lambdaKn \betaKn^2 \Norm{\nabla \eIt}{L^2(\Kn)^d}^2 \\
\nonumber
& \lesssim \lambdaKn \big(\Norm{\dpt (\Id - \Pi_r^t)u}{L^2(\Kn)}^2 +  \Norm{\dpt \Pi_r^t (u - \uI)}{L^2(\Kn)}^2 \big) \\
\nonumber
& \quad + \lambdaKn \betaKn^2 \big(\Norm{(\Id - \Pi_r^t) \nabla u}{L^2(\Kn)^d}^2 + \Norm{\Pi_r^t \nabla (u - \uI)}{L^2(\Kn)^d}^2 \big) \\
\nonumber
& \lesssim \lambdaKn \big(\Norm{\dpt (\Id - \Pi_r^t)u}{L^2(\Kn)}^2 +  \Norm{\dpt (u - \uI)}{L^2(\Kn)}^2 \big) \\
\nonumber
& \quad + \lambdaKn \betaKn^2 \big(\Norm{(\Id - \Pi_r^t) \nabla u}{L^2(\Kn)^d}^2 + \Norm{ \nabla (u - \uI)}{L^2(\Kn)^d}^2 \big) \\
\nonumber
& \lesssim \lambdaKn \big(\tau_n^{2q} \Norm{u}{H^{q + 1}(I_n; L^2(K))}^2 + \hK^{2s + 2} \Norm{\dpt u}{L^2(I_n; H^{s + 1}(K))}^2 \big) \\
& \quad + \lambdaKn \betaKn^2 \big(\tau_n^{2q + 2} \Norm{\nabla u}{H^{q + 1}(I_n; L^2(K)^d)}^2 + \hK^{2s} \Norm{u}{L^2(I_n; H^{s+1}(K))}^2 \big).
\end{align*}

\paragraph{Conclusion.} Estimate~\eqref{EQN::INTERPOLATION-ESTIMATE} follows combining~\eqref{EQN::SPLIT-INTERPOLATION-ERROR} with the above four estimates.
\end{proof}

\begin{lemma}[Estimate of~$\chil(\vh)$]
\label{LEMMA::ESTIMATE-CHIL}
Under Assumption~\ref{ASM::MESH-REGULARITY} on the mesh-regularity, Assumption~\ref{TheAssumption} on the choice of~$\lambdaKn$ and the relation of~$\tau$ and~$h_{\min}$, and Assumption~\ref{ASM::DATA} on the data of the problem, for all~$\vh \in \Vht$ with~$\TNorm{\vh} = 1$, the term~$\chil(\vh)$ can be bounded as follows:
\begin{equation*}
\begin{split}
\chil(\vh) & \lesssim \bigg[\sum_{n = 1}^N \sum_{K \in \Omegah}
\min\Big\{\frac{1}{\lambdaKn \betaKn^2}, \frac{1}{\nu} \Big\} \hK^{2s + 4}
\Norm{f}{L^2(I_n; H^{s+1}(K))}^2
+ \sum_{K \in \Omegah} \hK^{2s
+ 2} \SemiNorm{u_0}{H^{s+1}(K)}^2\bigg]^{\frac12}.
\end{split}
\end{equation*}
\end{lemma}
\begin{proof}
Using the orthogonality properties of~$\PiO{k}{K}$ and~$\PPiO{k-1}{K}$, we have
\begin{align}
\nonumber
\chil(\vh)
& =
\ell_f(\vh) - \ell_f(\PiO{k}{K} \vh) + \ell_{u_0}(\vh) - \ell_{u_0}(\PiO{k}{K} \vh) \\
\nonumber
& = \sum_{n = 1}^N \sum_{K \in \Omegah} \big((\Id - \PiO{k}{K})f, (\Id - \PiN{k}{K})\vh)_{\Kn}
+ \sum_{K \in \Omegah} ( {(\Id - \PiO{k}{K}) u_0}, \vh(\cdot, 0))_{K} \\
\label{EQN::IDENTITY-CHIL}
& =: \sum_{n = 1}^N \sum_{K \in \Omegah}
 \chiFone^{\Kn}
 + \sum_{K \in \Omegah} \chiUo^K.
\end{align}

Using the scaled Poincar\'e--Friedrichs inequality in Lemma~\ref{LEMMA::POINCARE-Friedrichs}, the approximation properties of~$\PiO{k}{K}$ from Lemma~\ref{LEMMA::BRAMBLE-HILBERT}, and the definition of~$\SemiNorm{\cdot}{\Knsupg}$ in~\eqref{EQN::LOCAL-SUPG-NORM}, for all~$K \in \Omegah$ and~$n \in \{1, \ldots, N\}$, we get
\begin{align*}
\chiFone^{\Kn} & = \big((\Id - \PiO{k}{K}) f, (\Id - \PiN{k}{K})\vh\big)_{\Kn}
\lesssim \hK \Norm{(\Id - \PiO{k}{K}) f}{L^2(\Kn)} \Norm{\nabla (\Id - \PiN{k}{K}) \vh}{L^2(\Kn)^d} \\
& \lesssim \hK^{s+2} \Norm{f}{L^2(I_n; H^{s+1}(K))} \min\bigg\{\frac{\SemiNorm{\vh}{\Knsupg}}{\lambdaKn^{1/2} \betaKn}, \Norm{\nabla \vh}{L^2(\Kn)^d}\bigg\},
\end{align*}
which, combined with the Cauchy--Schwarz inequality and the definition of the energy norm~$\TNorm{\cdot}$ in~\eqref{EQN::ENERGY-NORM}, implies
\begin{equation}
\label{EQN::CHI-F1}
\sum_{n = 1}^N \sum_{K \in \Omegah} \chiFone^{\Kn} \lesssim
\Big(\sum_{n = 1}^N \sum_{K \in \Omegah} \hK^{2s+4} \min\Big\{\frac{1}{\lambdaKn \betaKn^2}, \frac{1}{\nu}\Big\}\Norm{f}{L^2(I_n; H^{s+1}(K))}^2 \Big)^{\frac12} \TNorm{\vh}.
\end{equation}

Finally, using estimate~\eqref{EQN::ESTIMATE-L2-SPACE} for~$\PiO{k}{K}$ and the definition of the jump functional~$\SemiNorm{\cdot}{\sfJ}$ in~\eqref{EQN::JUMP-FUNCTIONAL}, we obtain
\begin{equation}
\label{EQN::CHI-U0}
\sum_{K \in \Omegah} \chiUo^{K} \lesssim  \Big(\sum_{K \in \Omegah} \hK^{2s + 2} \SemiNorm{u_0}{H^{s + 1}(K)}^2 \Big)^{\frac12} \Norm{\vh}{L^2(\SO)} \lesssim \Big( \sum_{K \in \Omegah} \hK^{2s + 2} \SemiNorm{u_0}{H^{s + 1}(K)}^2 \Big)^{\frac12} \SemiNorm{\vh}{\sfJ}.
\end{equation}

The desired result follows combining identity~\eqref{EQN::IDENTITY-CHIL} with estimates~\eqref{EQN::CHI-F1}
and~\eqref{EQN::CHI-U0}.
\end{proof}

\begin{lemma}[Estimate of~$\chisupg(\vh)$]
\label{LEMMA::ESTIMATE-CHISUPG}
Under Assumption~\ref{ASM::MESH-REGULARITY} on the mesh-regularity, Assumption~\ref{TheAssumption} on the choice of~$\lambdaKn$ and the relation of~$\tau$ and~$h_{\min}$, and Assumption~\ref{ASM::DATA} on the data of the problem, for all~$\vh \in \Vht$ with~$\TNorm{\vh} = 1$, the term~$\chisupg(\vh)$ can be bounded as follows:
\begin{align*}
\chisupg(\vh) & \lesssim \bigg[\sum_{n = 1}^N \sum_{K \in \Omegah} \bigg(\frac{\lambdaKn \tau_n^{2q}}{\betaKn^2} \Norm{u}{H^{q+1}(I_n; L^2(K))}^2 + \frac{\lambdaKn \hK^{2s+2}}{\betaKn^2} \Norm{\dpt u}{L^2(I_n; H^{s+1}(K))}^2 \\
\nonumber
& \quad + \Big(
\frac{\nu^2 \lambdaKn \hK^{2s-2}}{\betaKn^2} + \lambdaKn(1 + \betaKn^2) \hK^{2s}\Big) \Norm{u}{L^2(I_n; H^{s+1}(K))}^2 \\
& \quad + \Big(\frac{\nu^2 \lambdaKn \tau_n^{2q+2}}{\hK^2 \betaKn^2} + \lambdaKn \tau_n^{2q+2}\Big) \Norm{\nabla u}{H^{q+1}(I_n; L^2(K)^d)}^2 + \frac{\lambdaKn \hK^{2s + 4}}{\tau_n^2 \betaKn^2} \Norm{f}{L^2(I_n; H^{s+1}(K))}^2 \\
& \quad +  \lambdaKn^2 \hK^{2s} \min\Big\{\frac{1}{\lambdaKn \betaKn^2}, \frac{1}{\nu}\Big\}\Norm{\bbeta}{L^{\infty}(I_n; W^{s, \infty}(K)^d)}^2 \Norm{f}{L^2(I_n; H^{s}(K))}^2 \bigg)\bigg]^{\frac12}.
\end{align*}
\end{lemma}
\begin{proof}
Adding and subtracting suitable terms, we get
\begin{align}
\nonumber
\chisupg(\vh) & =
\ell_f^{\supg}(\vh) - \sht(\uIt, \vh) \\
\nonumber
& = \sum_{n = 1}^N \sum_{K \in \Omegah} \Big(\lambdaKn (f, \tLK \vh)_{\Kn} - \lambdaKn (\LhtK \uIt, \tLhtK \vh)_{\Kn} \\
\nonumber
& {\quad - \lambdaKn \betaKn^2 \int_{I_n} \saK((\Id - \PiN{k}{K}) \uIt, (\Id - \PiN{k}{K}) \vh) \dt\Big) }\\
\nonumber
& = \sum_{n = 1}^N \sum_{K \in \Omegah} \Big(\lambdaKn (\LK u, \tLK \vh)_{\Kn} - \lambdaKn (\LhtK \uIt, \tLhtK \vh)_{\Kn} \\
\nonumber
& \quad - \lambdaKn \betaKn^2 \int_{I_n} \saK((\Id - \PiN{k}{K}) \uIt, (\Id - \PiN{k}{K}) \vh) \dt\Big) \\
\nonumber
& = \sum_{n = 1}^N \sum_{K \in \Omegah} \Big( \lambdaKn \big(\LK u - \LhtK \uIt, \tLhtK \vh\big)_{\Kn} +  \lambdaKn(\LK u, \dpt (\Id - \PiO{k}{K}) \vh)_{\Kn} \\
\nonumber
& \quad  + \lambdaKn(\LK u \bbeta, (\Id - \PPiO{k-1}{K}) \nabla \vh )_{\Kn} \\
\nonumber
& \quad - \lambdaKn \betaKn^2 \int_{I_n} \saK((\Id - \PiN{k}{K}) \uIt, (\Id - \PiN{k}{K}) \vh) \dt
\Big) \\
\label{EQN::CHISUPG-SPLIT}
& =: \sum_{n = 1}^N \sum_{K \in \Omegah} \big(\chisupgone^{\Kn} + \chisupgtwo^{\Kn} + \chisupgthree^{\Kn} + \chisupgfour^{\Kn}\big).
\end{align}
We bound each term~$\chisupgi$, $i = 1, \ldots, 4$, separately.

Using the Cauchy--Schwarz and the triangle inequalities, and the definition of~$\SemiNorm{\cdot}{\Knsupg}$ in~\eqref{EQN::LOCAL-SUPG-NORM}, we get
\begin{align}
\nonumber
\chisupgone^{\Kn} & = \lambdaKn (\LK u - \LhtK \uIt, \tLhtK \vh)_{\Kn} \\
\nonumber
& = \lambdaKn \big(\dpt (u - \PiO{k}{K} \uIt) - \nu \div (\nabla u - \PPiO{k-1}{K} \nabla \uIt) + \bbeta \cdot (\nabla u - \PPiO{k-1}{K} \nabla \uIt), \tLhtK \vh\big)_{\Kn} \\
\nonumber
& \le \frac{\lambdaKn^{1/2}}{\betaKn} \Big(\Norm{\dpt(u - \PiO{k}{K} \uIt)}{L^2(\Kn)} + \nu \Norm{\div(\nabla u - \PPiO{k-1}{K} \nabla \uIt)}{L^2(\Kn)} \\
\label{EQN::SUPG-1-SPLIT}
& \qquad + \Norm{\bbeta \cdot (\nabla u - \PPiO{k-1}{K} \nabla \uIt)}{L^2(\Kn)}\Big) \SemiNorm{\vh}{\Knsupg}.
\end{align}
We now focus on the local interpolation error terms on the right-hand side of~\eqref{EQN::SUPG-1-SPLIT}.
Using the triangle inequality, the commutativity of~$\PiO{k}{K}$ with the first-order time derivative operator~$\dpt$ and that of~$\dpt$ with the VE interpolant operator, the estimates for~$\Pi_r^t$ and~$\PiO{k}{K}$ in Lemma~\ref{LEMMA::BRAMBLE-HILBERT}, the stability of~$\Pi_r^t$ in Lemma~\ref{LEMMA::STAB-PI-TIME-H1}, and the VE interpolation estimate in Lemma~\ref{LEMMA::VE-APPROX}, it follows that
\begin{align}
\nonumber
\Norm{\dpt(u - \PiO{k}{K} \uIt)}{L^2(\Kn)} & \le
\Norm{(\Id - \PiO{k}{K}) \dpt u}{L^2(\Kn)} + \Norm{\PiO{k}{K} \dpt(\Id - \Pi_r^t) u}{L^2(\Kn)} \\
\nonumber
& \quad + \Norm{\PiO{k}{K} \dpt \Pi_r^t(u - \uI)}{L^2(\Kn)} \\
\nonumber
& \lesssim \Norm{(\Id - \PiO{k}{K}) \dpt u}{L^2(\Kn)} + \Norm{\dpt(\Id - \Pi_r^t) u}{L^2(\Kn)} + \Norm{\dpt (u - \uI)}{L^2(\Kn)} \\
\label{EQN::SUPG-1-1}
& \lesssim \hK^{s+1} \Norm{\dpt u}{L^2(I_n; H^{s+1}(K))} + \tau_n^q \Norm{u}{H^{q+1}(I_n; L^2(K))}.
\end{align}
As for the second term on the right-hand side of~\eqref{EQN::SUPG-1-SPLIT}, we use the triangle inequality, the estimates for~$\Pi_r^t$ and~$\PPiO{k-1}{K}$ in Lemma~\ref{LEMMA::BRAMBLE-HILBERT} and their stability properties, the commutativity of~$\Pi_r^t$ and the spatial divergence operator~$\div$, and the inverse estimates in~\eqref{EQN::INVERSE-ESTIMATE-TIME} and~\eqref{EQN::INVERSE-ESTIMATE-DIVERGENCE} to obtain
\begin{align}
\nonumber
\nu & \Norm{\div(\nabla u - \PPiO{k-1}{K} \nabla \uIt)}{L^2(\Kn)} \\
\nonumber
& \ \le \nu \Norm{\div(\nabla u - \PPiO{k-1}{K} \nabla u)}{L^2(\Kn)} + \nu \Norm{\div \PPiO{k-1}{K}(\nabla u - \Pi_r^t \nabla u)}{L^2(\Kn)} \\
\nonumber
& \quad  + \nu \Norm{\Pi_r^t \div \PPiO{k-1}{K} \nabla (u - \uI)}{L^2(\Kn)} \\
\label{EQN::SUPG-1-2}
& \ \lesssim \nu \hK^{s-1} \Norm{u}{L^2(I_n; H^{s+1}(K))} + \nu \frac{\tau^{q+1}}{\hK} \Norm{\nabla u}{H^{q+1}(I_n; L^2(K)^d)}.
\end{align}
Note that the negative power of $h_K$ could be avoided by asking more space regularity for~$u$, but such effort is not required since this term will be balanced by suitable factors in the final estimates.
The last term on the right-hand side of~\eqref{EQN::SUPG-1-SPLIT} can be bounded similarly as follows:
\begin{align}
\nonumber
& \Norm{\bbeta\cdot (\nabla u - \PPiO{k-1}{K} \nabla \uIt)}{L^2(\Kn)} \\
\nonumber
& \qquad \le \Norm{\bbeta \cdot (\nabla u - \Pi_r^t \nabla u)}{L^2(\Kn)} + \Norm{\bbeta \cdot \Pi_r^t(\nabla u - \PPiO{k-1}{K} \nabla u)}{L^2(\Kn)} + \Norm{\bbeta \cdot \Pi_r^t \PPiO{k-1}{K} (\nabla u  - \nabla u_I)}{L^2(\Kn)} \\
\label{EQN::SUPG-1-3}
& \qquad \lesssim \betaKn \tau_n^{q+1}\Norm{\nabla u}{H^{q+1}(I_n; L^2(K)^d)} + \betaKn \hK^{s} \Norm{u}{L^2(I_n; H^{s+1}(K))}.
\end{align}

Combining estimates~\eqref{EQN::SUPG-1-1}, \eqref{EQN::SUPG-1-2}, and~\eqref{EQN::SUPG-1-3} with identity~\eqref{EQN::SUPG-1-SPLIT}, we get
\begin{align}
\nonumber
\sum_{n = 1}^N \sum_{K \in \Omegah} \chisupgone^{\Kn}
& \lesssim \bigg[\sum_{n = 1}^N \sum_{K \in \Omegah} \bigg(\frac{\lambdaKn \tau_n^{2q}}{\betaKn^2} \Norm{u}{H^{q+1}(I_n; L^2(K))}^2 + \frac{\lambdaKn \hK^{2s+2}}{\betaKn^2} \Norm{\dpt u}{L^2(I_n; H^{s+1}(K))}^2 \\
\nonumber
& \quad + \Big(
\frac{\nu^2 \lambdaKn \hK^{2s-2}}{\betaKn^2} + \lambdaKn \hK^{2s}\Big) \Norm{u}{L^2(I_n; H^{s+1}(K))}^2 \\
\label{EQN::ESTIMATE-SUPG-1}
& \quad + \Big(\frac{\nu^2 \lambdaKn \tau_n^{2q+2}}{\hK^2 \betaKn^2} + \lambdaKn \tau_n^{2q+2}\Big) \Norm{\nabla u}{H^{q+1}(I_n; L^2(K)^d)}^2\bigg)\bigg]^{\frac12}.
\end{align}

Due to identity~\eqref{EQN::SUPG-IDENTITY}, we get
\begin{alignat*}{3}
\chisupgtwo^{\Kn} & = \lambdaKn (\LK u, \dpt(\Id - \PiO{k}{K}) \vh)_{\Kn} = \lambdaKn((\Id - \PiO{k}{K}) f, (\Id - \PiO{k}{K}) \dpt \vh)_{\Kn}, \\
\chisupgthree^{\Kn} & = \lambdaKn (\LK u \bbeta, (\Id - \PPiO{k-1}{K}) \nabla \vh)_{\Kn} = \lambdaKn ((\Id - \PPiO{k-1}{K}) (f \bbeta), (\Id - \PPiO{k-1}{K}) \nabla \vh)_{\Kn}.
\end{alignat*}

Using the polynomial inverse estimate~\eqref{EQN::INVERSE-ESTIMATE-TIME}, the scaled Poincar\'e--Friedrichs inequality in Lemma~\ref{LEMMA::POINCARE-Friedrichs}, the approximation properties in Lemma~\ref{LEMMA::BRAMBLE-HILBERT} for~$\PiO{k}{K}$, and
the definition in~\eqref{EQN::LOCAL-SUPG-NORM} of~$\SemiNorm{\cdot}{\Knsupg}$, we have
\begin{align}
\nonumber
\sum_{n = 1}^N \sum_{K \in \Th} \chisupgtwo^{\Kn}
& = \sum_{n = 1}^N \sum_{K \in \Omegah} \lambdaKn\big((\Id - \PiO{k}{K})f, \dpt(\Id - {\PiO{k}{K}}) \vh\big)_{\Kn} \\
\nonumber
& \lesssim \sum_{n = 1}^N \sum_{K \in \Omegah} \frac{\lambdaKn}{\tau_n} \Norm{(\Id - \PiO{k}{K}) f}{L^2(\Kn)} \Norm{(\Id - \PiO{k}{K}) \vh}{L^2(\Kn)} \\
\nonumber
& \lesssim \sum_{n = 1}^N \sum_{K \in \Omegah} \frac{\lambdaKn \hK}{\tau_n} \Norm{(\Id - \PiO{k}{K}) f}{L^2(\Kn)} \Norm{\nabla (\Id - \PiO{k}{K}) \vh}{L^2(\Kn)^d}\\
\nonumber
& \lesssim \sum_{n = 1}^N \sum_{K \in \Th} \frac{\lambdaKn^{1/2} \hK}{\tau_n
\betaKn } \Norm{(\Id - \PiO{k}{K})f}{L^2(\Kn)} \SemiNorm{\vh}{\Knsupg} \\
\label{EQN::CHI-F2}
& \lesssim  \bigg(\sum_{n = 1}^N\sum_{K \in \Th} \frac{\lambdaKn \hK^{2s+4}}{\tau_n^2 \betaKn^2} \Norm{f}{L^2(I_n; H^{s+1}(K))}^2 \bigg)^{\frac12} \SemiNorm{\vh}{\supg}.
\end{align}

As for~$\chisupgthree^{\Kn}$, we use the Cauchy--Schwarz inequality, the approximation properties of~$\PPiO{k-1}{K}$, the fact that~$\nabla \PiN{k}{K} \vh \in \Pp{k-1}(K)^d$, the definition of~$\SemiNorm{\cdot}{\Knsupg}$ in~\eqref{EQN::LOCAL-SUPG-NORM}, and the H\"older inequality to get
\begin{align*}
\chisupgthree^{\Kn} & = \lambdaKn \big((\Id - \PPiO{k-1}{K})(f \bbeta), (\Id - \PPiO{k-1}{K}) \nabla \vh\big)_{\Kn} \\
& \lesssim \lambdaKn \Norm{(\Id - \PPiO{k}{K})(f \bbeta)}{L^2(\Kn)^d} \Norm{(\Id - \PPiO{k-1}{K}) \nabla \vh}{L^2(\Kn)^d} \\
& \lesssim
\lambdaKn \hK^s \Norm{\bbeta}{L^{\infty}(I_n; W^{s, \infty}(K)^d)} \Norm{f}{L^2(I_n; H^s(K))} \min\Big\{ \frac{\SemiNorm{\vh}{\Knsupg}}{\lambdaKn^{1/2} \betaKn}, \frac{\Norm{\nabla \vh}{L^2(\Kn)^d}}{\nu^{1/2}}\Big\}.
\end{align*}
The following estimate can be then obtained using the Cauchy--Schwarz inequality and the definition of the energy norm~$\TNorm{\cdot}$:
\begin{equation}
\label{EQN::CHI-F3}
\sum_{n = 1}^N \sum_{K \in \Omegah} \chisupgthree^{\Kn} \lesssim \Big( \sum_{n = 1}^N \sum_{K \in \Omegah} \lambdaKn^2 \hK^{2s} \min\Big\{\frac{1}{\lambdaKn \betaKn^2}, \frac{1}{\nu}\Big\}\Norm{\bbeta}{L^{\infty}(I_n; W^{s, \infty}(K)^d)}^2 \Norm{f}{L^2(I_n; H^{s}(K))}^2\Big)^{\frac12} \TNorm{\vh}.
\end{equation}

Finally, we estimate the term~$\chisupgfour^{\Kn}$. To do so, we use the stability bound of the bilinear form~$\saK(\cdot, \cdot)$ in~\eqref{EQN::STAB-saK}, the commutativity of~$\Pi_r^t$ and the spatial gradient operator~$\nabla$, the stability properties of~$\Pi_r^t$ and~$\PiN{k}{K}$, the triangle inequality, the definition of~$\SemiNorm{\cdot}{\Knsupg}$ in~\eqref{EQN::LOCAL-SUPG-NORM}, the estimate for~$\PiN{k}{K}$ in Lemma~\ref{LEMMA::BRAMBLE-HILBERT}, and the VE interpolation estimate in Lemma~\ref{LEMMA::VE-APPROX} to deduce
\begin{align*}
\chisupgfour^{\Kn} & = -\lambdaKn \betaKn^2 \int_{I_n} \saK((\Id - \PiN{k}{K}) \uIt, (\Id - \PiN{k}{K}) \vh) \dt \\
& \le \hat{\alpha} \lambdaKn \betaKn^2 \Norm{\Pi_r^t\nabla (\Id - \PiN{k}{K}) \uI}{L^2(\Kn)^d} \Norm{\nabla (\Id - \PiN{k}{K}) \vh}{L^2(\Kn)^d} \\
& \le \hat{\alpha} \lambdaKn^{1/2} \betaKn \big(
\Norm{\nabla (\uI - u)}{L^2(\Kn)^d} + \Norm{\nabla (\Id - \PiN{k}{K})u}{L^2(\Kn)^d} \\
& \quad + \Norm{\nabla \PiN{k}{K}(u - \uI)}{L^2(\Kn)^d}\big) \SemiNorm{\vh}{\Knsupg} \\
& \lesssim \lambdaKn^{1/2} \betaKn \hK^s \Norm{u}{L^2(I_n; H^{s+1}(K))}\SemiNorm{\vh}{\Knsupg}.
\end{align*}
The sum of the above local estimate over all the elements~$K \in \Omegah$ and the time steps~$n \in \{1, \ldots, N\}$ leads to
\begin{equation}
\label{EQN::ESTIMATE-SUPG-4}
\sum_{n = 1}^N \sum_{K \in \Omegah} \chisupgfour^{\Kn} \lesssim \bigg[\sum_{n = 1}^N \sum_{K \in \Omegah}  \bigg(\lambdaKn \betaKn^2 \hK^{2s} \Norm{u}{L^2(I_n; H^{s+1}(K))}^2 \bigg)\bigg]^{\frac12}.
\end{equation}

Combining identity~\eqref{EQN::CHISUPG-SPLIT} with estimates~\eqref{EQN::ESTIMATE-SUPG-1}, \eqref{EQN::CHI-F2}, \eqref{EQN::CHI-F3}, and~\eqref{EQN::ESTIMATE-SUPG-4}, the desired error bound is obtained.
\end{proof}

\begin{remark}[Alternative bound for $\chisupgtwo^{\Kn}$]\label{rem:alternative-1}
Whenever the mesh $\Omega_h$ is fixed and does not depend on the particular time interval (c.f. Remark \ref{rem:mesh-changing}) one could also handle $\chisupgtwo^{\Kn}$ in \eqref{EQN::CHI-F2} in a more efficient way, leading to a bound where the term $\tau_n^{-2}$ is substituted by the more favorable $\tau_n^{-1}$, provided that~$\dpt f$ is regular enough.
We detail the alternative bound here below.
Integration by parts in time, the commutativity of the first-order time derivative operator~$\dpt$ and~$\PiO{k}{K}$, and the regularity of~$f$ lead to the following identity:
\begin{align*}
\sum_{n = 1}^N \sum_{K \in \Th} \chisupgtwo^{\Kn}
& = \sum_{n = 1}^N \sum_{K \in \Omegah} \lambdaKn\big((\Id - \PiO{k}{K})f, \dpt(\Id - {\PiO{k}{K}}) \vh\big)_{0, \Kn} \\
& = - \sum_{n = 1}^N \sum_{K \in \Omegah} \lambdaKn \big((\Id - \PiO{k}{K}) \dpt f, (\Id - \PiO{k}{K}) \vh \big)_{0, \Kn} \\
& \quad + \sum_{K \in \Th} \lambdaKn \Big(\big((\Id - \PiO{k}{K}) f(\cdot, T), (\Id - \PiO{k}{K}) \vh(\cdot, T))_{0, K} \\
& \qquad \quad  + \sum_{n = 1}^{N - 1}\big((\Id - \PiO{k}{K}) f(\cdot, \tn), (\Id - \PiO{k}{K})\jump{\vh}_{n} \big)_{0, K} \\
& \qquad \quad - \big((\Id - \PiO{k}{K}) f(\cdot, 0), (\Id - \PiO{k}{K})\vh(\cdot, 0) \big)_{0, K} \Big),
\end{align*}
which, together with the Cauchy--Schwarz inequality, the trace inequality~\eqref{EQN::TRACE-INEQ}, the definition of the upwind-jump functional~$\SemiNorm{\cdot}{\sf J}$ in~\eqref{EQN::JUMP-FUNCTIONAL}, the stability properties of~$\PiO{k}{K}$, and its approximation properties in Lemma~\ref{LEMMA::BRAMBLE-HILBERT}, implies
\begin{align}
\nonumber
\sum_{n = 1}^N & \sum_{K \in \Omegah} \chisupgtwo^{\Kn} \\
\nonumber
& \lesssim \Big(\sum_{n = 1}^N \sum_{K \in \Th} \lambdaKn^2\Norm{(\Id - \PiO{k}{K})\dpt f}{L^2(\Kn)}^2 \Big)^{\frac12} \Norm{\vh}{L^2(\QT)} \\
\nonumber
& \quad + \bigg(\sum_{n = 1}^N\sum_{K \in \Th} \lambdaKn^2 \Big( \tau_n^{-1} \Norm{(\Id - \PiO{k}{K}) f}{L^2(\Kn)}^2 + \tau_n \Norm{(\Id - \PiO{k}{K}) \dpt f}{L^2(\Kn)}^2\Big) \bigg)^{\frac12} \SemiNorm{\vh}{\sf J} \\
\nonumber
& \lesssim \Big(\sum_{n = 1}^N \sum_{K \in \Th} \lambdaKn^2 \hK^{2s+2} \Norm{\dpt f}{L^2(I_n; H^{s+1}(K))}^2\Big)^{\frac12} \Norm{\vh}{L^2(\QT)} \\
\nonumber
& \quad + \bigg(\sum_{n = 1}^N\sum_{K \in \Th} \Big( \frac{\lambdaKn^2 \hK^{2s+2}}{\tau_n} \Norm{f}{L^2(I_n; H^{s+1}(K))}^2 + \tau_n \lambdaKn^2 \hK^{2s+2} \Norm{\dpt f}{L^2(I_n; H^{s+1}(K))}^2\Big) \bigg)^{\frac12} \SemiNorm{\vh}{\sf J}.
\end{align}
\eremk
\end{remark}

\begin{lemma}[Estimate of~$\chia(\vh)$]
\label{LEMMA::ESTIMATE-CHIA}
Under Assumption~\ref{ASM::MESH-REGULARITY} on the mesh-regularity, Assumption~\ref{TheAssumption} on the choice of~$\lambdaKn$ and the relation of~$\tau$ and~$h_{\min}$, and Assumption~\ref{ASM::DATA} on the data of the problem, for all~$\vh \in \Vht$ with~$\TNorm{\vh} = 1$, the term~$\chia(\vh)$ can be bounded as follows:
\begin{equation*}
\chia(\vh) \lesssim \bigg( \sum_{n = 1}^N \sum_{K \in \Omegah} \Big(\nu \hK^{2s} \Norm{u}{L^2(I_n; H^{s+1}(K))}^2 + \nu \tau_n^{2q + 2} \Norm{\nabla u}{H^{q+1}(I_n; L^2(K)^d)}^2 \Big)\bigg)^{\frac12}.
\end{equation*}
\end{lemma}
\begin{proof}
Using the polynomial consistency of the bilinear form~$\aht(\cdot, \cdot)$, the stability bound in~\eqref{EQN::STAB-aK}, the triangle inequality, the commutativity of~$\Pi_r^t$ and the spatial gradient operator~$\nabla$, the stability of~$\Pi_r^t$, the estimates for~$\PiN{k}{K}$ in Lemma~\ref{LEMMA::BRAMBLE-HILBERT} and for~$\uI$ in Lemma~\ref{LEMMA::VE-APPROX}, and the Cauchy--Schwarz inequality, we obtain
\begin{align*}
\chia(\vh) & = \nu a^T(u, \vh) - \nu \aht(\uIt, \vh) \\
& = \sum_{n = 1}^N \sum_{K \in \Th} \nu \big((\nabla (\Id - \PiN{k}{K}) u, \nabla \vh)_{\Kn} + \ahKn(\PiN{k}{K} u - \uIt, \vh) \big) \\
& \le \sum_{n = 1}^N \sum_{K \in \Omegah} \nu \big(\Norm{\nabla \ePiN}{L^2(\Kn)^d} + \alpha^* \Norm{\nabla(\PiN{k}{K} u  - \uIt)}{L^2(\Kn)^d}\big)\Norm{\nabla \vh}{L^2(\Kn)^d} \\
& \le \sum_{n = 1}^N \sum_{K \in \Omegah} \nu \big((1 + \alpha^*) \Norm{\nabla \ePiN}{L^2(\Kn)^d} + \alpha^* \Norm{(\Id - \Pi_r^t)\nabla u}{L^2(\Kn)^d} \\
& \qquad + \alpha^*\Norm{\Pi_r^t\nabla(u - \uI)}{L^2(\Kn)^d}\big)\Norm{\nabla \vh}{L^2(\Kn)^d} \\
& \lesssim \bigg( \sum_{n = 1}^N \sum_{K \in \Omegah} \Big(\nu \hK^{2s} \Norm{u}{L^2(I_n; H^{s+1}(K))}^2 + \nu \tau^{2q + 2} \Norm{\nabla u}{H^{q+1}(I_n; L^2(K)^d)}^2 \Big)\bigg)^{\frac12} \TNorm{\vh},
\end{align*}
which completes the proof.
\end{proof}

\begin{lemma}[Estimate of~$\chimb(\vh)$]
\label{LEMMA::ESTIMATE-CHIMB}
Under Assumption~\ref{ASM::MESH-REGULARITY} on the mesh-regularity, Assumption~\ref{TheAssumption} on the choice of~$\lambdaKn$ and the relation of~$\tau$ and~$h_{\min}$, and Assumption~\ref{ASM::DATA} on the data of the problem, for all~$\vh \in \Vht$ with~$\TNorm{\vh} = 1$, the term~$\chimb(\vh)$ can be bounded as follows:
\begin{align*}
\chimb & \lesssim \bigg(\sum_{n = 1}^N \sum_{K \in \Omegah} \Big(
\Big(\Norm{\bbeta}{L^{\infty}(I_n; W^{1, \infty}(K)^d)}^2 \hK^{2s+2} + \frac{\hK^{2s+2}}{\lambdaKn \betaKn^2} \Norm{\bbeta}{L^{\infty}(I_n; W^{s+1}(K)^d)}^2 \\
& \quad + \frac{\hK^{2s+2}}{\lambdaKn}
+ \hK^{2s+2} \min\Big\{\frac{1}{\lambdaKn \betaKn^2}, \frac{1}{\nu}\Big\} \Norm{\bbeta}{L^{\infty}(I_n; W^{s+1, \infty}(K)^d)}^2 \\
& \quad + \frac{\hK^{2s+4}}{\lambdaKn \betaKn^2 \tau_n^2}  + \frac{\hK^{2s+4}}{\tau_n^2} \min\Big\{\frac{1}{\lambdaKn \betaKn^2}, \frac{1}{\nu}\Big\}\Big)\Norm{u}{L^2(I_n; H^{s+1}(K))}^2
\\
& \quad  + \Big(\frac{\tau_n^{2q+2}}{\lambdaKn} + \hK^2 \tau_n^{2q}\min\Big\{\frac{1}{\lambdaKn \betaKn^2}, \frac{1}{\nu}\Big\}  + \tau_n^{2q + 1} + \tau_n^{2q+2} \Norm{\bbeta}{L^{\infty}(I_n; W^{1 ,\infty}(K)^d)}^2\Big) \Norm{u}{H^{q+1}(I_n; L^2(K))}^2 \\
& \quad + \Big(\hK^2 \tau_n^{2q+2} \Norm{\bbeta}{L^{\infty}(I_n; W^{1, \infty}(K)^d)}^2 + \frac{\hK^2 \tau_n^{2q+2}}{\lambdaKn} \Big) \Norm{\nabla u}{H^{q+1}(I_n; L^2(K)^d)}^2  \Big) \bigg)^{\frac12}.
\end{align*}
\end{lemma}
\begin{proof}
Integrating by parts in time, and using the fact that $u$ is continuous in time and the flux--jump identity~$\phi(\cdot, \tn^+) \jump{\psi}_n + \psi(\cdot, \tn^-) \jump{\phi}_n = \jump{\phi \psi}_n$, we get the following identity:
\begin{align}
\nonumber
\chimb & = \sum_{n = 1}^{N} \sum_{K \in \Omegah} \Big( -(u, \dpt \vh)_{\Kn} + \int_{I_n} \! \mhK(\uIt, \dpt \vh) + (\bbeta \cdot \nabla u,  \vh)_{\Kn} \\
\nonumber
& \quad - \frac12 (\bbeta \cdot \PPiO{k}{K} \nabla \uIt, \PiO{k}{K} \vh)_{\Kn} + \frac12 (\PiO{k}{K} \uIt, \bbeta \cdot \PPiO{k}{K} \nabla \vh)_{\Kn}\Big) \\
\nonumber
& \quad + \Big((u, \vh)_{\ST} - \mh(\uIt(\cdot, T), \vh(\cdot, T))\Big) \\
\nonumber
& \quad + \sum_{n = 1}^{N - 1} \Big( (u(\cdot, \tn), \jump{\vh}_n)_{\Omega} - \mh(\uIt(\cdot, \tn^-), \jump{\vh}_n) \Big) \\
\label{EQN::CHIMB-SPLIT}
& =: \sum_{n = 1}^N \sum_{K \in \Th} \chiV^{\Kn} + \chiT + \sum_{n = 1}^{N - 1} \chiJ^n.
\end{align}

\paragraph{$\bullet$ Estimate of~$\chiV^{\Kn}$.}
We first consider the volume terms~$\chiV^{\Kn}$.
Using the definition of the bilinear form~$\mh(\cdot, \cdot)$, the definition of the operators~$\tLhtK$ and~$\tLK$, the skew-symmetry of the bilinear form~$b(\cdot, \cdot)$, and the orthogonality properties of~$\PiO{k}{K}$ and~$\PPiO{k}{K}$, we have
\begin{align*}
\sum_{n = 1}^N \sum_{K \in \Omegah} \chiV^{\Kn} & = \sum_{n = 1}^{N} \sum_{K \in \Omegah} \bigg( -(u, \dpt \vh)_{\Kn} + \int_{I_n} \! \mhK(\uIt, \dpt \vh) - (u,  \bbeta \cdot \nabla \vh)_{\Kn}  \\
& \quad - \frac12 (\bbeta \cdot \PPiO{k}{K} \nabla \uIt, \PiO{k}{K} \vh)_{\Kn} + \frac12 (\PiO{k}{K} \uIt, \bbeta \cdot \PPiO{k}{K} \nabla \vh)_{\Kn}\bigg)\\
& = \sum_{n = 1}^{N} \sum_{K \in \Omegah} \bigg[-(u, \tLK \vh)_{\Kn} + (\PiO{k}{K}\uIt, \dpt \PiO{k}{K} \vh + \bbeta \cdot \PPiO{k}{K} \nabla \vh)_{\Kn} \\
& \quad + \int_{I_n} \smK((\Id - \PiO{k}{K}) \uIt, \dpt (\Id - \PiO{k}{K}) \vh) \dt \\
& \quad - \frac12 \Big((\bbeta \cdot \PPiO{k}{K} \nabla \uIt, \PiO{k}{K} \vh)_{\Kn}
+ (\PiO{k}{K} \uIt, \bbeta \cdot \PPiO{k}{K} \nabla \vh)_{\Kn} \Big) \bigg] \\
& = \sum_{n = 1}^{N} \sum_{K \in \Omegah} \bigg[-(u, \dpt(\Id - \PiO{k}{K}) \vh)_{\Kn}  - \big(u\bbeta, (\Id - \PPiO{k}{K}) \nabla \vh\big)_{\Kn}  \\
& \quad - (\PiO{k}{K}\uIt - u, \dpt \PiO{k}{K} \vh + \bbeta \cdot \PPiO{k}{K} \nabla \vh)_{\Kn} \\
& \quad + \int_{I_n} \smK((\Id - \PiO{k}{K}) \uIt, \dpt (\Id - \PiO{k}{K}) \vh) \dt  \\
& \quad - \frac12 \Big((\bbeta \cdot \PPiO{k}{K} \nabla \uIt, \PiO{k}{K} \vh)_{\Kn} + (\PiO{k}{K} \uIt, \bbeta \cdot \PPiO{k}{K} \nabla \vh)_{\Kn}  \Big) \bigg] \\
& = : -\sum_{n = 1}^N \sum_{K \in \Omegah} \Big( \chiVone^{\Kn} + \chiVtwo^{\Kn} + \chiVthree^{\Kn} + \chiVfour^{\Kn} + \frac12 \chiVfive^{\Kn}\Big).
\end{align*}

The terms~$\chiVone^{\Kn}$ and~$\chiVtwo^{\Kn}$ can be treated similarly to~$\chisupgtwo$ and~$\chisupgthree$ in Lemma~\ref{LEMMA::ESTIMATE-CHISUPG}, respectively.
The following estimates are then obtained:
\begin{align}
\nonumber
\sum_{n = 1}^N \sum_{K \in \Omegah} \chiVone^{\Kn} & =
\sum_{n = 1}^N \sum_{K \in \Omegah} \big(u, \dpt(\Id - \PiO{k}{K}) \vh \big)_{\Kn} \\
\label{EQN::CHIV-ONE}
& \lesssim \Big(\sum_{n = 1}^N \sum_{K \in \Th} \frac{\hK^{2s+4}}{\lambdaKn \betaKn^2 \tau_n^2
} \Norm{u}{L^2(I_n; H^{s+1}(K))}^2\Big)^{\frac12} \SemiNorm{\vh}{\supg},
\end{align}
\begin{align}
\nonumber
\sum_{n = 1}^N \sum_{K \in \Omegah} \chiVtwo^{\Kn} & = \sum_{n = 1}^N \sum_{K \in \Omegah} (u\bbeta, (\Id - \PPiO{k}{K}) \nabla \vh)_{\Kn}\\
\label{EQN::CHIV-TWO}
& \lesssim \Big(\sum_{n = 1}^N \sum_{K \in \Omegah} \hK^{2s + 2} \min\Big\{\frac{1}{\lambdaKn \betaKn^2}, \frac{1}{\nu} \Big\}\Norm{\bbeta}{L^{\infty}(I_n; W^{s + 1, \infty}(K)^d)}^2 \Norm{u}{L^2(I_n; H^{s+1}(K))}^2\Big)^{\frac12} \TNorm{\vh}.
\end{align}

The term~$\chiVthree^{\Kn}$ can be bounded using the definition of~$\SemiNorm{\cdot}{\Knsupg}$ in~\eqref{EQN::LOCAL-SUPG-NORM}, the Cauchy--Schwarz and the triangle inequalities, the estimates for~$\PiO{k}{K}$ and~$\Pi_r^t$ in Lemma~\ref{LEMMA::BRAMBLE-HILBERT}, and the VE interpolation estimate in Lemma~\ref{LEMMA::VE-APPROX}, as follows:
\begin{align}
\nonumber
\sum_{n = 1}^N \sum_{K \in \Omegah} \chiVthree^{\Kn} & = \sum_{n = 1}^N \sum_{K \in \Omegah} \Big( (\PiO{k}{K} \uIt - u, \tLhtK \vh)_{\Kn} + (\PiO{k}{K} \uIt - u, \bbeta \cdot \PPiO{k}{K} (\Id - \PPiO{k-1}{K}) \nabla \vh)_{\Kn}\Big)\\
\nonumber
& \le \sum_{n = 1}^N \sum_{K \in \Omegah} \lambdaKn^{-1/2}\Norm{\PiO{k}{K} \uIt - u}{L^2(\Kn)} \SemiNorm{\vh}{\Knsupg} \\
\nonumber
& \le \sum_{n = 1}^N \sum_{K \in \Omegah} \lambdaKn^{-1/2} \big(\Norm{\PiO{k}{K}(\uIt - u)}{L^2(\Kn)} + \Norm{\ePiO}{L^2(\Kn)} \big) \SemiNorm{\vh}{\Knsupg} \\
\label{EQN::CHIV-THREE}
& \lesssim \bigg(\sum_{n = 1}^N \sum_{K \in \Omegah} \Big(\frac{\tau_n^{2q+2}}{\lambdaKn} \Norm{u}{H^{q+1}(I_n; L^2(K))}^2 + \frac{\hK^{2s+2}}{\lambdaKn} \Norm{u}{L^2(I_n; H^{s+1}(K))}^2\Big) \bigg)^{\frac12} \SemiNorm{\vh}{\supg}.
\end{align}

As for the term~$\chiVfour^{\Kn}$, we use the stability bound~\eqref{EQN::STAB-smK}, the polynomial inverse estimate~\eqref{EQN::INVERSE-ESTIMATE-TIME}, the scaled Poincar\'e--Friedrichs inequality in Lemma~\ref{LEMMA::POINCARE-Friedrichs}, and similar steps as those used to estimate~$\chisupgthree$ in~\eqref{EQN::CHI-F3} to obtain the following estimate:
\begin{align}
\nonumber
\sum_{n = 1}^N \sum_{K \in \Omegah} \chiVfour^{\Kn} & = -\sum_{n = 1}^N \sum_{K \in \Omegah} \int_{I_n} \smK((\Id - \PiO{k}{K})\uIt, \dpt (\Id - \PiO{k}{K}) \vh) \dt \\
\nonumber
& \le \sum_{n = 1}^N \sum_{K \in \Omegah} \hat{\mu} \Norm{(\Id - \PiO{k}{K})\uIt}{L^2(\Kn)} \Norm{\dpt(\Id - \PiO{k}{K}) \vh}{L^2(\Kn)} \\
\nonumber
& \lesssim \sum_{n = 1}^N \sum_{K \in \Omegah} \frac{\mu^* \hK}{\tau_n} \Norm{(\Id - \PiO{k}{K}) \uIt}{L^2(\Kn)} \Norm{\nabla (\Id - \PiN{k}{K}) \vh}{L^2(\Kn)^d} \\
\nonumber
& \lesssim \sum_{n = 1}^N \sum_{K \in \Omegah} \min\bigg\{\frac{\SemiNorm{\vh}{\Knsupg}}{\lambdaKn^{1/2} \betaKn}, \frac{\Norm{\nabla \vh}{L^2(\Kn)^d}}{\nu^{1/2}}\bigg\} \Big(\frac{\hK^{s+2}}{\tau_n} \Norm{u}{L^2(I_n; H^{s+1}(K))} \\
\nonumber
& \quad + \hK \tau_n^{q} \Norm{u}{H^{q+1}(I_n; L^2(K))} \Big) \\
\nonumber
& \lesssim \bigg(\sum_{n = 1}^N \sum_{K \in \Omegah} \min\Big\{\frac{1}{\lambdaKn \betaKn^2}, \frac{1}{\nu}\Big\} \Big(\frac{\hK^{2s+4}}{\tau_n^2} \Norm{u}{L^2(I_n; H^{s+1}(K))}^2 \\
\label{EQN::CHIV-FOUR}
& \quad + \hK^2 \tau_n^{2q} \Norm{u}{H^{q+1}(I_n; L^2(K))}^2 \Big) \bigg)^{\frac12} \TNorm{\vh}.
\end{align}

Adding and subtracting suitable terms, recalling the antisymmetry of the form $b(\cdot,\cdot)$ and using the orthogonality properties of~$\PiO{k}{K}$ and~$\PPiO{k}{K}$, the following identity can be obtained:
\begin{align}
\nonumber
\sum_{n = 1}^N \sum_{K \in \Omegah} \chiVfive^{\Kn}
& = \sum_{n = 1}^N \sum_{K \in \Omegah} \Big( (\bbeta \cdot \PPiO{k}{K} \nabla \uIt, \PiO{k}{K} \vh)_{\Kn} + (\PiO{k}{K} \uIt, \bbeta \cdot \PPiO{k}{K} \nabla \vh)_{\Kn}\Big) \\
\nonumber
& = \sum_{n = 1}^N \sum_{K \in \Omegah} \Big( (\bbeta \cdot (\PPiO{k}{K} - \Id) \nabla \uIt, \PiO{k}{K} \vh)_{\Kn} + (\bbeta \cdot \nabla \uIt, \vh)_{\Kn} \\
\nonumber
& \quad - (\bbeta \cdot \nabla \uIt, (\Id - \PiO{k}{K})\vh)_{\Kn} + ((\PiO{k}{K} - \Id) \uIt, \bbeta \cdot \PPiO{k}{K} \nabla \vh)_{\Kn} \\
\nonumber
& \quad - (\uIt, \bbeta \cdot (\Id - \PPiO{k}{K}) \nabla \vh)_{\Kn} + (\uIt, \bbeta \cdot \nabla \vh)_{\Kn} \Big) \\
\nonumber
& = -\sum_{n = 1}^N \sum_{K \in \Omegah} \Big( ((\Id - \PPiO{k}{K}) \nabla \uIt, (\Id -\PPiO{0}{K})\bbeta \PiO{k}{K} \vh)_{\Kn} \\
\nonumber
& \quad + ((\Id - \PiO{k}{K}) (\bbeta \cdot \nabla \uIt), (\Id - \PiO{k}{K})\vh)_{\Kn} \\
\nonumber
& \quad + ((\Id - \PiO{k}{K}) \uIt, (\Id - \PPiO{0}{K})\bbeta \cdot \PPiO{k}{K} \nabla \vh)_{\Kn} \\
\nonumber
& \quad + ((\Id - \PPiO{k}{K})(\uIt \bbeta), (\Id - \PPiO{k}{K}) \nabla \vh)_{\Kn}\Big)  \\
\label{EQN::CHIV-SPLIT}
& =: -\sum_{n = 1}^N \sum_{K \in \Omegah} \big(\Theta_1^{\Kn} + \Theta_2^{\Kn} + \Theta_3^{\Kn} + \Theta_4^{\Kn} \big).
\end{align}

The first term on the right-hand side of~\eqref{EQN::CHIV-SPLIT} can be estimated using the Cauchy--Schwarz inequality, the estimates for~$\PiO{k}{K}$, $\PiO{0}{K}$, and~$\Pi_r^t$ in Lemma~\ref{LEMMA::BRAMBLE-HILBERT} and their stability properties, and the VE interpolation estimate in Lemma~\ref{LEMMA::VE-APPROX}, as follows:
\begin{align}
\nonumber
\Theta_1^{\Kn} & = ((\Id - \PPiO{k}{K}) \nabla \uIt, (\Id - \PPiO{0}{K}) \bbeta \PiO{k}{K} \vh)_{\Kn} \\
\nonumber
& = ((\Id - \PPiO{k}{K}) \nabla u, (\Id - \PPiO{0}{K}) \bbeta \PiO{k}{K} \vh)_{\Kn} \\
\nonumber
& \quad + ((\Id - \PPiO{k}{K}) \nabla (u - \uIt), (\Id - \PPiO{0}{K}) \bbeta \PiO{k}{K} \vh)_{\Kn} \\
\nonumber
& \lesssim \hK\Norm{\bbeta}{L^{\infty}(I_n; W^{1, \infty}(K)^d)} \Big(\Norm{(\Id - \PPiO{k}{K})\nabla u}{L^2(\Kn)^d}  + \Norm{\nabla(u - \uIt)}{L^2(\Kn)^d} \Big) \Norm{\vh}{L^2(\Kn)} \\
\label{EQN::THETA-1}
& \lesssim \Norm{\bbeta}{L^{\infty}(I_n; W^{1, \infty}(K)^d)} \big(\hK^{s+1} \Norm{u}{L^2(I_n; H^{s+1}(K))} + \hK\tau_n^{q+1} \Norm{\nabla u}{H^{q+1}(I_n; L^2(K)^d)} \big)\Norm{\vh}{L^2(\Kn)}.
\end{align}

As for the second term on the right-hand side of~\eqref{EQN::CHIV-SPLIT}, we use the estimates for~$\PiO{k}{K}$ and~$\Pi_r^t$ in Lemma~\ref{LEMMA::BRAMBLE-HILBERT} and their stability properties, the VE interpolation estimate in Lemma~\ref{LEMMA::VE-APPROX}, and the scaled Poincar\'e--Friedrichs inequality in Lemma~\ref{LEMMA::POINCARE-Friedrichs}, to get
\begin{align}
\nonumber
\Theta_2^{\Kn} & = ((\Id - \PiO{k}{K})(\bbeta \cdot \nabla \uIt), (\Id - \PiO{k}{K}) \vh)_{\Kn} \\
\nonumber
& = ((\Id - \PiO{k}{K})(\bbeta \cdot \nabla u), (\Id - \PiO{k}{K}) \vh)_{\Kn} \\
\nonumber
& \quad + ((\Id - \PiO{k}{K})(\bbeta \cdot \nabla (\uIt - u)), (\Id - \PiO{k}{K}) \vh)_{\Kn} \\
\nonumber
& \lesssim \hK \Big(\hK^s\Norm{\bbeta}{L^{\infty}(I_n; W^{s, \infty}(K)^d)} \Norm{u}{L^2(I_n; H^{s+1}(K))} \\
\nonumber
& \quad + \betaKn \Norm{\nabla(u - \uIt)}{L^2(\Kn)^d}\Big) \Norm{\nabla (\Id - \PiN{k}{K}) \vh}{L^2(\Kn)^d} \\
\nonumber
& \lesssim \bigg(\frac{\hK^{s+1}}{\lambdaKn^{1/2} \betaKn} \Norm{\bbeta}{L^{\infty}(I_n; W^{s, \infty}(K)^d)} \Norm{u}{L^2(I_n; H^{s+1}(K))} + \frac{\hK \tau_n^{q+1}}{\lambdaKn^{1/2}} \Norm{\nabla u}{H^{q+1}(I_n; L^2(K)^d)} \\
\label{EQN::THETA-2}
& \quad + \frac{\hK^{s+1}}{\lambdaKn^{1/2}} \Norm{u}{L^2(I_n; H^{s+1}(K))}\bigg) \SemiNorm{\vh}{\Knsupg}.
\end{align}

Using the triangle inequality, the estimates for~$\PPiO{0}{K}$, $\PiO{k}{K}$, and~$\Pi_r^t$ in Lemma~\ref{LEMMA::BRAMBLE-HILBERT} and their stability properties, the VE interpolation estimate in Lemma~\ref{LEMMA::VE-APPROX}, and the VE inverse estimate~\eqref{EQN::INVERSE-ESTIMATE-SPACE}, we obtain the following estimate:
\begin{align}
\nonumber
\Theta_3^{\Kn} & =  ((\Id - \PiO{k}{K}) \uIt, (\Id - \PPiO{0}{K}) \bbeta \cdot \PPiO{k}{K} \nabla \vh)_{\Kn} \\
\nonumber
& \lesssim \hK \Norm{\bbeta}{L^{\infty}(I_n; W^{1, \infty}(K)^d)} \Norm{(\Id - \PiO{k}{K})\uIt}{L^2(\Kn)} \Norm{\nabla \vh}{L^2(\Kn)^d} \\
\nonumber
& \lesssim \Norm{\bbeta}{L^{\infty}(I_n; W^{1, \infty}(K)^d)} \Big(\Norm{u - \uIt}{L^2(\Kn)} + \Norm{(\Id - \PiO{k}{K}) u}{L^2(\Kn)} \\
\nonumber
& \quad + \Norm{\PiO{k}{K}(u - \uIt)}{L^2(\Kn)}\Big) \Norm{\vh}{L^2(\Kn)} \\
\label{EQN::THETA-3}
& \lesssim \Norm{\bbeta}{L^{\infty}(I_n; W^{1, \infty}(K)^d)} \Big(\tau_n^{q+1} \Norm{u}{H^{q+1}(I_n; L^2(K))} + \hK^{s+1} \Norm{u}{L^2(I_n; H^{s+1}(K))}\Big) \Norm{\vh}{L^2(\Kn)}.
\end{align}

Finally, using the estimates for~$\PPiO{k}{K}$ and~$\Pi_r^t$ in Lemma~\ref{LEMMA::BRAMBLE-HILBERT} and their stability properties, the Cauchy--Schwarz inequality, the definition of~$\SemiNorm{\cdot}{\Knsupg}$ in~\eqref{EQN::LOCAL-SUPG-NORM}, and the VE interpolation estimate in Lemma~\ref{LEMMA::VE-APPROX}, we have
\begin{align}
\nonumber
\Theta_4^{\Kn} & = ((\Id - \PPiO{k}{K}) (\uIt \bbeta), (\Id - \PPiO{k}{K}) \nabla \vh)_{\Kn} \\
\nonumber
& = ((\Id - \PPiO{k}{K})(u\bbeta), (\Id - \PPiO{k}{K}) \nabla \vh)_{\Kn} \\
\nonumber
& \quad - ((\Id - \PPiO{k}{K})((u - \uIt) \bbeta), (\Id - \PPiO{k}{K}) \nabla \vh)_{\Kn} \\
\nonumber
& \lesssim \Big(\hK^{s+1}\Norm{\bbeta}{L^{\infty}(I_n; W^{s+1, \infty}(K)^d)} \Norm{u}{L^2(I_n; H^{s+1}(K))} + \betaKn \Norm{u - \uIt}{L^2(\Kn)} \Big) \Norm{(\Id - \PPiO{k}{K}) \nabla \vh}{L^2(\Kn)^d)} \\
\nonumber
& \lesssim \bigg( \frac{\hK^{s+1}}{\lambdaKn^{1/2} \betaKn} \Norm{\bbeta}{L^{\infty}(I_n; W^{s+1, \infty}(K)^d)} \Norm{u}{L^2(I_n; H^{s+1}(K))}  \\
\label{EQN::THETA-4}
& \quad + \frac{\tau_n^{q+1}}{\lambdaKn^{1/2}} \Norm{u}{H^{q+1}(I_n; L^2(K))} + \frac{\hK^{s+1}}{\lambdaKn^{1/2}} \Norm{u}{L^2(I_n; H^{s+1}(K))}\bigg) \SemiNorm{\vh}{\Knsupg}.
\end{align}

Combining identity~\eqref{EQN::CHIV-SPLIT} with the local estimates~\eqref{EQN::THETA-1}, \eqref{EQN::THETA-2}, \eqref{EQN::THETA-3}, and~\eqref{EQN::THETA-4}, we get
\begin{align}
\nonumber
\sum_{n = 1}^N \sum_{K \in \Omegah} \chiVfive^{\Kn}
& \lesssim \bigg[\sum_{n = 1}^N \sum_{K \in \Omegah} \Big(
\Big(\Norm{\bbeta}{L^{\infty}(I_n; W^{1, \infty}(K)^d)}^2 \hK^{2s+2} + \frac{\hK^{2s+2}}{\lambdaKn \betaKn^2} \Norm{\bbeta}{L^{\infty}(I_n; W^{s+1}(K)^d)}^2 \\
\nonumber
& \quad + \frac{\hK^{2s+2}}{\lambdaKn}\Big)\Norm{u}{L^2(I_n; H^{s+1}(K))}^2 \\
\nonumber
& \quad + \Big(\hK^2 \tau_n^{2q+2} \Norm{\bbeta}{L^{\infty}(I_n; W^{1, \infty}(K)^d)}^2 + \frac{\hK^2 \tau_n^{2q+2}}{\lambdaKn} \Big) \Norm{\nabla u}{H^{q+1}(I_n; L^2(K)^d)}^2  \\
\label{EQN::CHIV-FIVE}
& \quad + \Big(\tau_n^{2q+2} \Norm{\bbeta}{L^{\infty}(I_n; W^{1 ,\infty}(K)^d)}^2 + \frac{\tau_n^{2q+2}}{\lambdaKn} \Big) \Norm{u}{H^{q+1}(I_n; L^2(K))}^2
\Big) \bigg]^{\frac12} \TNorm{\vh}.
\end{align}

\paragraph{Estimate of~$\chiT$ and~$\chiJ^n$.} Using the polynomial consistency of the bilinear form~$\mhK(\cdot, \cdot)$, the stability bound~\eqref{EQN::STAB::mK}, the trace inequality in Lemma~\ref{LEMMA::TRACE-INEQ}, the commutativity of~$\PiO{k}{K}$ and the VE interpolant with the first-order time derivative operator~$\dpt$, the Cauchy--Schwarz inequality, and the definition of the upwind-jump functional~$\SemiNorm{\cdot}{\sfJ}$ in~\eqref{EQN::JUMP-FUNCTIONAL}, we obtain the following estimate:

\begin{align}
\nonumber
\chiT + \sum_{n = 1}^{N - 1} \chiJ^n & = (u, \vh)_{\ST} - \mh(\uIt(\cdot, T), \vh(\cdot, T)) \\
\nonumber
& \quad + \sum_{n = 1}^{N - 1} \Big( (u(\cdot, \tn), \jump{\vh}_n)_{\Omega} - \mh(\uIt{}(\cdot, \tn^-), \jump{\vh}_n) \Big) \\
\nonumber
& = \sum_{K \in \Omegah} \Big( ((\Id - \PiO{k}{K}) u(\cdot, T), \vh(\cdot, T))_{K} + \mhK(\PiO{k}{K} u - \uIt(\cdot, T), \vh(\cdot, T)) \\
\nonumber
& \quad + \sum_{n = 1}^{N - 1} \sum_{K \in \Omegah} \Big( ((\Id - \PiO{k}{K}) u(\cdot, \tn), \jump{\vh}_n)_{K} + \mhK(\PiO{k}{K} u(\cdot, \tn) - \uIt(\cdot, \tn^-), \jump{\vh}_n) \Big) \\
\nonumber
& \lesssim \bigg(\sum_{n = 1}^N \sum_{K \in \Omegah} \Big(\tau_n^{-1} \Norm{(\Id - \PiO{k}{K}) u}{L^2(\Kn)}^2 + \tau_n \Norm{(\Id - \PiO{k}{K}) \dpt u}{L^2(\Kn)} \\
\nonumber
& \quad + \tau_n^{-1} (\mu^*)^2 \big(\Norm{(\Id - \PiO{k}{K})u}{L^2(\Kn)}^2 + \Norm{u - \uIt}{L^2(\Kn)}^2 \big) \\
\nonumber
& \quad + \tau_n (\mu^*)^2 \big(\Norm{(\Id - \PiO{k}{K}) \dpt u}{L^2(\Kn)}^2 + \Norm{\dpt(u - \uIt)}{L^2(\Kn)}^2\big)\Big) \bigg)^{\frac12} \SemiNorm{\vh}{\sfJ} \\
\nonumber
& \lesssim \bigg(\sum_{n = 1}^N \sum_{K \in \Omegah} \Big(\frac{\hK^{2s+2}}{\tau_n} \Norm{u}{L^2(I_n; H^{s+1}(K))}^2 + \tau_n \hK^{2s+2} \Norm{\dpt u}{L^2(I_n; H^{s+1}(K))}^2 \\
\label{EQN::CHITJ}
& \quad + \tau_n^{2q+1} \Norm{u}{H^{q+1}(I_n; L^2(K))}^2 \Big) \bigg)^{\frac12} \SemiNorm{\vh}{\sfJ}.
\end{align}
The desired result is then obtained by combining identity~\eqref{EQN::CHIMB-SPLIT} with estimates~\eqref{EQN::CHIV-ONE}, \eqref{EQN::CHIV-TWO}, \eqref{EQN::CHIV-THREE}, \eqref{EQN::CHIV-FOUR}, \eqref{EQN::CHIV-FIVE}, and~\eqref{EQN::CHITJ}.
\end{proof}

\begin{remark}[Alternative bound for $\chiVone^{\Kn}$]\label{rem:alternative-2}
Along the same lines as in Remark \ref{rem:alternative-1}, whenever the mesh $\Omega_h$ is fixed one could also handle $\chiVone^{\Kn}$ in \eqref{EQN::CHIV-ONE} in a more efficient way, leading to a bound where the term $\tau_n^{-2}$ is substituted by the more favorable $\tau_n^{-1}$, provided that~$\dpt u$ is regular enough. One obtains
\begin{align}
\nonumber
\sum_{n = 1}^N & \sum_{K \in \Omegah} \chiVone^{\Kn} =
\sum_{n = 1}^N \sum_{K \in \Omegah} \big(u, \dpt(\Id - \PiO{k}{K}) \vh \big)_{0, \Kn} \\
\nonumber
& \lesssim \Big(\sum_{n = 1}^N \sum_{K \in \Th} \hK^{2s+2} \Norm{\dpt u}{L^2(I_n; H^{s+1}(K))}^2\Big)^{\frac12} \Norm{\vh}{L^2(\QT)} \\
\nonumber
& \quad + \bigg(\sum_{n = 1}^N\sum_{K \in \Th} \Big( \frac{\hK^{2s+2}}{\tau_n} \Norm{u}{L^2(I_n; H^{s+1}(K))}^2 + \tau_n \hK^{2s+2} \Norm{\dpt u}{L^2(I_n; H^{s+1}(K))}^2\Big) \bigg)^{\frac12} \SemiNorm{\vh}{\sf J}.
\end{align}
\eremk
\end{remark}

\medskip\noindent
{\bf Proof of Theorem \ref{THM::ERROR-ESTIMATES}.}
Combining the \emph{a priori} error bound in Proposition~\ref{PROP::A-PRIORI-BOUNDS} with the estimates in Lemmas~\ref{LEMMA::INTERPOLATION-ERROR-ESTIMATE}, \ref{LEMMA::ESTIMATE-CHIL}, \ref{LEMMA::ESTIMATE-CHISUPG}, \ref{LEMMA::ESTIMATE-CHIA}, and~\ref{LEMMA::ESTIMATE-CHIMB} one obtains a \emph{general convergence result}, which underlines the different local contributions.
The proof of Theorem \ref{THM::ERROR-ESTIMATES} follows as a simplified case, simply by elaborating on the bounds using the specific theorem assumptions and finally dropping all the regularity terms.

\subsection{Avoiding degeneration of the error estimates for \texorpdfstring{$\tau \ll h$}{tau << h}}
\label{sec:no-tau-below}

Under suitable conditions, we are able to eliminate the terms where $\tau$ appears at the denominator in Theorem \ref{THM::ERROR-ESTIMATES}.
In the present section, we describe briefly the involved modifications. As a starting point, we obviously set ourselves within the range of assumptions outlined in Theorem~\ref{THM::ERROR-ESTIMATES}.

Differently from the previous part, see Remark \ref{rem:mesh-changing}, we now require that the mesh $\Omega_h$ is fixed and does not change from one time-slab to the next. This is quite natural as the ``transfer'' error induced by a change of mesh depends on the spatial meshsize~$h$, but the number of such occurrences grows as $\tau^{-1}$. Furthermore, we require the polynomial order in time $r \ge 1$.
Finally, we assume (mainly for simplicity of exposition) that the spatial mesh (family) is quasi-uniform.

The first key point is substituting the space--time projection operator introduced at the beginning of Section \ref{SUBSECT::A-PRIORI-ERROR} with the following approximant, for which we keep the same notation.
Given $u \in H^1(0, T; L^2(\Omega))$ the discrete function $\uIt \in \Vht$ is \emph{continuous in time} and defined by

$$
\left\{
\begin{aligned}
& m_h(\uIt(\cdot,t_n), v_h) = (u(\cdot,t_n),v_h)_{{0, \Omega}} \qquad \forall n \in \{0,1,..,N\}, \ \forall v_h \in \Vh, \\
& \int_{t_{n-1}}^{t_n} m_h(\partial_t \uIt(\cdot,t), (t-t_{n-1})^l v_h) {\dt} = \int_{t_{n-1}}^{t_n} (\partial_t u(\cdot,t), (t-t_{n-1})^l v_h){\dt} \\
& \qquad \qquad  \forall n \in \{1,..,N\}, \ \forall v_h \in \Vh , \forall l \in \{1,2,..,r-1\} \, .
\end{aligned}
\right.
$$
It can be checked that the above interpolant satisfies, for all $n=1,2,..,N$, the continuity properties
$$
\| \uIt \|_{L^2(I_n,L^2(\Omega))} \lesssim \| u \|_{L^2(I_n,L^2(\Omega))} + \tau_n \| \partial_t u \|_{L^2(I_n,L^2(\Omega))} \ {\text{ and }} \
 \| \partial_t \uIt \|_{L^2(I_n,L^2(\Omega))} \lesssim \| \partial_t u \|_{L^2(I_n,L^2(\Omega))},
$$
and also that optimal (in $h,\tau$) approximation error bounds hold for $(u - \uIt)$ in the various norms of interest needed in our analysis.

Thanks to the continuity of $\uIt$, the error term $\SemiNorm{\eIt}{\sfJ}^2$ in Lemma \ref{LEMMA::INTERPOLATION-ERROR-ESTIMATE} vanishes; note that such quantity was the only responsible for the $1/\tau_n$ term appearing in Lemma \ref{LEMMA::INTERPOLATION-ERROR-ESTIMATE}.
Furthermore, thanks to the peculiar definition of~$\uIt$ here above, the term~$\chiT + \sum_{n = 1}^{N - 1} \chiJ^n$ appearing in~\eqref{EQN::CHITJ} also vanishes.
In the simpler finite element case, which is included in our analysis as explained in Section~\ref{SUBSECT::VE-VERSIONS}, these are the only terms leading to negative powers of~$\tau$ in our {error estimate}.

In the more general case of virtual elements, we need to modify the scheme, adding the following additional stabilization term to the discrete form~$\Bht(\cdot, \cdot)$:
\begin{equation}\label{eq:new-term}
\sum_{K \in \Omega_h} \sum_{n=1}^N \lambdaKn \int_{I_n} \smK(\dpt (\Id - \PiO{k}{K}) \uh, \dpt (\Id - \PiO{k}{K}) \vh) \dt,
\end{equation}
which, quite trivially, leads to a stability bound controlling a stronger norm $\TNorm{\uh}^2$, c.f. \eqref{EQN::ENERGY-NORM}, now including also the term
\begin{equation}\label{eq:new-ctrl}
\sum_{K \in \Omega_h} \sum_{n=1}^N \lambdaKn \| \dpt (\Id - \PiO{k}{K}) \vh \|_{L^2(K_n)}^2.
\end{equation}
The control on the test function in the above norm allows us to deal with all the remaining ``bad'' terms in our convergence analysis, namely
$\sum_{n = 1}^N \sum_{K \in \Th} \chisupgtwo^{\Kn}$ in Lemma \ref{LEMMA::ESTIMATE-CHISUPG}, plus~$\sum_{n = 1}^N \sum_{K \in \Omegah} \chiVone^{\Kn}$ and $\sum_{n = 1}^N \sum_{K \in \Omegah} \chiVfour^{\Kn}$ in Lemma \ref{LEMMA::ESTIMATE-CHIMB}. Indeed, we can now avoid using the inverse estimate in time and simply exploit directly control on~\eqref{eq:new-ctrl}; the Cauchy--Schwarz inequality and standard manipulations yield the error terms
$$
\begin{aligned}
& \bullet \sum_{n = 1}^N \sum_{K \in \Omegah} \lambdaKn \| (\Id - \PiO{k}{K})f \|_{L^2(\Kn)}^2
\lesssim \, h^{2k+3} , \\
& \bullet \sum_{n = 1}^N \sum_{K \in \Omegah} \lambdaKn^{-1} \| (\Id - \PiO{k}{K}) u \|_{L^2(\Kn)}^2 \lesssim \, \max{\{\nu,h\}} \, h^{2k} \, , \\
& \bullet \sum_{n = 1}^N \sum_{K \in \Omegah} \int_{I_n} \lambdaKn^{-1} \smK((\Id - \PiO{k}{K})\uIt, (\Id - \PiO{k}{K})\uIt)
\lesssim \, \max{\{\nu,h\}} \, \big( h^{2k} + \tau^{2k} (\tau/h)^2 \big) \, .
\end{aligned}
$$
The bounds here above can be easily obtained by the same techniques used in the rest of this contribution and therefore we avoid showing the details.
Finally, it can be checked that term~\eqref{eq:new-term} is also of optimal order with respect to the interpolation error, as usual assuming sufficient regularity of the solution~$u$.

\begin{remark}[The case $r=0$]
In the case $r=0$, we cannot take a continuous-in-time interpolant, but we can choose the unique approximant that satisfies
$$
m_h(\uIt(\cdot,t_n), v_h) = (u(\cdot,t_n^{-}),v_h)_{{0, \Omega}} \qquad \forall n \in \{1,2,..,N\}, \ \forall v_h \in \Vh .
$$
It is immediate to check that this choice is still sufficient to make term~\eqref{EQN::CHITJ} vanish. In order to deal with
$\SemiNorm{\eIt}{\sfJ}^2$ in Lemma \ref{LEMMA::INTERPOLATION-ERROR-ESTIMATE}, we simply avoid such term in the interpolation estimates. Therefore, in order to avoid negative powers of $\tau$, the final error bound for the $r=0$ case will be in the weaker norm
$$
|\!|\!| \wh |\!|\!|_{\mathcal{E}, r=0}^2 := \Norm{\wh}{L^2(\ST)}^2 + \Norm{\wh}{L^2(\QT)}^2 + \nu \Norm{\nabla \wh}{L^2(\QT)^d}^2 + \SemiNorm{\wh}{\supg}^2
$$
where the jump terms have been excluded. Finally, note that in the~$r=0$ case there is no need to introduce~\eqref{eq:new-term} since all the associated terms are now vanishing.
\eremk
\end{remark}

\section{Numerical tests}\label{sec:num}

In this section, we present some numerical results in three space dimensions
in order to (1) validate the theoretical derivations from the practical perspective, and evidence the effectiveness of the proposed stabilized scheme
compared to a non-stabilized approach. For the sake of efficiency, we make use of the serendipity version of VEM, see Section \ref{SUBSECT::VE-VERSIONS}.
We will refer to the stabilized method, i.e.,  the scheme described in~\eqref{EQN::FULLY-DISCRETE-SCHEME} as \stab{},
whereas the scheme \emph{without} the SUPG stabilization terms will be denoted as \nostab{}.
As described in the previous sections,
the \stab{} scheme depends on a set of parameters.
In all the numerical experiments of this section,
we will set these parameters as
$$
\zeta = 0.1,\quad \Cinv
\approx 10k^2 \quad \text{and} \quad \lambdaKn = \zeta \min\bigg\{\frac{\hK^2}{\nu \Cinv^2}, \frac{\hK}{\betaQT}\bigg\}\,.
$$
An (interpolatory) Lagrangian basis is used for the space~$\Pp{k}(I_n)$ for~$n = 1, \ldots, N$, so that the associated degrees of freedom (DoFs) are pointwise evaluations.
Finally, we use the classical \texttt{dofi-dofi} choice (see, e.g., \cite[\S4.6]{volley}) for the VE stabilization term.

The numerical experiments are organized as follows.
In Section~\ref{sec:num:conv}, we evaluate the convergence trend of the schemes \stab{} and \nostab{},
both in the convection- and the diffusion-dominated regimes.
Then, in Section~\ref{sec:num:bench}, we apply these two schemes to a benchmark problem similar to the \say{Three body movement} proposed in~\cite[\S5.1]{Ahmed_Matthies_Tobiska_Xie:2011}.

\subsection{A standard convergence test}\label{sec:num:conv}
In this section, we solve the time dependent convection-diffusion problem in the
space--time domain~$\QT = (0, 1)^3 \times (0, 1.5)$.
As transport advective field, we set the following space--time dependent function:
$$
\bbeta(x,\,y,\,z,\,t):=\left[\begin{array}{r}
     e^{t/2}\sin(\pi(x+y+2z))\\
     e^{t/2}\sin(\pi(x+y+2z))\\
    -e^{t/2}\sin(\pi(x+y+2z))
\end{array}\right]\,
$$
that increases exponentially with time.
We consider the following two values of the diffusive coefficients:
$$
\nu = 1 \quad \text{ and } \quad \nu = 10^{-10},
$$
so as to obtain a problem characterized by a diffusion- or a convection-dominated regime, respectively.
Then, the right-hand side will be properly modified according~$\nu$ so that the exact solution is
given by
$$
u(x,\,y,\,z,\,t) = e^{0.3t}\sin(\pi x)\cos(\pi y)\sin(\pi z)\,.
$$
In the proposed tests,
we consider a family of four spatial meshes with decreasing meshsize~$h$,
while the time domain is split in uniform intervals with time step~$\tau
\simeq h$.
In all the numerical experiments, the VE
approximation degree coincides with the time polynomial degree,
and we refer to it simply as~$k$.
Moreover, to test the robustness of the proposed method with respect to element distortion,
we consider two different types of spatial meshes:
\begin{itemize}
\item \texttt{cube}: structured meshes composed by cubes, see Figure~\ref{fig:meshes}(left panel);
\item \texttt{voro}: meshes composed by polyhedral elements that may have small edges or faces, see Figure~\ref{fig:meshes}(right panel).
\end{itemize}

\begin{figure}[!htb]
\centering
\begin{tabular}{ccc}
\includegraphics[width=0.31\textwidth]{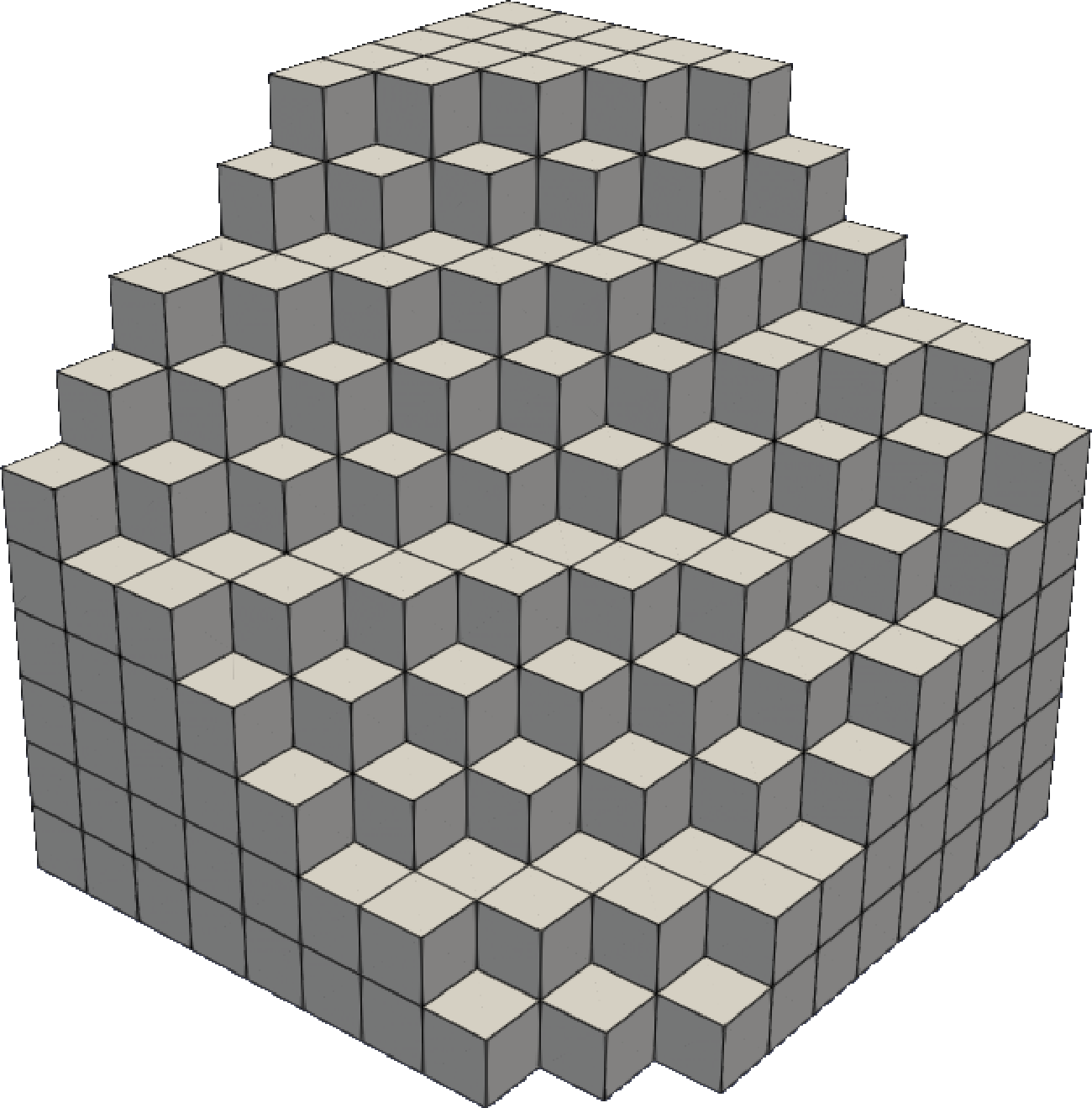} &\qquad\qquad&
\includegraphics[width=0.31\textwidth]{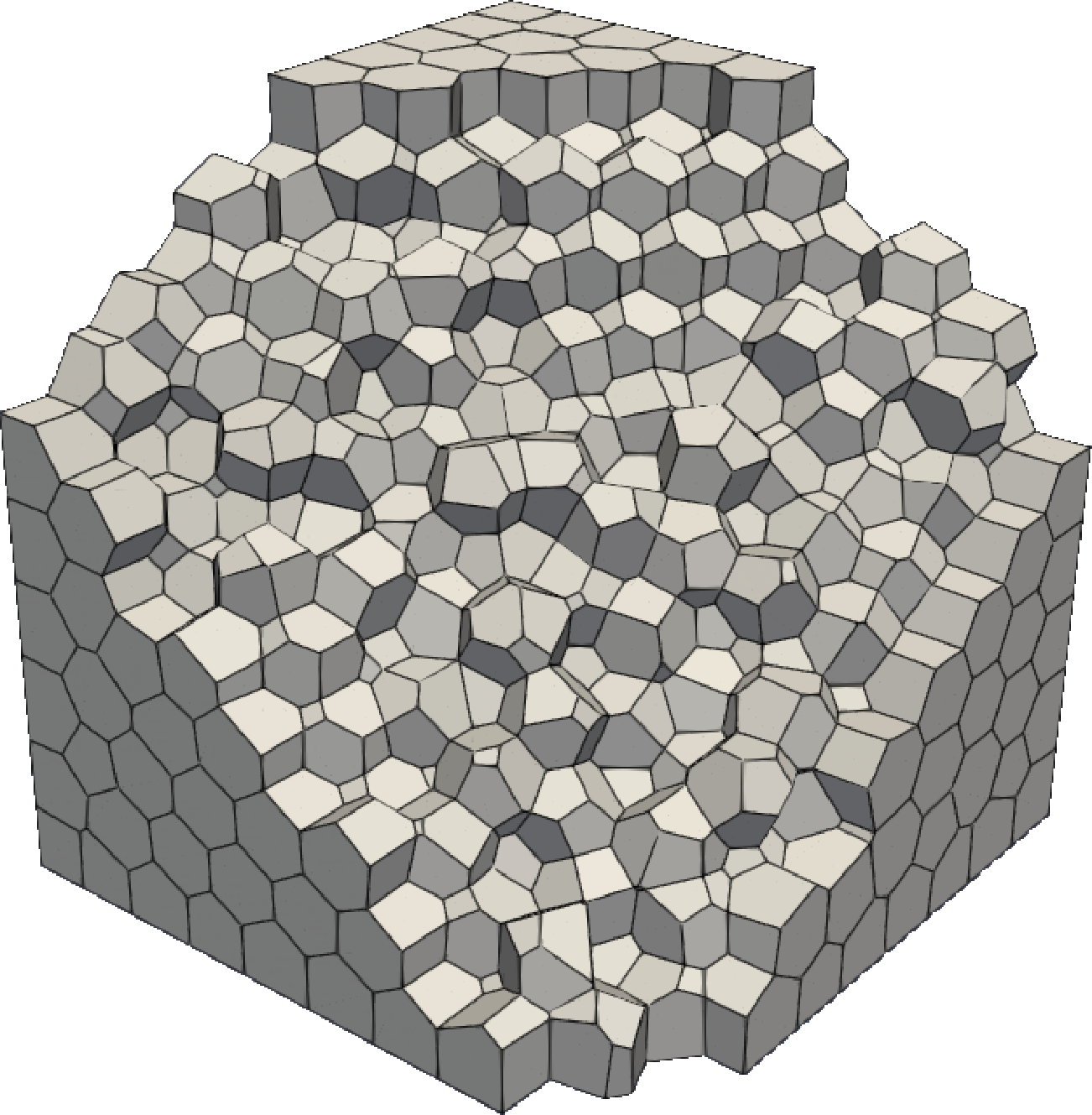} \\
\texttt{cube} & &\texttt{voro} \\
\end{tabular}
\caption{The interior of some spatial meshes of each type with approximately the same meshsize.}
\label{fig:meshes}
\end{figure}

For both schemes, we compute the following errors:
\begin{itemize}
\item the \textbf{$H^1$ seminorm error} at the final time
$$
e_{H^1}^T :=
\Big(\sum_{K\in\Omegah}
\Norm{\nabla(u - \Pi_k^\nabla \uh)(\cdot, T)}{L^2(K)^3}^2 \Big)^{\frac12};
$$
\item the \textbf{$L^2$ norm error} at final time
$$
e_{L^2}^T :=
\Big( \sum_{K \in \Omegah} \Norm{(u - \Pi_k^0 \uh)(\cdot, T)}{L^2(K)}^2 \Big)^{\frac12};
$$
\item the \textbf{$H^1$ norm error} on the space-time cylinder
$$
e_{H^1}^{Q_T} :=
\Big(\sum_{n = 1}^N \sum_{K \in \Th} \big(\Norm{u - \Pi_k^0 u}{L^2(\Kn)}^2 + \Norm{\nabla (u - \Pi_k^\nabla u)}{L^2(\Kn)^3}^2 \big) \Big)^{\frac12}.
$$
\end{itemize}
For the scheme \stab{}, under the above condition $\tau \simeq h$,
the following asymptotic behaviour is expected:
$$
e_{H^1}^T = \mathcal{O}\left(h^k\right), \qquad
e_{L^2}^T = \mathcal{O}\left(h^{k+1}\right),\qquad
e_{H^1}^{Q_T} = \mathcal{O} \left(h^k\right),
$$
in both the convection- and the diffusion-dominated regimes.
Furthermore, only for the solution~$\uh$ obtained with the scheme \stab{},
we compute also the following quantity:
$$
e_{\Vht} := \TNorm{\uh - u_I}\,,
$$
where~$u_I$ is the DoF-interpolant of the exact solution $u$.
According to Theorem~\ref{THM::ERROR-ESTIMATES},
the asymptotic behaviour of~$e_{\Vht}$ depends on the regime we are considering.
More specifically, in the convection-dominated regime, it decays as~$\mathcal{O}(h^{k+1/2})$,
whereas, in a diffusion dominated regime, it decays as~$\mathcal{O}(h^k)$.

In Figure~\ref{fig:conv:1}, we show the errors obtained for~$k=1$ and~$2$ in
the diffusion-dominated regime, i.e., for~$\nu = 1$.
The schemes \stab{} and \nostab{} have the expected convergence rates.
Moreover, for the same mesh and approximation degree, the absolute values of the errors obtained with \stab{} and \nostab{} are very close to each other.
This is a numerical evidence that the stabilization terms added in the \stab{} scheme do not affect the convergence rates in the diffusion-dominated regime.
In Figure~\ref{fig:conv:1}(fourth panel), we observe a superconvergence trend of
the error~$e_{\Vht}$ for \texttt{cube} meshes and~$k = 1$;
this may be due to the shape regularity of the mesh and
the fact that we are evaluating an error based only on DoF values.

\begin{figure}[!htb]
\centering
\begin{tabular}{ccc}
\includegraphics[width=0.45\textwidth]{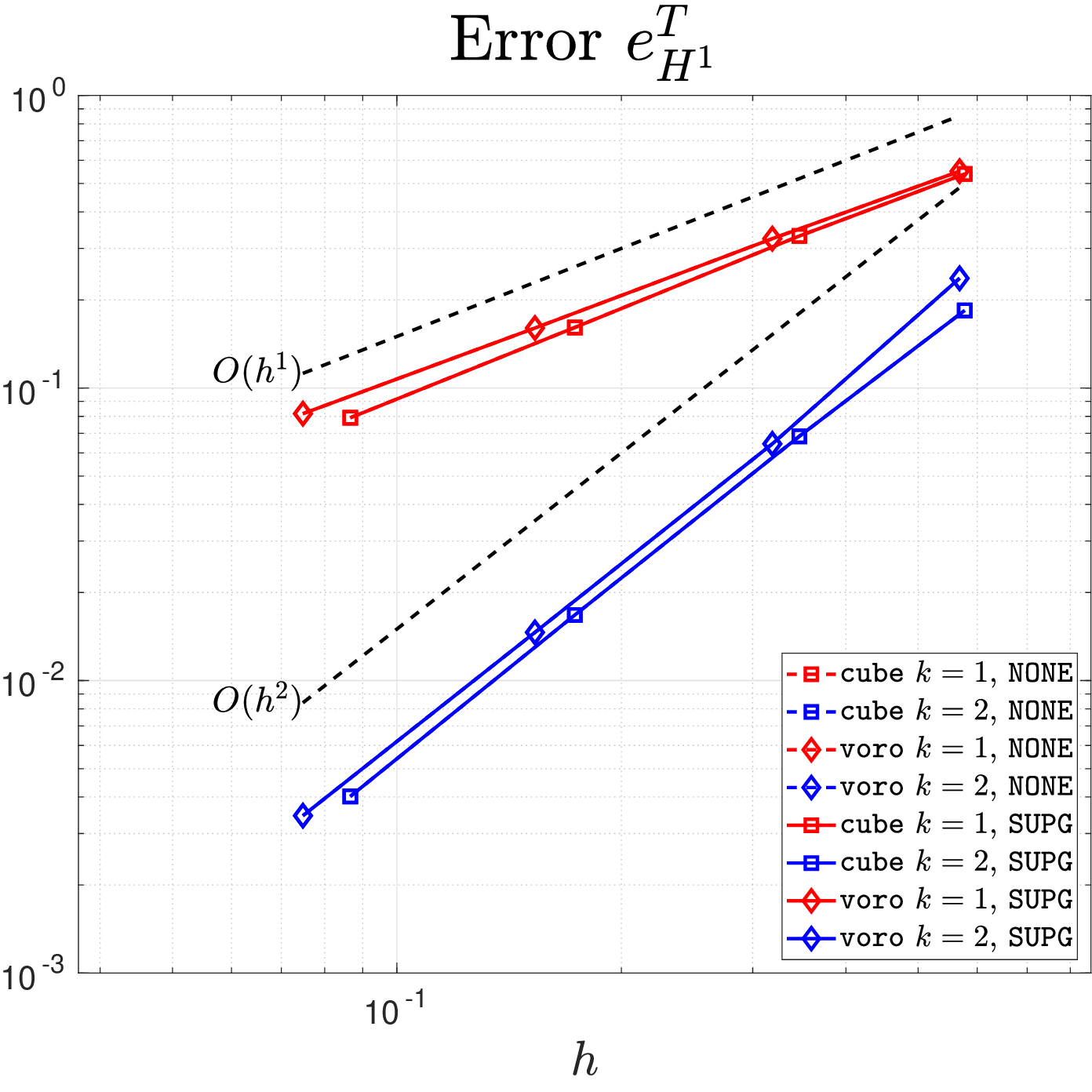} &\qquad&
\includegraphics[width=0.45\textwidth]{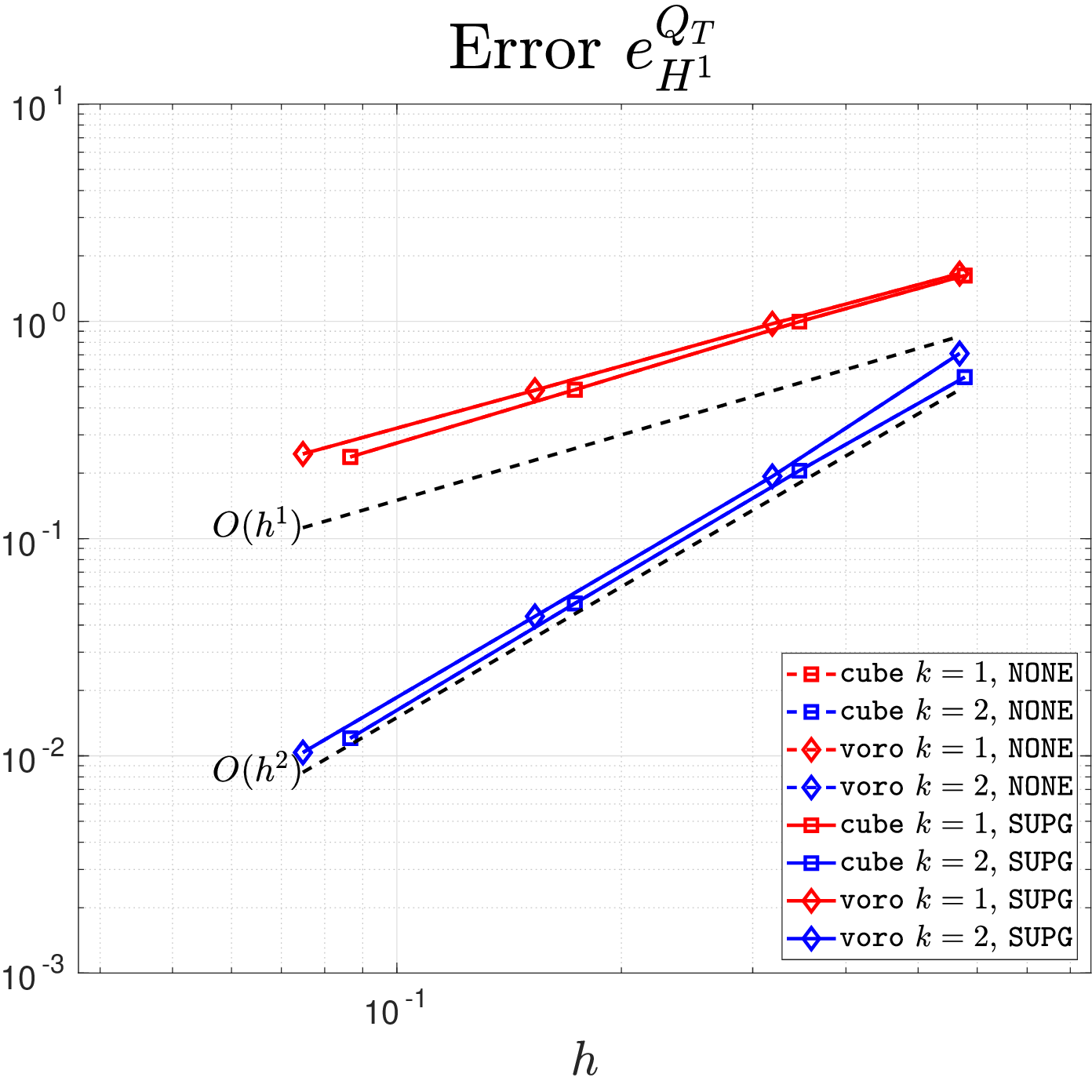} \\
\includegraphics[width=0.45\textwidth]{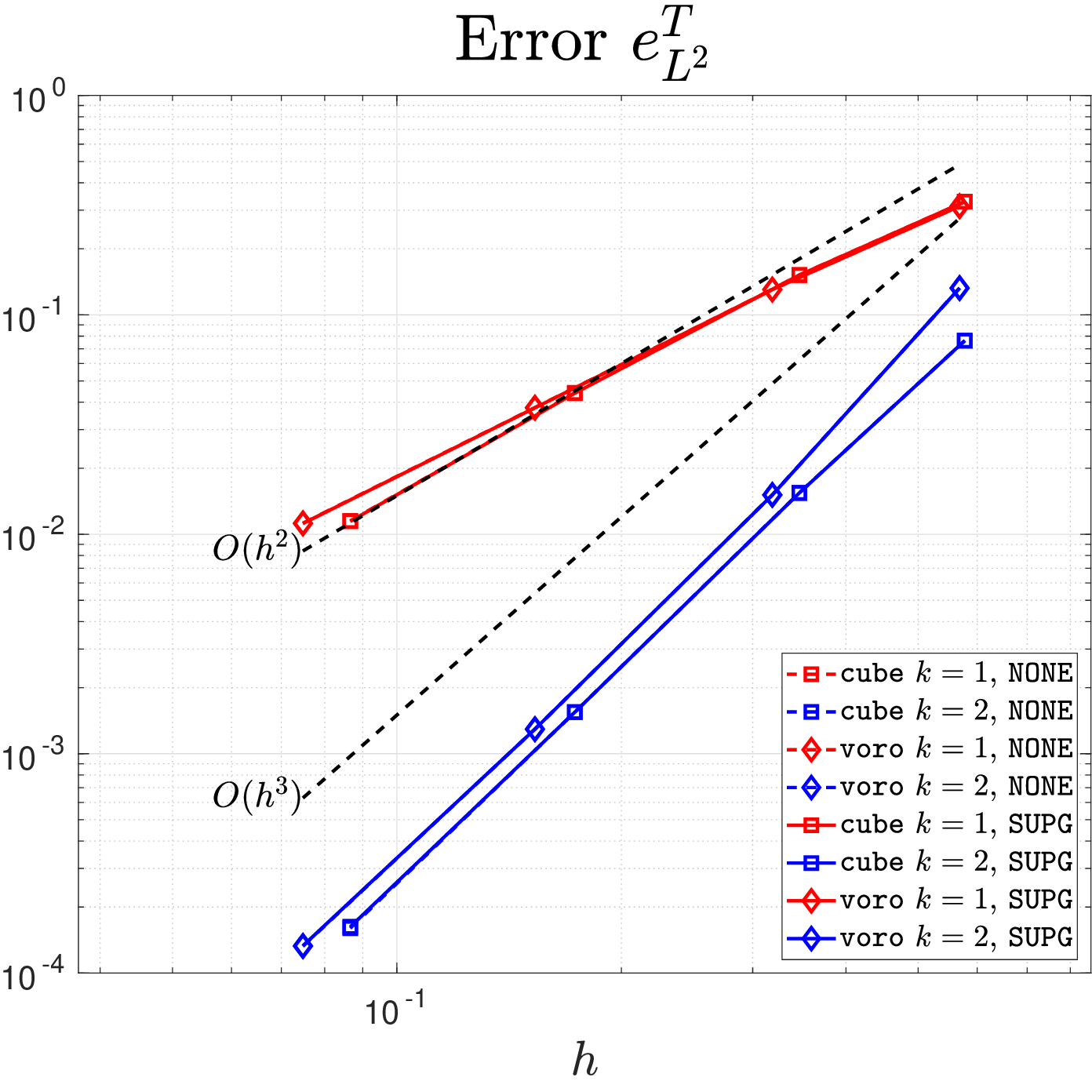} &\qquad&
\includegraphics[width=0.45\textwidth]{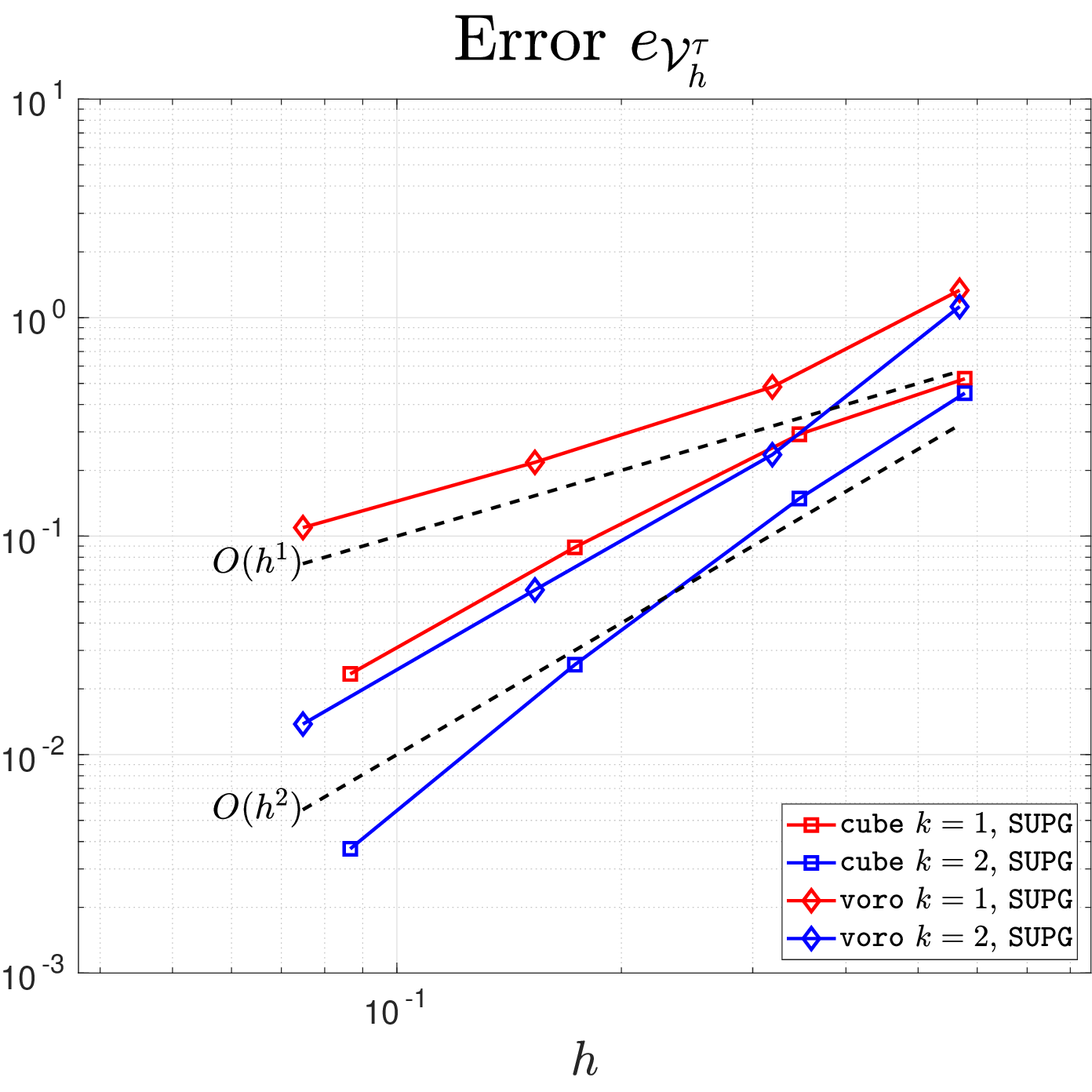} \\
\end{tabular}
\caption{Convergence Analysis: the trend of all errors
taken into account in the diffusion-dominated regime, i.e., for~$\nu=1$.}
\label{fig:conv:1}
\end{figure}

Now, we consider the convection-dominated regime.
In Figure~\ref{fig:conv:2}, we show the errors obtained for~$\nu = 10^{-10}$.
In this case, the results for the two approximation degrees considered are different.
Indeed, for~$k=1$, a similar behaviour is observed for the schemes \stab{} and \nostab{}.
More precisely, the convergence lines are close to each other for the errors~$e_{H^1}^T$ and $e_{H^1}^{Q_T}$,
whereas, for the error $e_{L^2}^T$, the error trend for the \nostab{} scheme is not optimal
in the last refinement step.
The advantage of using the \stab{} scheme becomes more evident for~$k=2$.
For \emph{all} errors and for \emph{both} type of meshes,
the convergence rates for the \nostab{} scheme degrade.
Moreover, if we compare the error $e_{\Vht}$ in Figure~\ref{fig:conv:1} and~\ref{fig:conv:2},
we obtain an additional $1/2$ in the convergence rates in the convection-dominated regime, which is in agreement with Theorem~\ref{THM::ERROR-ESTIMATES}.
This fact is more evident for the \texttt{voro} meshes,
which do not satisfy any shape regularity that may affect the trend of the error.

\begin{figure}[!htb]
\centering
\begin{tabular}{ccc}
\includegraphics[width=0.45\textwidth]{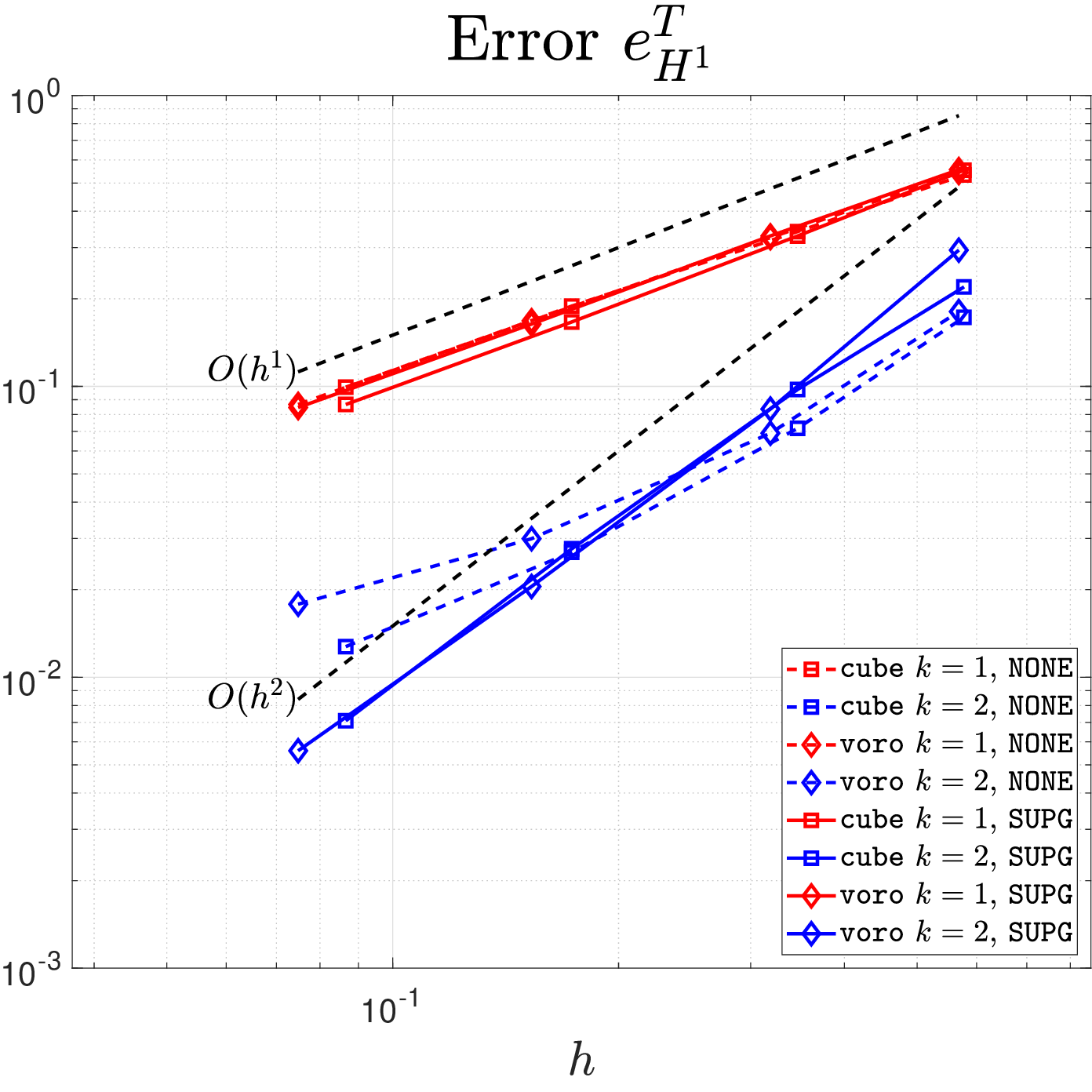} &\qquad&
\includegraphics[width=0.45\textwidth]{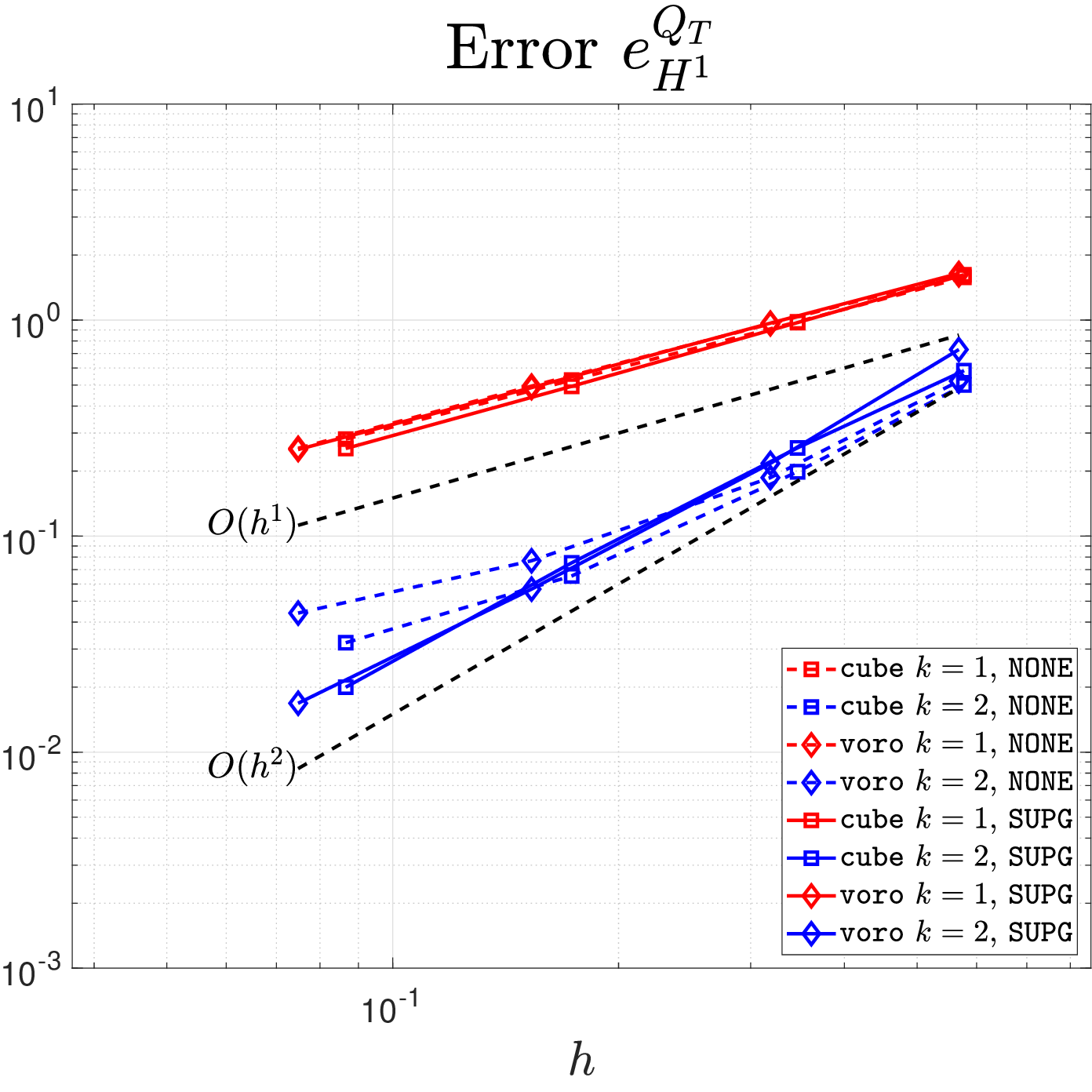} \\
\includegraphics[width=0.45\textwidth]{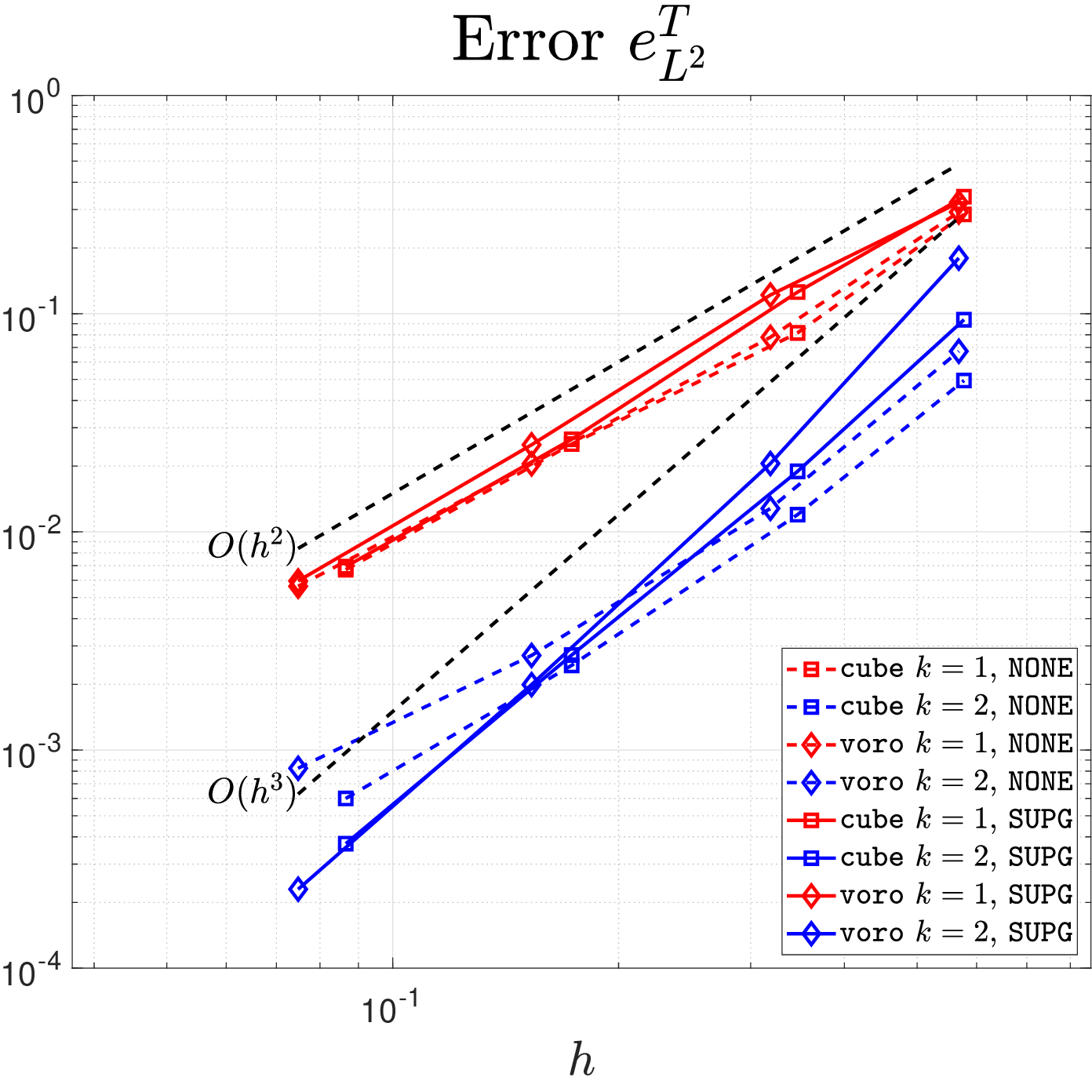} &\qquad&
\includegraphics[width=0.45\textwidth]{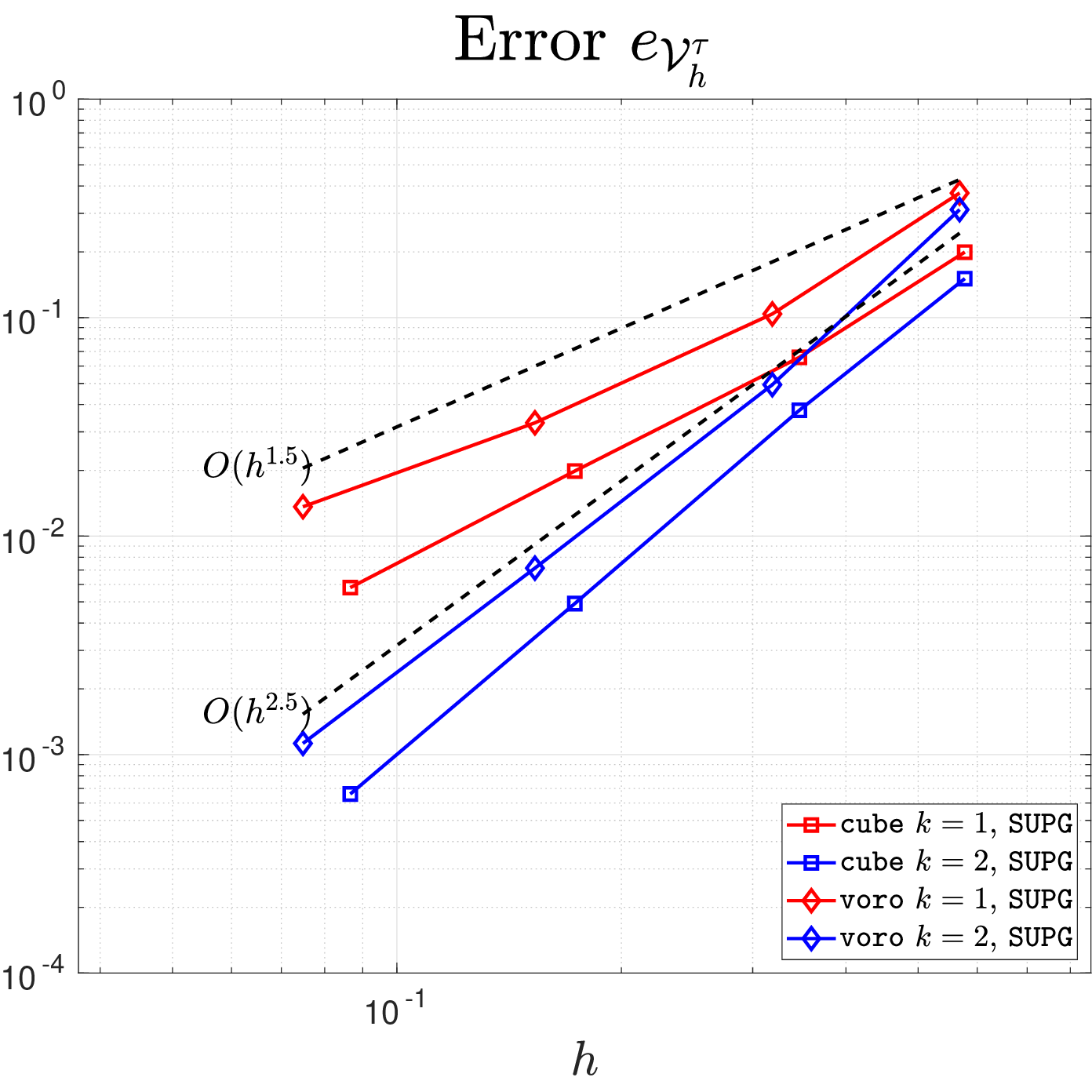} \\
\end{tabular}
\caption{Convergence Analysis: the trend of all errors
taken into account in the convection-dominated regime, i.e., for~$\nu=10^{-10}$.}
\label{fig:conv:2}
\end{figure}

\subsection{A qualitative purely advective test}\label{sec:num:bench}
In this section, we make a qualitative assessment of the proposed scheme.
To achieve this goal,
we produce a benchmark problem similar to the bi-dimensional Example~2 of~\cite[\S5.1]{Ahmed_Matthies_Tobiska_Xie:2011}, here developed in three space dimensions.
In that work, the authors considered three disjoint bodies
subject to a rotating advection field and
set the diffusive coefficient to~$10^{-20}$
to mimic a transport problem.

Let~$\Omega=(0,\, 1)^3$ and
the initial condition be given by
$$
u_0(x,\,y,\,z) := \left\{\begin{array}{rl}
1 &\text{if}\:\sqrt{(x-0.25)^2+(y-0.50)^2+(z-0.50)^2} \leq 0.2\\
0 &\text{otherwise}
\end{array}\right.
$$
that represents a ball of radius $r=0.2$ centred in $C(0.25,\,0.5,\,0.5)$.
We
set~$f = 0$, $\nu = 10^{-20}$, and the advection field
$$
\bbeta(x,\,y,\,z,\,t) := \left[\begin{array}{c}
     0.5 - y \\
     x - 0.5\\
    0.0
\end{array}\right]\,.
$$
For these data, the ball is expected to rotate around the barycenter
of the unit cube and,
since the diffusive coefficient is close to zero,
it has to preserve its shape.

We compute the discrete solution for~$k = 1$,
a fine tetrahedral spatial mesh,
a fixed time step~$\tau = 10^{-1}$, and a final time~$T = 6$.

In Figure~\ref{fig:exe2}, we show some clips of the discrete solution at different times
for both the \nostab{} and the \stab{} schemes, respectively.
Such clips are obtained using the \say{clip} filter of Paraview~\cite{paraview}
where we associate to each mesh vertex the value of the discrete function $\uh$.
As a consequence, we do not see exact circles in such plots,
as the shape is affected by the aforementioned geometric interpolation and visualization procedure.

At~$t=0$, we plot the initial condition for both cases.
For~\emph{all} the other time instances shown in Figure~\ref{fig:exe2},
several spurious values of~$\uh$ appear in the interior of the domain~$\Omega$ with the \nostab{} scheme, i.e.,
the discrete solution~$\uh$ is not constantly zero outside the moving sphere.
Such instabilities increase with time, as observed in the results obtained at the final time~$T = 6$.

The \stab{} scheme does not exhibit these instabilities,
as the solution outside the sphere is more uniform and,
according to the colorbar, it is closer to zero then the \nostab{} discrete solution.
Furthermore, despite the effect of the Paraview interpolation,
the sphere seems more uniform and rounded.

This benchmark test highlights the importance of using a stabilized scheme even for $k=1$,
a point that was not
evident from the error plots of Section~\ref{sec:num:conv},
c.f. the errors of \nostab{} and \stab{} for $k=1$ in Figures \ref{fig:conv:1} and~\ref{fig:conv:2}.

\begin{figure}[!htb]
\centering
\begin{tabular}{ccc}
t=0.0 &\nostab{} &\stab{} \\
t=1.5&\includegraphics[width=0.34\textwidth]{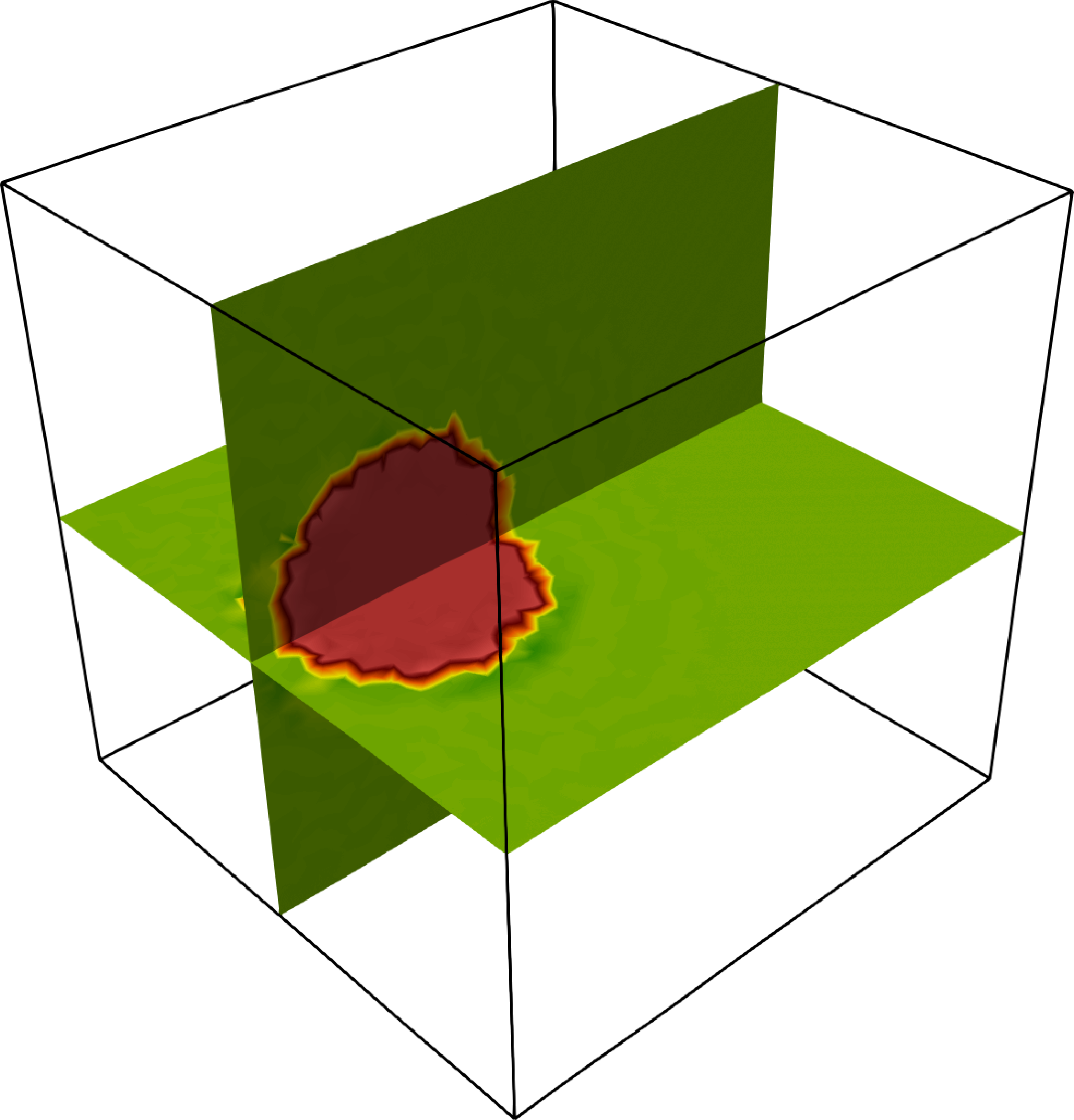}
&\includegraphics[width=0.34\textwidth]{imm/exe2New/exe2_SUPG_00.eps} \\
t=3.0&\includegraphics[width=0.34\textwidth]{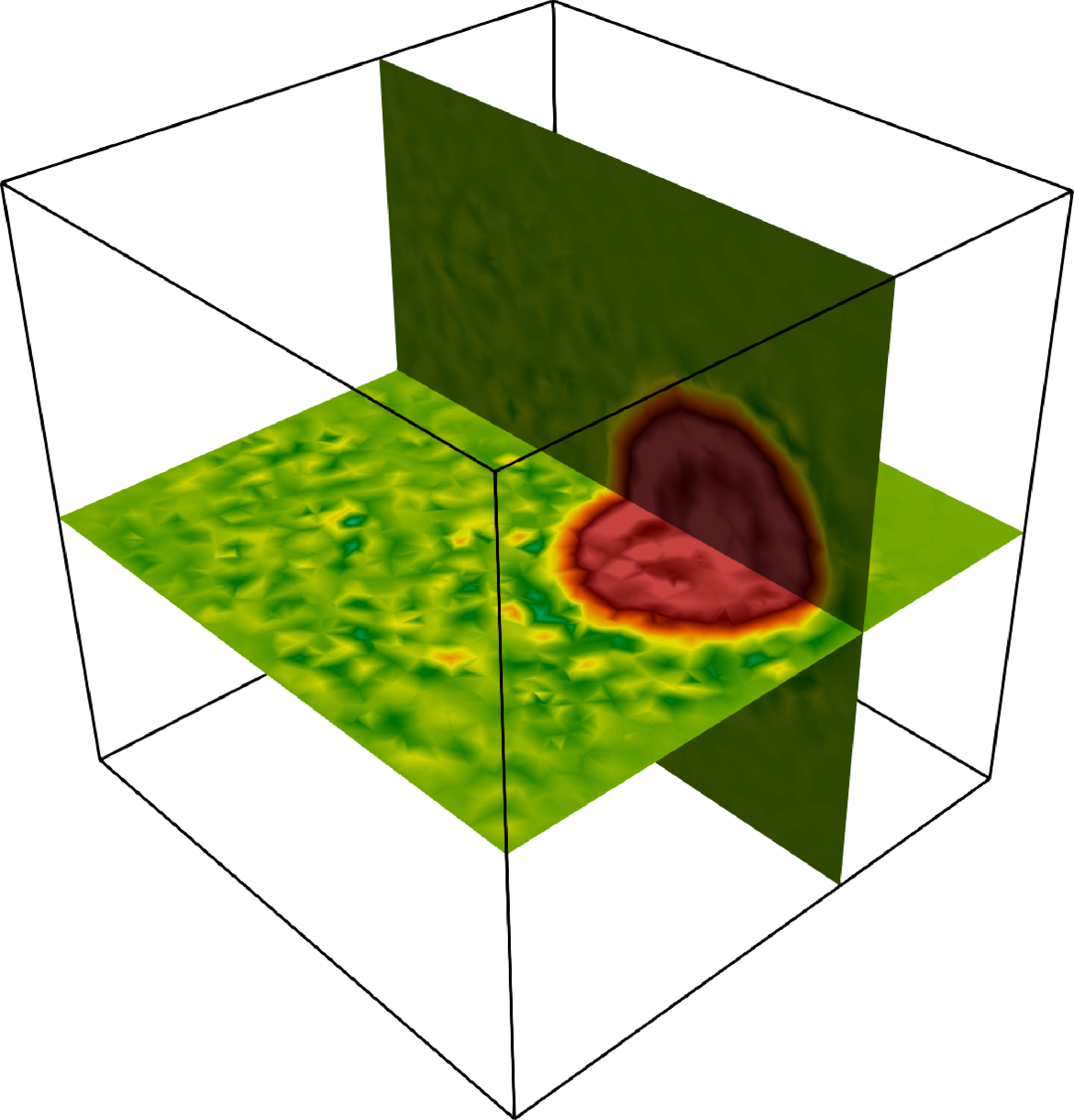}
&\includegraphics[width=0.34\textwidth]{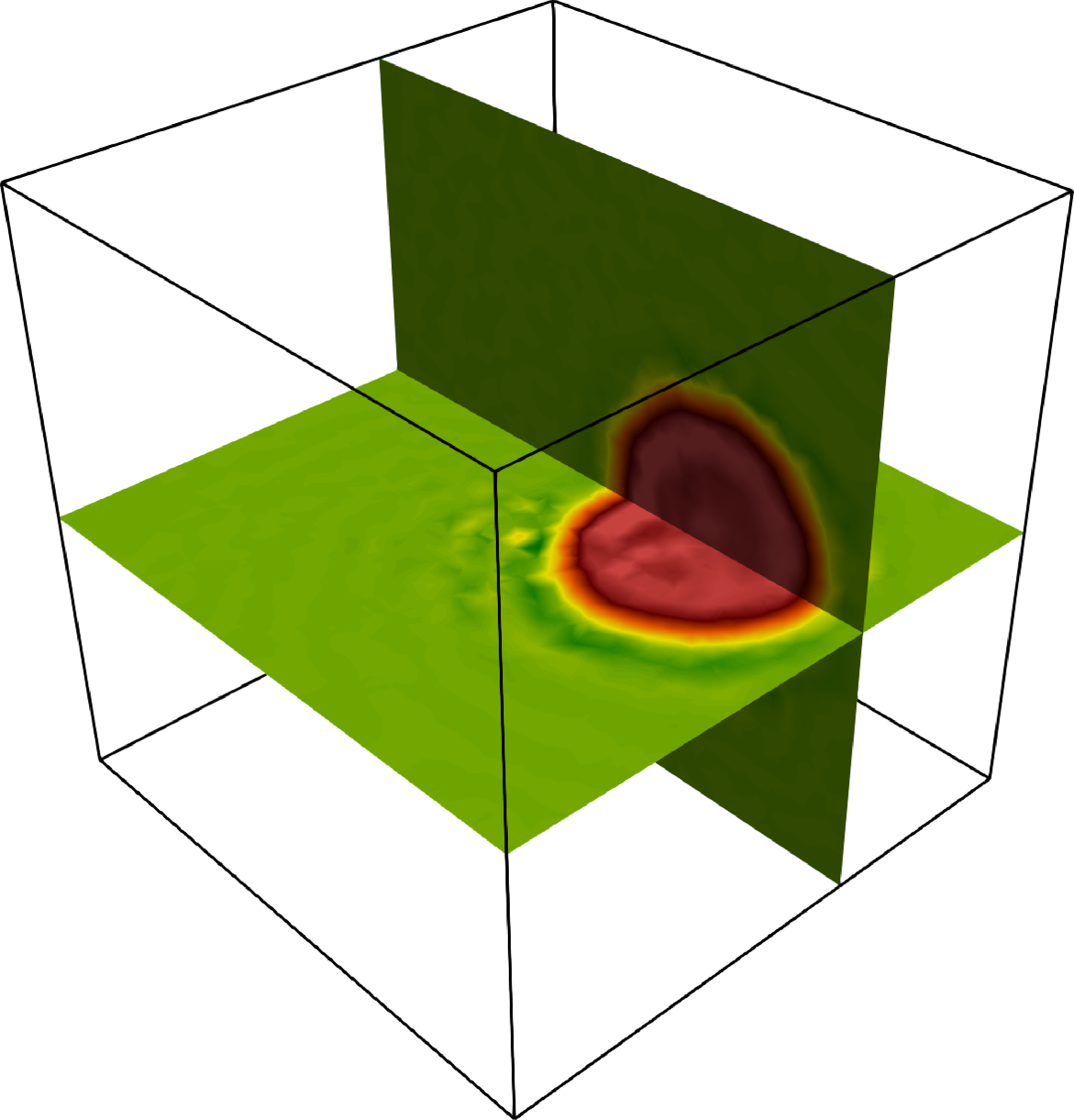} \\
t=6.0&\includegraphics[width=0.34\textwidth]{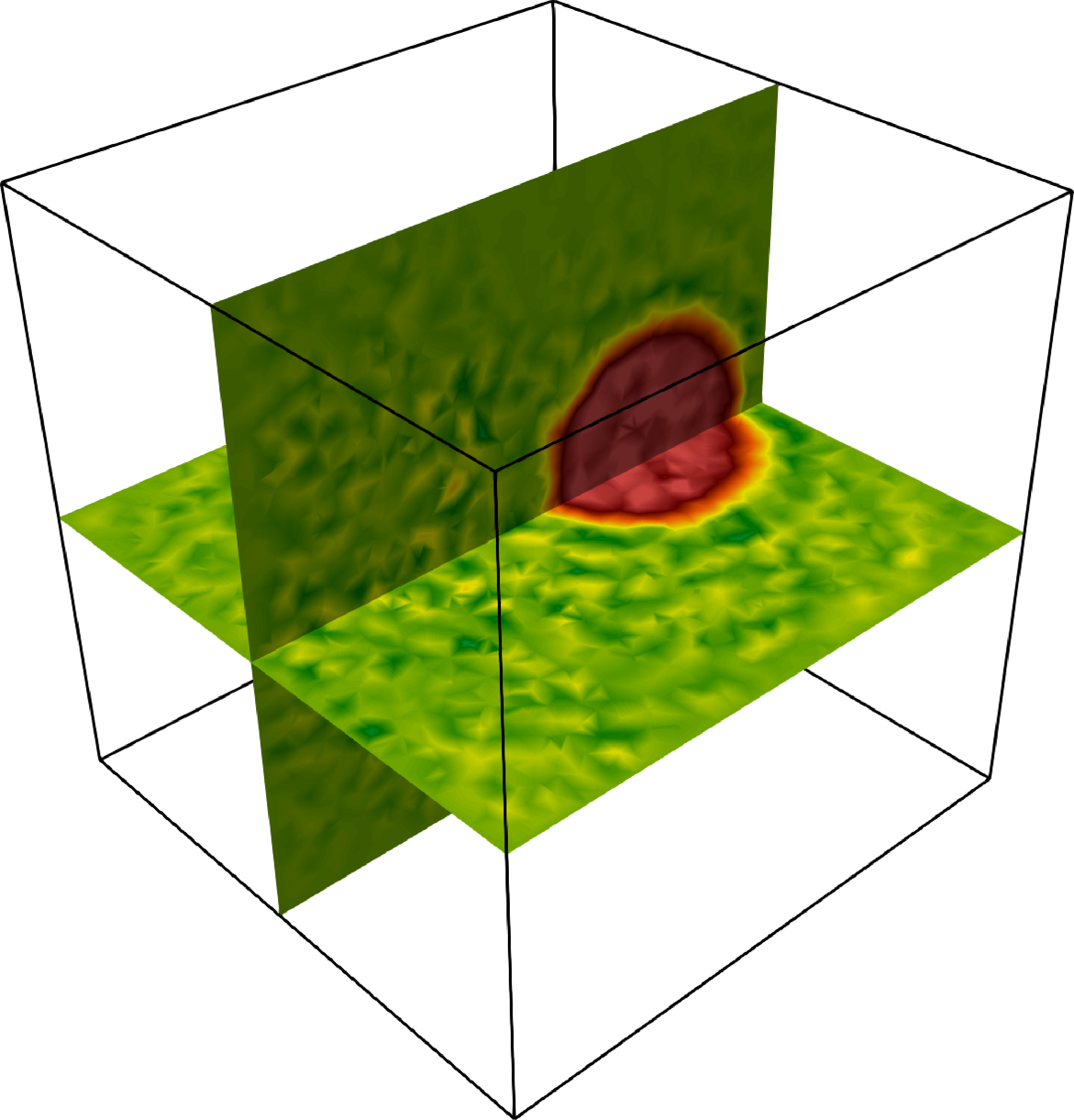}
&\includegraphics[width=0.34\textwidth]{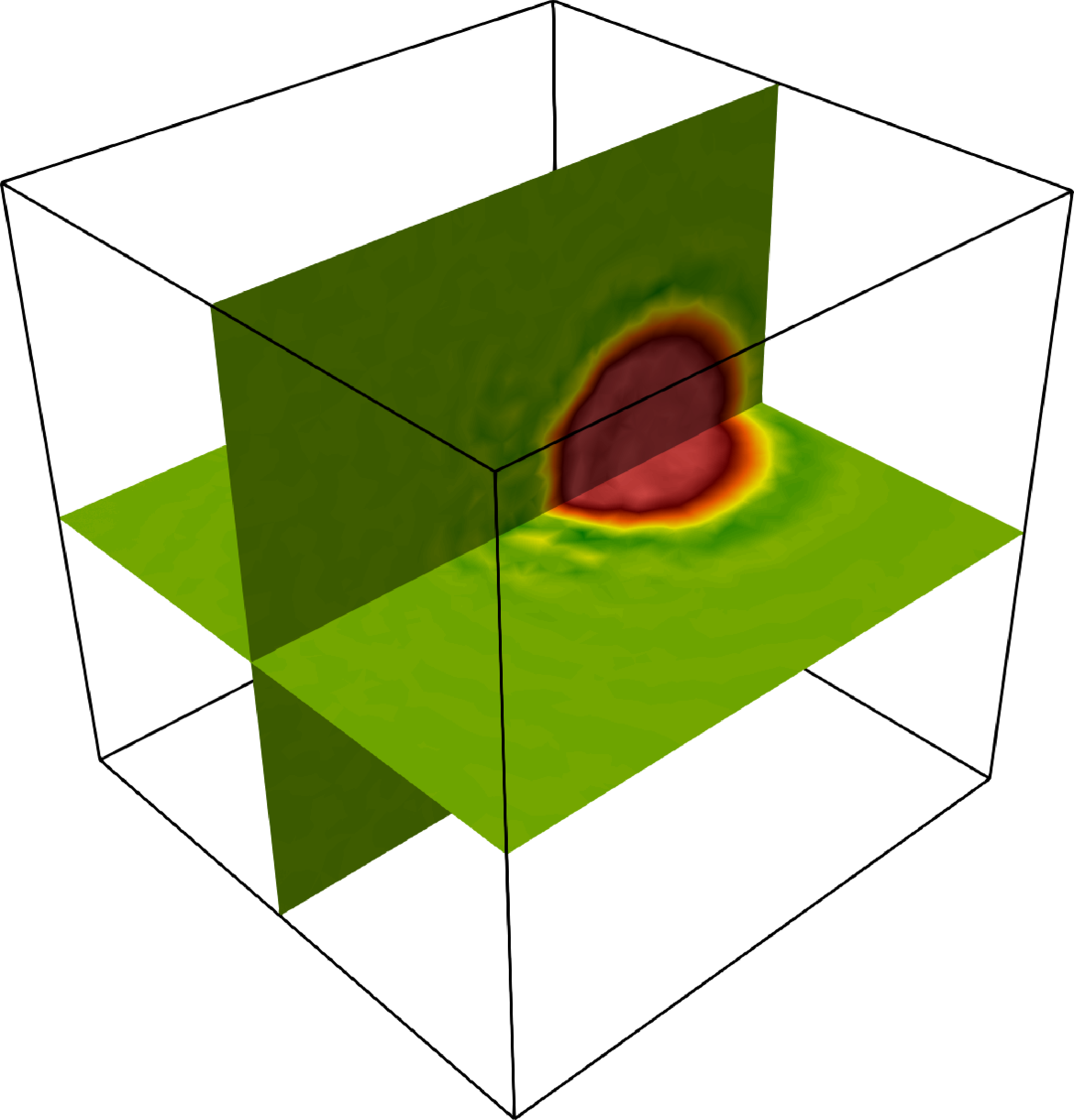} \\
&\includegraphics[width=0.34\textwidth]{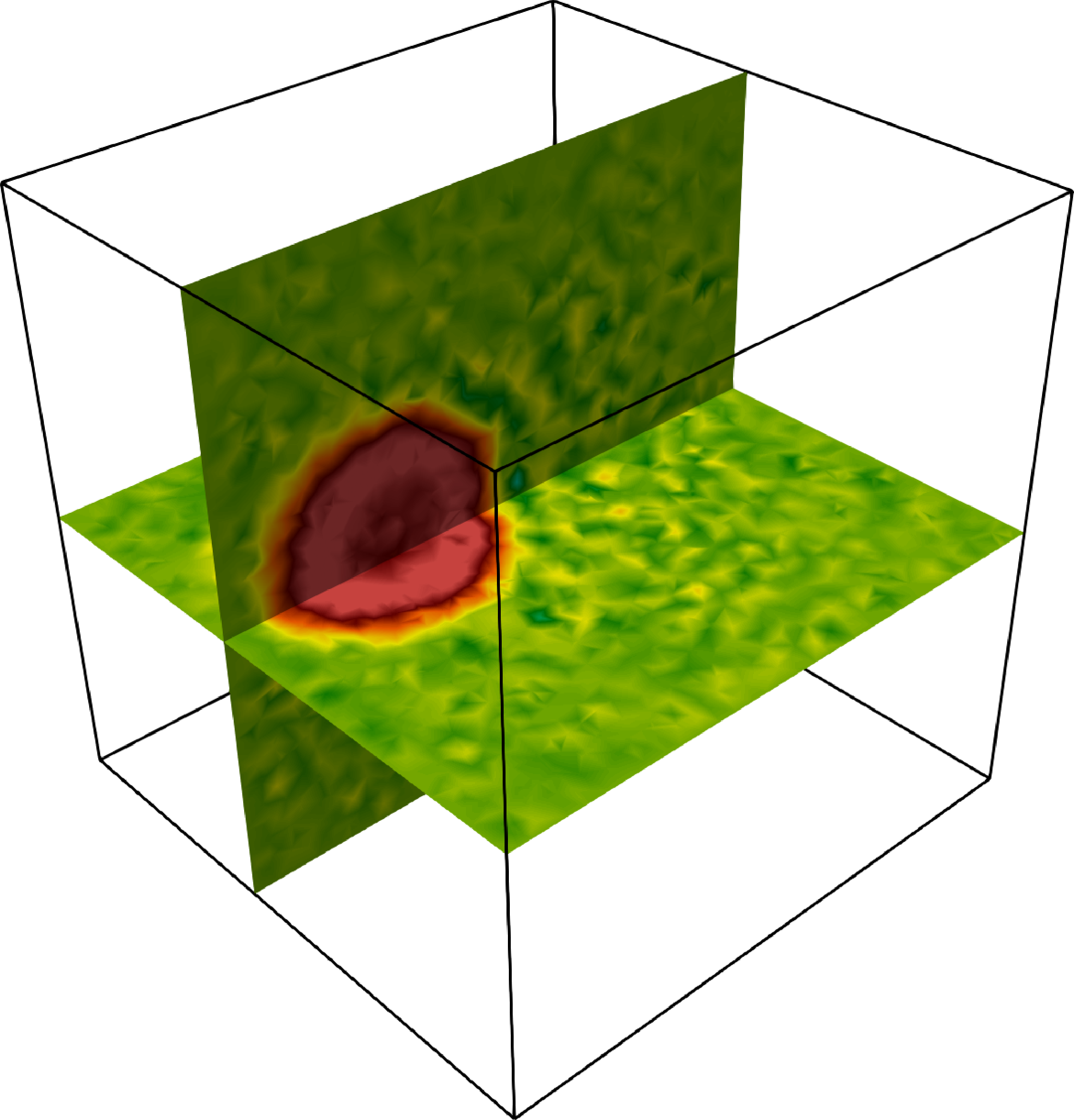}
&\includegraphics[width=0.34\textwidth]{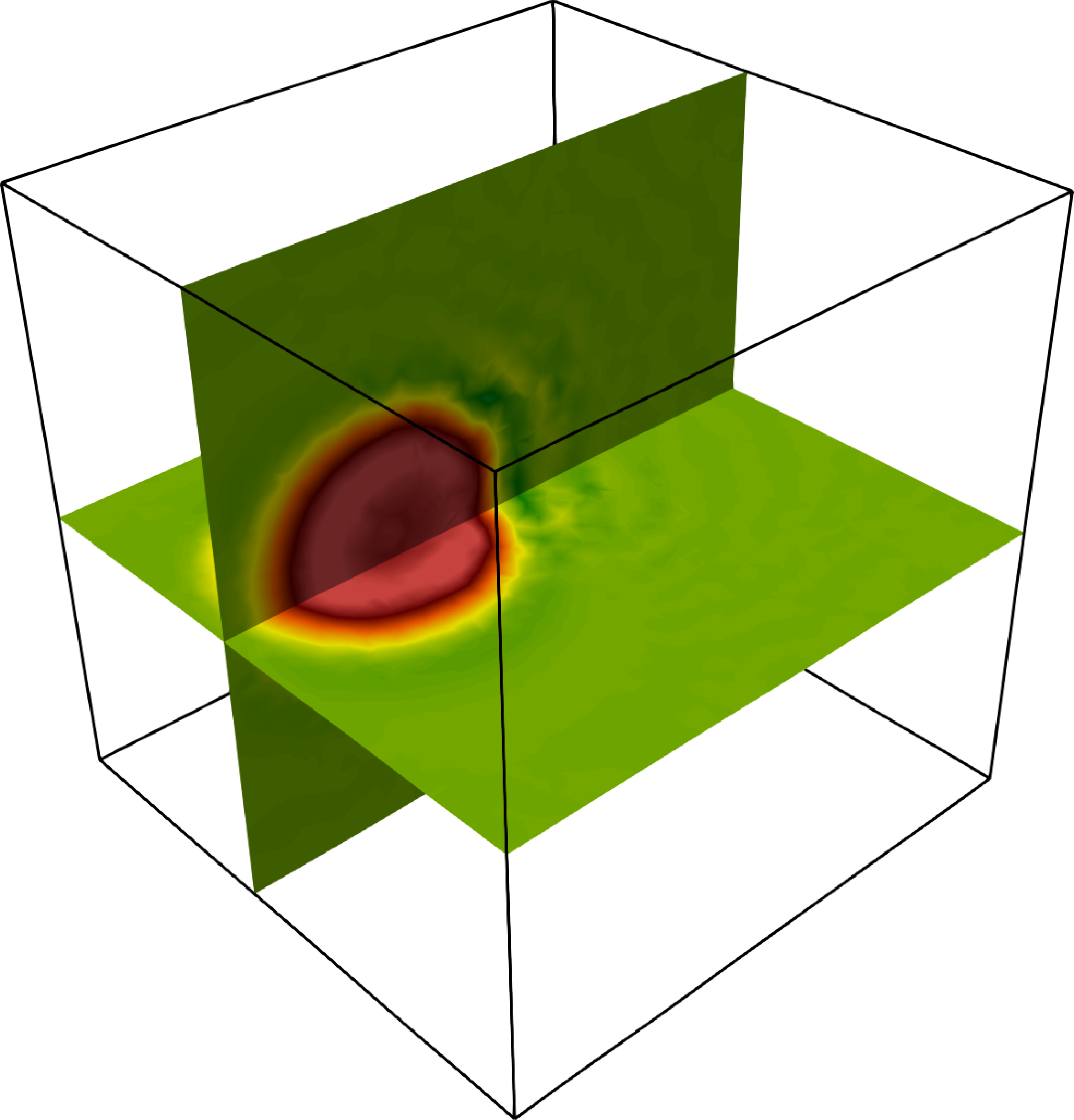} \\
&\multicolumn{2}{c}{\includegraphics[width=0.50\textwidth]{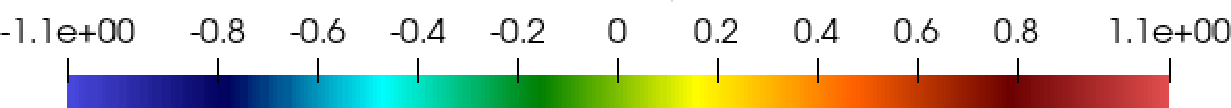}}\\
\end{tabular}
\caption{Benchmark problem: comparison between the discrete solutions obtained with the schemes \nostab{} and \stab{} at different time steps.}
\label{fig:exe2}
\end{figure}


\section{Conclusions\label{SECT::CONCLUSIONS}}
In this work, we considered a high-order SUPG-stabilized fully discrete scheme that combines finite or virtual element spatial discretizations with an upwind-DG time-stepping.
For this fully discrete scheme with finite element spatial discretizations, a robust analysis was missing in the literature. Moreover, this is the first work where a SUPG stabilization has been considered in a high-order-in-time fully discrete setting with virtual element spatial discretizations.

Using nonstandard test functions, we have shown that the method is inf-sup stable with respect to a norm involving an~$L^2(0, T; L^2(\Omega))$-term without requiring any transformation of the original problem.
Such a stability estimate is used to show that the method is robust and provides optimal convergence rates in the convection- and diffusion-dominated regimes.

We have presented some numerical experiments in~$(3 + 1)$-dimensions that show the robustness of the method, as well as the expected convergence rates of order~$\mathcal{O}(h^{k + \frac12})$ for the error in the energy norm, in the convection-dominated regime~$0 < \nu \ll 1$.

\section*{Acknowledgements}
 The first and second authors were partially funded by the European Union (ERC Synergy, NEMESIS, project number 101115663). Views and opinions expressed are however those of the authors only and do not necessarily reflect those of the European Union or the ERC Executive Agency.
The third author acknowledges support from the
Italian Ministry of University and Research through the project PRIN2020 ``Advanced polyhedral discretizations of heterogeneous PDEs for multiphysics problems", and from the INdAM-GNCS through the
project CUP E53C23001670001.

\paragraph{Declaration.} The authors declare no competing interests.


\end{document}